\documentclass[11pt]{article}
\usepackage{amsmath,amsthm,amsfonts,amssymb,amscd}
\usepackage{lastpage}
\usepackage{enumerate}
\usepackage{fancyhdr}
\usepackage{mathrsfs}
\usepackage{xcolor}
\usepackage{graphicx}
\usepackage{listings}
\usepackage{hyperref}
\usepackage{enumitem}
\usepackage{tikz}
\usepackage{tkz-fct} 
\usepackage{comment}

\usepackage[colorinlistoftodos]{todonotes}

\numberwithin{equation}{section}
\numberwithin{footnote}{section}

\usepackage{parskip}
\newtheorem{thm}{Theorem}[section]

\newtheorem*{introthm*}{Theorem}

\newtheorem{cor}[thm]{Corollary}

\newtheorem*{lem*}{Lemma}
\newtheorem{prop}[thm]{Proposition}

\newtheorem{defn}[thm]{Definition}

\newtheorem{war}[thm]{Warning}

\newtheorem{rem}[thm]{Remark}
\newtheorem*{notation*}{Notation}

\newcommand{\blue}{\color{blue}}
\newcommand{\red}{\color{red}}
\newcommand{\cyan}{\color{cyan}}

\newcommand{\ee}{{\mathrm e}}
\newcommand{\ii}{{\mathrm i}}
\newcommand{\Gr}{{\rm Gr}}

\renewcommand{\P}{{\mathbb P}}
\newcommand{\R}{{\mathbb R}}
\newcommand{\Q}{{\mathbb Q}}
\newcommand{\Z}{{\mathbb Z}}
\renewcommand{\H}{{\mathbb H}}
\newcommand{\C}{{\mathbb C}}

\newcommand{\A}{\mathbb A}
\newcommand{\CP}{{\mathbb C \mathbb P}}

\DeclareMathOperator{\rk}{rk}

\begin{document}
{\center{\Large{\bf Adventures in Algebraic Geometry}\\[.2em] \hskip 150pt  \normalsize{by Lev Borisov}}}

\section{Preface.}
\indent

\vskip -1em

It is a truth universally\footnote{ in grant applications}  acknowledged that algebraic geometry deals with solution sets of systems of polynomial equations. However, it has become significantly more abstract, and study often bifurcates between very basic examples and rigorous treatment in the language of schemes. I don't think either approach is ideal for an introductory course in the subject, and these notes are an attempt of providing some meaningful algebraic geometry content while avoiding or suppressing many of the technical details.
They are loosely based on the introductory course in algebraic geometry run at Rutgers University in the Spring of 2024. The intended audience are advanced undergraduate and beginner graduate students who are interested in algebraic geometry, as well as researchers of all levels in adjacent areas. When it comes to mathematical background, I am assuming basic knowledge of commutative algebra and complex analysis. Some familiarity with differential geometry and algebraic topology is also helpful. 

\smallskip
The defining feature of these notes is the liberal use of hand waving arguments where precise arguments would use too much technique, such as schemes and cohomology, or just where I get lazy. Thus, it is not meant to serve as a substitute of thorough study of standard texts such as Robin Hartshorne's book \cite{Hartshorne} or Ravi Vakil's notes \cite{Vakil}. Instead, it is a highly idiosyncratic attempt to highlight some of the beautiful examples of algebraic geometry and to introduce the reader to the vast expanse of the field.

\smallskip
Most of the time, we will be working over the field $\C$ of complex numbers, so that intuition and ideas from complex analysis and differential geometry can be applied. Occasional references are made to varieties over other fields of characteristics zero, but we really don't venture into positive characteristics. Each section roughly corresponds to one 80 minute lecture; two sections can be covered in a week. 
Exercises after each section are meant to be rather easy, with few exceptions.

\smallskip
These notes are intended to be a free resource and will be available on the arXiv  and on my home page. Comments and suggestions are appreciated.

\tableofcontents
 
 \newpage

\section{Algebraic subsets of $\mathbb C^n$ and Hilbert's Nullstellensatz.}
We start with some examples of systems of polynomial equations that have been studied by algebraic geometers.
\begin{align*}
\bullet~~&xy-1=0 \mathrm{~in~} \C^2.
\hskip 80pt~~\bullet~~x^2-y^2=1 \mathrm{~in~} \Q^2.
 \\
\bullet~~&\left\{
\begin{array}{r}
xt-yz=0\\
xz-y^2=0\\
yt-z^2=0
\end{array}\right.  \mathrm{~in~} \C^2.
\hskip 50pt~~\bullet~~y^2-y=x^3-x \mathrm{~in~} \Q^2,\R^2, \mathrm{~or~} \C^2.
 \\
\bullet~~&x^2=y^3 \mathrm{~in~} \C^2. 
\hskip 94pt~~\bullet~~\left\{
\begin{array}{l}
\det A =0\\
\mathrm{tr} \,A =0
\end{array}\right.  \mathrm{~in~} \C^{n^2}. 
\\
\bullet~~&\mathrm{all~}3\times 3 \mathrm{~minors~of~skew-symmetric~}
n\times n\mathrm{~matrices~are~zero,} \mathrm{~in~}\C^{\frac {n(n-1)}2}. 
\end{align*}

\noindent
Typical questions that one can ask about these equations and their solution sets include the following.

\begin{itemize}
\item Find all solutions. 

\item Given two solution sets, are they ``the same"? Similar?

\item If the solution set is a (real or complex) manifold, what is its topology?

\item What are the automorphisms of the solution set (in the  appropriate sense)?

\end{itemize}

\smallskip
{\blue \hrule}\footnote{ There are a lot of horizontal lines in these notes to remind the reader to take a sip of their drink of choice.}

\smallskip\noindent
We will now focus on $\C^n$ and the ring of polynomial functions $\C[x_1,\ldots,x_n]$ on it. 

\begin{defn}
For a subset $S$ of $\C[x_1,\ldots,x_n]$, we define its {\blue set of zeros} by 
$$
Z(S):=\{(a_1,\ldots,a_n)\in \C^n,~{\rm such~that~}\forall f\in S {\rm~there~holds~}f(a_1,\ldots,a_n)=0\}.
$$
\end{defn}

\noindent
We list some easy properties of $Z(S)$ below. The reader is encouraged to prove these assertions on their own.
\begin{itemize}
\item $Z(S)$ depends only on the ideal $I=(S)$. Thus, one can always replace $S$ by its finite subset without changing $Z(S)$, because 
$\C[x_1,\ldots,x_n]$ is a Noetherian ring. 

\item If $I \subseteq J$ then $Z(I)\supseteq Z(J)$.

\item $Z(I+J) = Z(I)\cap Z(J)$. More generally, $Z(\sum_\alpha I_\alpha) = \bigcap_\alpha Z(I_\alpha)$.

\item For ideals $I$ and $J$ there holds $Z(I\cap J) = Z(IJ) = Z(I)\cup Z(J)$.

\item  
We can moreover restrict to the case of {\blue radical} ideals. Recall that 
$$
\sqrt I = \{f\in \C[x_1,\ldots,x_n],~{\rm such~that~} f^k \in I{\rm ~for ~some ~}k\geq 1\}
$$
and an ideal $I$ is called radical if $\sqrt I =I$. There holds $Z(\sqrt I) = Z(I)$, and the ideal $\sqrt I$ is radical by Exercise 1.

\end{itemize}

\smallskip\noindent
{\bf Examples.} $\C^n=Z(\{0\})$;~~$\emptyset = Z(\C[x_1,\ldots,x_n]) = Z(\{1\})$;~~

$\{(a_1,\ldots,a_n)\} = Z((x_1-a_1,\ldots,x_n-a_n))$.

\smallskip
\smallskip
We can also go back, from subsets of $\C^n$ to ideals in $\C[x_1,\ldots,x_n]$.

\smallskip
\begin{defn}
For a subset $V\subseteq \C^n$ we define the ideal 
$$
I(V) := \{f\in \C[x_1,\ldots,x_n],~{\rm such~that~}f(p)=0{\rm~for~all~}p\in V\}.
$$
\end{defn}

\begin{rem}\label{maxpts}
For any 
point $p=(a_1,\ldots,a_n)\in \C^n$ the kernel of the evaluation map $f\in \C[x_1,\ldots,x_n]\mapsto f(p)\in \C$ is the maximal ideal
$(x-a_1,\ldots,x-a_n)$. Thus, $I(V)$ is the intersection of all maximal ideals that correspond to points $p\in V$.
\end{rem}

\smallskip\smallskip
\begin{defn}
We define {\blue Zariski topology} on $\C^n$ by declaring sets $Z(I)$  to be its closed sets.
\end{defn}

This is indeed a topology because finite unions and arbitrary intersections of $Z(I)$ are also sets of this form. Note that Zariski topology is very weak, in the sense that it does not have a lot of closed or open sets. For example, when $n=1$, proper Zariski closed subsets of $\C$ are precisely the finite subsets. In particular, it is (horribly) non-Hausdorff, i.e. one can not find disjoint open neighborhoods of distinct points. Still, Zariski topology provides an important conceptual viewpoint. 

\smallskip
{\blue \hrule}

\smallskip
Remark \ref{maxpts} in fact describes all maximal ideals of  $\C[x_1,\ldots,x_n]$.

\smallskip
\begin{thm}\label{whn}(weak Hilbert's Nullstellensatz)
All maximal ideals of the ring $\C[x_1,\ldots,x_n]$ are of the form $(x-a_1,\ldots,x-a_n)$.
\end{thm}

\smallskip
\begin{cor}\label{corwhn}
For an ideal $I$ of $\C[x_1,\ldots,x_n]$ we have 
$$
(Z(I)=\emptyset) \iff  (I = (1)).
$$
\end{cor}

\begin{proof}
It is a standard consequence of Zorn's lemma that an ideal $I$ is proper if and only if it is contained in some maximal ideal. This ideal corresponds to a point in $Z(I)$ by Theorem \ref{whn}.\end{proof}

We do not give a proof of Theorem \ref{whn} but refer the reader to \cite{DF} instead. The statement fails for fields which are not algebraically closed, already for $n=1$.
As a consequence of the weak Hilbert's Nullstellensatz, the set $Z(I)$ can be identified with the set of maximal ideals of $\C[x_1,\ldots,x_n]$ which contain $I$, which in turn are in natural bijection to the set of maximal ideals of the quotient ring $\C[x_1,\ldots,x_n]/I$.

\smallskip
\begin{thm}\label{hn}(regular strength Hilbert's Nullstellensatz)
For any ideal $I$ in $\C[x_1,\ldots,x_n]$ there holds
$$
I(Z(I))=\sqrt I.
$$
\end{thm}

\begin{proof}
If $f\in \sqrt I$, then $f^k\in I$ for some $k\geq 1$, therefore $f^k(p)=0$ for all $p\in Z(I)$. Thus $f(p)=0$ for all $p\in Z(I)$, so $f\in I(Z(I))$. This proves $\supseteq$ inclusion.

The other direction is harder. We will use what's known as the Rabinowitz trick, which is a particular case of localization. Consider $0\neq f\in I(Z(I))$. In $\C[x_1,\ldots,x_n,t]$ the ideal 
$$\hat I = (ft-1)+(I)
$$
has no zeros. Here $(I)$ denotes the ideal in the ring $ \C[x_1,\ldots,x_n,t]$ generated by $I$. Indeed, any point $(a_1,\ldots,a_n,b)$ in the set of zeros of  $\hat I$ must have $(a_1,\ldots,a_n)$ in the set of zeros of $I$. Then $ft-1$ takes value $0\,b-1\neq 0$ on $(a_1,\ldots,a_n,b)$, contradiction.

By Corollary \ref{corwhn}, the ideal $\hat I$ is the whole ring, and in particular it contains $1$.\footnote{ It is a common trick to replace the statement that an ideal is equal to a ring by the statement that this ideal contains $1$. If you study commutative algebra long enough, you will get sick of it.}
In other words, we have in $\C[x_1,\ldots,x_n,t]$
$$
1 = (f(x_1,\ldots,x_n)t-1)g_0(x_1,\ldots,x_n,t) + \sum_i f_i(x_1,\ldots,x_n) g_i(x_1,\ldots,x_n,t)
$$
where $f_i\in I$. 
Consider the image of this equality in the field of rational functions in $n$ variables $\C(x_1,\ldots,x_n)$ under the algebra map 
$$\C[x_1,\ldots,x_n,t]\to \C(x_1,\ldots,x_n)$$which sends $x_i$ to $x_i$ and $t$ to 
$f(x_1,\ldots,x_n)^{-1}$. Then the above equation becomes
$$
1 = 0 + \sum_i f_i(x_1,\ldots,x_n) \frac { h_i(x_1,\ldots,x_n) }{ f(x_1,\ldots,x_n)^{k_i}}
$$
for some polynomial $h_i$ and some $k_i\geq 0$.
After clearing denominators, we conclude that $f^k$ lies in $I$, which proves the $\subseteq$ inclusion.
\end{proof}

As a corollary of Theorem \ref{hn}, we see that Zariski closed subsets $Z(I)$ of $\C^n$ are in bijection with the radical ideals of $\C[x_1,\ldots,x_n]$ (including the whole ring).

\smallskip

{\bf Exercise 1.} Prove that $\sqrt{\sqrt I} = \sqrt{I} $.

{\bf Exercise 2.} Prove that $Z(I(V)) = \overline V$ (the closure in Zariski topology).

{\bf Exercise 3.} Find counterexamples to Theorems \ref{whn}, \ref{hn} and Corollary \ref{corwhn} for $\C$ replaced by $\R$.

\section{Irreducible components. Dimension theory. Projective spaces. Algebraic varieties.}

Let $I$ be a proper ideal in the polynomial ring  $\C[x_1,\ldots,x_n]$ and let $Z(I)$ be the corresponding nonempty Zariski closed subset in $\C^n$. 

\smallskip
\begin{defn}
We call $Z(I)$ {\blue reducible} if it can be written as a union of two proper closed subsets $Z(I)= Z(I_1)\cup Z(I_2)$ with $Z(I_i)\neq Z(I)$. It is called {\blue irreducible} otherwise.
\end{defn}

For example, for $n=2$ and $I=(x_1x_2)$ we have 
$$Z(I) =\{(a_1,a_2), ~a_1a_2=0\} = Z(I_1) \cup Z(I_2)$$
where $I_1=(x_1)$ and $I_2=(x_2)$.

\smallskip
We remark that 
while disconnected algebraic sets are always reducible, the converse is not true, as the below picture indicates.

\begin{tikzpicture}
\draw[blue] plot[domain=-1.2:{1.2},samples=100,variable=\x] 
(2.3-\x*\x,.8+\x- \x*\x*\x);
\draw (1.5,0) node[anchor=north]{irreducible};
\draw[blue] plot[domain=-1.2:{1.2},samples=100,variable=\x] 
(6.3-\x*\x,.8+\x- \x*\x*\x);
\draw[red](4.5,1)--(6.5,1);
\draw (5.5,0) node[anchor=north]{reducible, connected};
\draw[blue] plot[domain=-1.2:{1.2},samples=100,variable=\x] 
(10.3-\x*\x,.8+\x- \x*\x*\x);
\draw  (8.7,1) node [anchor=north]{\circle* 2};
\draw[red](8.5,1)--(10.5,1);
\draw (9.5,0) node[anchor=north]{reducible, disconnected};
\end{tikzpicture}

\smallskip
\begin{thm}
Every closed subset $Z(I)$ can be uniquely written as a union of irreducible closed subsets of $Z(I)$ none of which is contained in another one.
\end{thm}

\begin{proof}
In the existence direction, we  start with $Z(I)$ and keep decomposing it into unions of smaller subsets.
If we could keep doing this indefinitely, we would get a decreasing chain of closed subsets which are properly contained in one another. This would correspond to an increasing chain of ideals of $\C[x_1,\ldots,x_n]$, in contradiction with it being Noetherian. Thus, we can write $Z(I)$ as a finite union of irreducible subsets, and it remains to remove the subsets that are contained in others.

In the uniqueness direction, suppose we have two finite decompositions
$$Z(I)= \bigcup_\alpha Z(I_\alpha) = \bigcup_\beta Z(I_\beta).$$
For each $\alpha$ we have 
$$
Z(I_\alpha) = \bigcup_\beta (Z(I_\beta)\cap Z(I_\alpha)).
$$
By irreducibility of $Z(I_\alpha)$, one of  $Z(I_\beta)\cap Z(I_\alpha)$ must equal $Z(I_\alpha)$. This implies that for every $\alpha$ there exists $\beta$ such that $Z(I_\alpha)\subseteq Z(I_\beta)$. Of course, the same is true for each $\beta$, so we get $Z(I_\alpha)\subseteq Z(I_\beta)\subseteq Z(I_{\alpha'})$. Then $Z(I_\alpha)\subseteq Z(I_{\alpha'})$ implies $\alpha = \alpha'$ and we get $Z(I_\alpha) = Z(I_\beta)$. Since this works for each $\alpha$ (and for each $\beta$ in the other direction), we have proved uniqueness.
\end{proof}

\smallskip
\begin{rem}
The decomposition of $Z(I)$ into irreducible components is a particular example of the commutative algebra statement that for any ideal $I$ in a Noetherian ring, the radical $\sqrt I$ can be uniquely written as 
$$
\sqrt I = \bigcap_i p_i
$$
where $p_i$ are a finite set of minimal primes that contain $I$. See \cite{AM} for details, or just trust me.\footnote{ I am building up your tolerance for accepting random technical claims.}
\end{rem}

\smallskip
{\blue \hrule}

\smallskip
We will now discuss the concept of dimension of algebraic subsets $Z(I)$ of $\C^n$. If $Z(I)$ is irreducible, we will also call it an {\blue algebraic subvariety} of $\C^n$. We have a concept of dimension for smooth complex manifolds, and algebraic subvarieties are not that far from them, so it is natural to expect some kind of definition here. The technical arguments in commutative algebra that one needs in order to prove the following claims are rather formidable, but we will just take it for granted that dimension works as one expects. 

\smallskip
\begin{defn} For an irreducible $Z(I)$ the ideal $I$ is prime (see Exercise 1) so one can look at the quotient ring $A=\C[x_1,\ldots,x_n]/I$. Its quotient field $QF(A)$ is a finitely generated field extension of $\C$ and we define 
$$
\dim_\C Z(I) := {\rm tr.deg.}_\C \,QF(A)
$$
where ${\rm tr.deg.}$ stands for the transcendence degree of the extension.
\end{defn}

\smallskip
\begin{defn}\label{rationalfunctions}
The quotient field $QF(A)$ is called the field of {\blue rational functions on $X=Z(I)$}.
\end{defn}

Clearly, dimension of $\C^n$ is the transcendence degree of the field of rational functions $\C(x_1,\ldots,x_n)$ which is $n$. Furthermore, every chain
of irreducible subvarieties can be enhanced to a chain that has exactly one member of each dimension.
$$
{\rm point} \subsetneq {\rm curve} \subsetneq {\rm surface } \subsetneq \cdots  \subsetneq \C^n
$$
The corresponding property of the chains of prime ideals in $\C[x_1,\ldots, x_n]$ goes under the name of {\blue catenary}.

\smallskip
\begin{war}
Algebraic geometers have a very annoying habit of drawing complex varieties of dimension $n$ as objects of {\blue real} dimension $n$, with the notable exception of Riemann surfaces (complex dimension one) which are sometimes drawn as real surfaces. Pictures below are illustration of that.
Of course, once the dimension is three or higher, the paper itself is a limitation.
\end{war}

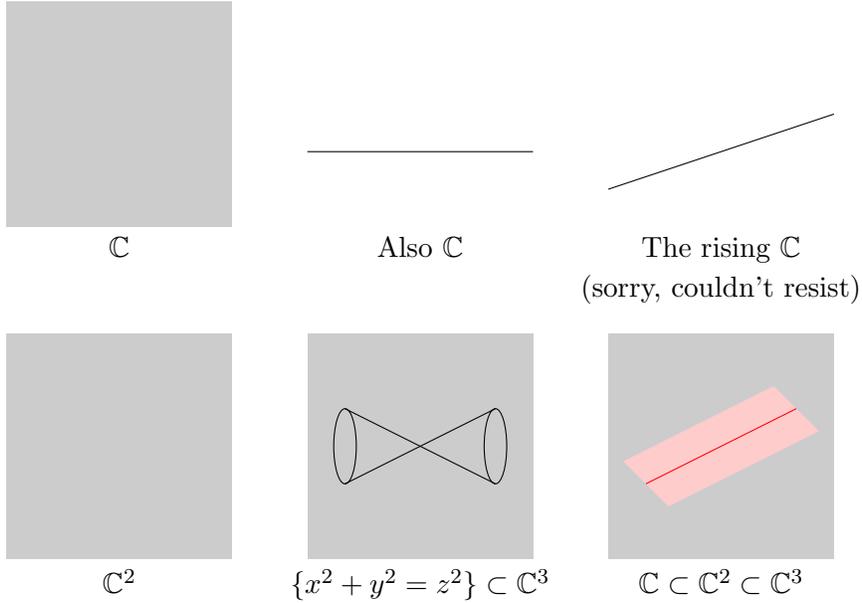
\begin{tikzpicture}
\fill[gray!40!white](0,0) rectangle  (3,3);
\draw (1.5,0) node[anchor=north]{$\C$};
\draw (4,1) -- (7,1);
\draw(5.5,0) node[anchor=north]{Also $\C$};
\draw (8,0.5) -- (11,1.5);
\draw(9.5,0) node[anchor=north]{The rising $\C$};
\draw(9.5,-0.5) node[anchor=north]{(sorry, couldn't resist)};
\end{tikzpicture}

\begin{tikzpicture}
\fill[gray!40!white](0,0) rectangle  (3,3);
\draw (1.5,0) node[anchor=north]{$\C^2$};
\fill[gray!40!white](4,0) rectangle  (7,3);
\draw (4.5,2) -- (6.5,1);
\draw (4.5,1) -- (6.5,2);
\draw (4.5,1.5) ellipse (.15 and .5);
\draw (6.5,1.5) ellipse (.15 and .5);
\draw(5.5,0) node[anchor=north]{$\{x^2 +y^2 = z^2\}\subset \C^3$};
\fill[gray!40!white](8,0) rectangle  (11,3);
\fill[red!20!white] (8.2,1.3) -- (10.2,2.3)--(10.8,1.7)-- (8.8,0.7)--cycle;
\draw[red] (8.5,1) -- (10.5,2);
\draw(9.5,0) node[anchor=north]{$\C\subset \C^2 \subset \C^3$};
\end{tikzpicture}

\noindent
Unfortunately, there is not much we can do about it. It is still better than trying to draw complex surfaces as real $4$-dimensional objects that they are.

\smallskip{\blue \hrule}
We now turn our attention to projective spaces and projective varieties.

\smallskip
\begin{defn}
A complex projective space $\C\P^n$ is defined as the set 
$$
(\C^{n+1}\setminus\{{\mathbf 0}\})/\sim
$$
of nonzero $(n+1)$-tuples of complex numbers under the equivalence
$$
(x_0,\ldots,x_n)\sim (\lambda x_0,\ldots,\lambda x_n), ~\lambda \in \C^*.
$$
In other words, $\C\P^n$ is the set of orbits of the scaling  action of the multiplicative group $\C^*$ on $\C^{n+1}\setminus\{{\mathbf 0}\}$.
Points on $\C\P^n$ are denoted by $(x_0:\ldots:x_n)$ with the colons used to distinguish them
from points in $\C^{n+1}$.
\end{defn}

As stated, $\C\P^n$ only has a structure of a set, but it is easy to see that it is in fact a complex manifold.
Specifically, for each $i\in \{0,\ldots,n\}$ we consider a subset $U_i$ of $\C\P^n$ with $x_i\neq 0$. Points of 
$U_i$ can be uniquely scaled to 
$$
\left(\frac {x_0}{x_i}:\ldots:1:\ldots: \frac{x_n}{x_i}\right)
$$
so $U_i$ can be identified with $\C^n$. One can then use this covering to get $\C\P^n$ the structure of a complex manifold, since transition functions can be easily seen to be holomorphic.

\smallskip
The transition functions are in fact rational, and this gives $\C\P^n$ a structure of an algebraic variety. {\red Wait, what?}

\smallskip
{\blue We have not yet defined complex algebraic varieties, what are they?} I will avoid putting down a formal definition, but one should think of them as sets glued from a finite number of charts $U_i$, with each of these charts identified 
with an algebraic subvariety $Z(I_i)$ in $\C^{k_i}$ for some $k_i$. The intersections $U_i\cap  U_j$ should be  
Zariski open subsets in the induced topologies on $U_i$ and $U_j$ and should themselves be identified with some 
irreducible subsets $Z(I_{ij})$ in $\C^{k_{ij}}$. The transition functions should be induced from ratios of polynomials on $\C^{k_{ij}}$,
with denominators that are nonzero on $U_i\cap U_j$.\footnote{ A remark for the specialists is that we ignore the issues of separatedness.} Complex algebraic varieties come with two topologies, the usual (a.k.a. strong) one and the Zariski one, both induced from the corresponding topologies on ambient $\C^{k_i}$ on the charts $U_i$.

\smallskip
As one expects, morphisms of algebraic varieties should be Zariski locally given by some ratios of polynomials. Overall, one should think of complex algebraic varieties as being very similar to complex manifolds, but with some singularities allowed. 
However, the morphisms are more restricted than those of complex manifolds. For example, the map $\C\to \C^2$ given by
$$
f(x) = (x,\ee^x)
$$
is holomorphic but not algebraic and its image is not an algebraic subvariety of $\C^2$. {\red The reader is warned that some smooth complex manifolds can not be given an algebraic structure. In the other extreme, there exist nonisomorphic smooth complex algebraic varieties which are biholomorphic to each other, see \cite{Hartshorne}.}

\smallskip
\begin{rem}
One can show that two different closed subvarieties $Z(I)\subseteq \C^n$ and $Z(J)\subseteq \C^m$ are isomorphic 
as algebraic varieties if and only if the corresponding finitely generated algebras $\C[x_1,\ldots,x_n]/I$ and $\C[x_1,\ldots, x_m]/J$ are isomorphic. It is a natural statement, since points are in bijections with maximal ideals in these rings.  
We think of them as just the same variety in two different realizations. Such varieties are called {\blue affine}.
\end{rem}

\smallskip
\begin{rem}
It is possible to show that for an algebraic variety $X$ covered by $U_i\cong Z(I_i)$, the quotients fields
of $\C[x_1,\ldots,x_{k_i}]/I$ are the same for each $i$. This associates to every algebraic variety a field of rational functions on it. You should think of this as an algebraic analog of the field of meromorphic functions on a complex manifold.
\end{rem}

\smallskip
{\blue \hrule}
The main advantage of $\C\P^n$ as compared to $\C^n$ is the following.

\smallskip
\begin{thm}
$\C\P^n$ is a compact complex manifold.
\end{thm}

\begin{proof}
We can take the quotient by $\C^*$ in two steps, first by $\R_{>0}$ and then by $S^1=\{z\in \C^*,|z|=1\}$. We immediately see that 
(as real manifolds)
$$
(\C^{n+1}\setminus\{{\mathbf 0}\})/\R_{>0}\cong \{(x_0,\ldots,x_n)\in \C^{n+1},~|x_0|^2 + \ldots + |x_n|^2 = 1\} \cong S^{2n+1}.
$$
Then $\C\P^n$ is an image of the compact set $S^{2n+1}$ under a continuous map and is therefore compact.
\end{proof}

Let us look at small examples in more detail. For $n=1$ we have $\C\P^1 = \C^1\cup \C^1$. As a set, it is $\C\sqcup\{\infty\}$. The
two open sets $U_0$ and $U_1$ have coordinates $z = \frac {x_1}{x_0}$ and $z^{-1} = \frac {x_0}{x_1}$ respectively. 
It is often called the Riemann sphere. 

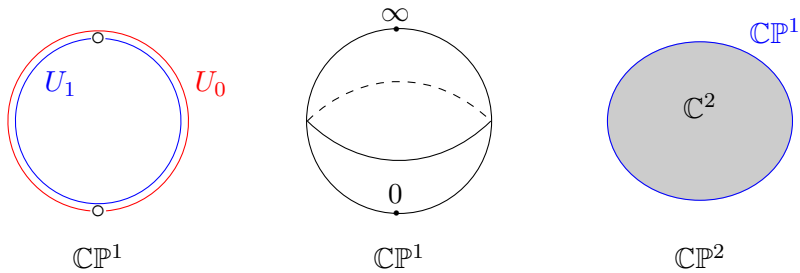
\begin{tikzpicture}
\draw[blue] (1.5,1.5) + (85:1.1) arc (85:95-360:1.1);
\draw[red] (1.5,1.5) + (-85:1.2) arc (-85:-95+360:1.2);
\draw (1.5,0) node[anchor=north]{$\C\P^1$};
\draw (1.5, 2.6) node {$\circ$};
\draw (1.5, 0.3) node {$\circ$};
\draw[red] (3, 2.0) node {$U_0$};
\draw[blue] (1.0, 2.0) node {$U_1$};
\draw (4.28,1.5)..controls(5,.8) and (6,.8)..(6.72,1.5);
\draw [dashed](4.28,1.5)..controls(5,2.2) and (6,2.2)..(6.72,1.5);
\draw (5.5,1.5) circle (35pt);
\draw (5.5,0) node[anchor=north]{$\C\P^1$};
\draw (5.45,2.7) node[anchor=south]{$\infty$};
\draw (5.5,2.73) node {\circle* 2};
\draw (5.45,0.27) node[anchor=south]{$0$};
\draw (5.5,0.27) node {\circle* 2};
\filldraw[gray!40!white] (9.5,1.5) ellipse (35pt and 30pt);
\draw[blue](9.5,1.5) ellipse (35pt and 30pt);
\draw (9.5,0.0) node [anchor=north]{$\C\P^2$};
\draw (9.5,2.0) node [anchor=north]{$\C^2$};
\draw[blue] (10.5,3.0) node [anchor=north]{$\C\P^1$};
\end{tikzpicture}

For $n=2$,  $\C\P^2$ is not just $\C^2\sqcup \{{\rm point}\}$. Rather, we have to add a whole line at infinity, because the complement to 
$\C\P^2 \setminus U_0$ consists of points $(0:x_1:x_2)$, which is a copy of $\C\P^1$. In fact, more generally we have 
$\C\P^n = \C^n\sqcup \C\P^{n-1}$.

\smallskip
The reason why algebraic geometers like $\C\P^n$ so much is because it is the ``smallest" way to compactify $\C^n$ to a compact complex manifold. In contrast, if we just wanted a compactification as a real manifold, we could do it with adding a single point.

\smallskip
Now  we will talk about algebraic subvarieties in $\C\P^n$. Consider the ring $\C[x_0,\ldots,x_n]$ equipped with the grading by the total degree of monomials
$$
\deg  \left(x_0^{a_0}\cdot\ldots \cdot x_n^{a_n}\right) = a_0+\ldots + a_n,
$$
in other words, each $x_i$ is homogeneous of degree $1$. 

\smallskip
An ideal $I \subseteq \C[x_0,\ldots,x_n]$ is called {\blue homogeneous} if for all $f\in I$ all of the homogeneous components of $f$ lie in $I$. Equivalently, an ideal $I$ is homogeneous if it has a set of homogeneous generators.
Similarly to the case of $\C^n$,  for a homogeneous ideal $J$ we define the set $Z(J)\subseteq \C\P^n$ of its zeros and for a set $V\in \C\P^n$ we define the homogeneous ideal $I(V)$ of polynomials that vanish on $V$. As in the case of $\C^n$, the sets $Z(J)$ define Zariski topology on $\C\P^n$ (these sets are declared to be closed). 

\smallskip
The relationship between ideals and sets is now more complicated. The ring $\C[x_0,\ldots, x_n]$  has a unique maximal homogeneous ideal $(x_0,\ldots,x_n)$ called the {\blue irrelevant ideal}.\footnote{ As the reader may guess, the irrelevant ideal is in fact very important, although we will not really see it in these notes.} The zero set of the irrelevant ideal is empty. Points in $\C\P^n$ correspond instead to homogeneous ideals given by $n$ linearly independent degree one generators. For example
 $$
 (x_1-\alpha_1x_0, \ldots, x_n-\alpha_n x_0),~(\alpha_1,\ldots, \alpha_n)\in \C^n
 $$
gives a point in $U_0\subset \C\P^n$.

\smallskip
We conclude by writing several examples of algebraic varieties in projective spaces.

$\bullet$ $~\{x_0x_2-x_1^2 =0\}\subset \C\P^2.$ We will later see that it is isomorphic to $\C\P^1$.

$\bullet$ $~\{x_0x_3-x_1x_2=0\}\subset \C\P^3.$ This one is isomorphic to $\C\P^1\times \C\P^1$.

$\bullet$ $~\{x_0x_2-x_1^2 =x_0 x_3 -x_1 x_2 = x_1 x_3 - x_2^2 =0\}\subset \C\P^2.$  Twisted cubic, isomorphic to  $\C\P^1$.

$\bullet$ $~\{y^2-4x^3-g_2\, x\,z^2 - g_3\,z^3=0\}\subset \C\P^2$ with coordinates $(x:y:z).$ Elliptic curve.

$\bullet$ $~\{x^3y+y^3z+z^3x=0\}\subset \C\P^2.$  Klein quartic; it has $168$ automorphisms.

$\bullet$ Variables  $x_{ij}$ for $1\leq i<j\leq n$ are homogenous coordinates in $\C\P^{\frac {n(n-1)}2-1}$. The equations 
are $x_{ij}x_{kl}-x_{ik}x_{jl}+x_{il}x_{jk} = 0$ for all $1\leq i<j<k<l\leq n$. This is the Grassmannian $\Gr(2,n)$ in its Pl\"ucker embedding.

\smallskip
{\bf Exercise 1.} Prove that for a proper ideal $I\subset \C[x_1,\ldots, x_n]$ the set $Z(I)\subseteq \C^n$ is irreducible if and only if $\sqrt I$ is a prime ideal.

{\bf Exercise 2.} Prove that $Z(I)$ is reducible if and only if it contains a proper nonempty subset which is both open and closed in the induced Zariski topology.

{\bf Exercise 3.} Verify that the transition functions for the cover of $\C\P^n$ by $U_i$ are indeed holomorphic.

\section{Hilbert polynomial. Statement of Bezout's Theorem on intersections of plane curves.}\label{sec.B1}
The algebraic manifestation of $\C\P^n$ and closed subsets $Z(I)\subseteq \C\P^n$ being compact is that the graded ring
$$
\C[x_0,\ldots,x_n]/I
$$
for a homogeneous ideal $I$ has finite-dimensional (as $\C$-vector spaces) graded components. This  observation should be contrasted with the $\C^n$ case where $\C[x_1,\ldots,x_n]/I$ is just an infinite-dimensional $\C$-vector space. Moreover, it turns out that the integers
$$
a_d := \dim_{\C} (\C[x_0,\ldots,x_n]/I)_{\deg = d}
$$
satisfy a number of fascinating properties. We will first state these properties, then give examples, and then prove them.

\smallskip
\begin{thm}\label{Hilbertpoly}
For a homogenous ideal $I$, the numbers $a_d$ defined above satisfy the following.
\begin{itemize}
\item There exists a one-variable polynomial $P$ with rational coefficients such that for all large enough $d$ we have $P(d)=a_d$. This $P$ is called the {\blue Hilbert polynomial} of $\C[x_0,\ldots,x_n]/I$ or of $Z(I)$.

\item Degree of $P$ is equal to the dimension of $Z(I)$.

\item The {\blue Hilbert function}
$
f(t): = \sum_{d\geq 0} a_d t^d
$
is a rational function of the form 
$$
f(t) = \frac {g(t)}{(1-t)^{\dim Z(I) + 1}}
$$
with a polynomial $g(t)$ such that $g(1)\neq 0$.
\end{itemize}
\end{thm}

\smallskip
The first example to consider is $I=\{ 0\}$. Then $a_d$ equals to the number of monomials of total degree $d$ in $n+1$ variables, which in turn equals to the number of ways of picking $n$ separators among $n+d$ dots. So we get for all $d\geq 0$
$$
a_d = \left(\begin{array}{c}n+d\\n\end{array}\right) = \frac 1{n!} (d+n)\ldots (d+1)
$$
which is clearly a degree $n$ polynomial in $d$. The generating function
\begin{align*}
&\sum_{d\geq 0}  \frac {t^d}{n!} (d+n)\ldots (d+1) =
\sum_{d\geq -n}  \frac {t^d}{n!} (d+n)\ldots (d+1) \\&
= \frac 1{n!} \left( \frac {\partial}{\partial t}\right)^n\left(\frac 1{1-t}\right)
=\frac 1{(1-t)^{n+1}}
\end{align*}
so the claims of Theorem \ref{Hilbertpoly} hold. 

\smallskip
Our next example is a curve in $\C\P^2$. Let $f(x_0,x_1,x_2)$ be a homogeneous polynomial of degree $k$ without repeated factors.\footnote{ As in the $\C^n$ case, we want our ideal $I$ to be radical.} Let $I=(f)$. Since multiplication by $f$ is an injective map 
in $\C[x_0,x_1,x_2]$, we have 
$$
(\C[x_0,x_1,x_2]/I)_{\deg = d} = (\C[x_0,x_1,x_2]_{\deg = d})/ (\C[x_0,x_1,x_2]_{\deg = d-k})
$$
and 
$$
a_d = \left\{ \begin{array}{ll}
\frac {(d+2)(d+1)}2 -  \frac {(d-k+2)(d-k+1)}2,&{\rm if~} d\geq k;\\[.2em]
\frac {(d+2)(d+1)}2,&{\rm if~} 0\leq d< k.
\end{array}
\right.
$$
Consequently, for $d$ large enough we get 
$$
a_d = dk  - \frac {k(k-3)}2
$$
which is linear in $d$.
For the generating function we get 
$$f(t) = \frac 1{(1-t)^3} - \frac {t^k}{(1-t)^3} = 
\frac {\sum_{i=0}^{k-1}t^i}{(1-t)^2},
$$
so again the conditions of Theorem \ref{Hilbertpoly} hold.

\smallskip
We will now discuss the proof of Theorem \ref{Hilbertpoly}. First of all, it is convenient to  consider not just $\C[x_0,\ldots,x_n]/I$ but all finitely generated graded modules $M$ over $\C[x_0,\ldots,x_n]$. We will also introduce the {\blue grading shift} notation:
For a graded module $M$, the graded module $M(r)$ has the same module structure, but the grading is redefined by
$$
M(r)_{\deg = k} = M_{\deg = k+r}.
$$
Then for any graded module $M$, multiplication by $x_n$ gives a map of graded modules\footnote{ It is good for one's mental health to only consider degree zero maps of graded modules, and use grading shifts when needed.} 
$$
M(-1) \stackrel {x_n}\longrightarrow M.
$$
The kernel and cokernel of this map are whimsically denoted by $K$ and $C$ respectively.
The exact sequence 
$$
0\to K\to M(-1)  \stackrel {x_n}\longrightarrow M \to C\to 0
$$
induces exact sequences at each degree, so we get\footnote{ This is Exercise 2.} for the generating functions
\begin{equation}\label{ex2}
f_{K}(t) - f_{M(-1)}(t) + f_M(t) - f_{C}(t) = 0.
\end{equation}
We then observe that $f_{M(-1)}(t) = t f_M(t)$ and we get 
$$(1-t)f_M(t) = f_C(t) - f_K(t).$$ It remains to argue that $K$ and $C$ are finitely generated modules over $\C[x_0,\ldots,x_{n-1}]$ and apply induction on $n$ (with $n=0$ case just graded finite dimensional vector spaces) to see that $f(t)$ is a rational function with denominator that is a power of $(1-t)$. The relation to $\dim Z(I)$ is significantly more difficult to prove and the interested reader is referred to 
\cite{AM} or \cite{Eisenbud}. 

\smallskip
\begin{rem}\label{rationalfunctionsprojective}
One can compute the field of rational functions on $Z(I)$ in a manner analogous to the affine case (Remark \ref{rationalfunctions}).
Specifically, for a prime homogeneous ideal $I$ we consider the graded integral domain $A=\C[x_0,\ldots,x_n]/I$ and the field of degree zero fractions, i.e. fractions of homogeneous elements of the same degree:
$$
{\rm Rat}(Z(I))= \left\{\frac {g_1}{g_2},~\deg g_1 = \deg g_2,~g_1,g_2\in A\right\} \subset QF(A).
$$
\end{rem}

\smallskip{\blue \hrule}
We will now talk about the Bezout's theorem. To begin with, one can show that all algebraic curves $C$ in $\C\P^2$ are given by a single homogeneous polynomial equation $f(x_0,x_1,x_2)=0$ of some degree $\deg C$, with $f$ not having any repeated irreducible factors.  Irreducible factors of $f$ correspond to the irreducible components of the curve. 
\smallskip
\begin{thm}\label{Bezout}(Bezout's theorem)
Let $C_1,C_2\subset \C\P^2$ be curves with no common components. Then the number of intersection points of $C_1$ and $C_2$, counted with multiplicities, is equal to $(\deg C_1)(\deg C_2)$. 
\end{thm}

We will delay the proof of Theorem \ref{Bezout} until the next section, but before we even start, 
the statement of the theorem makes no sense without us defining the multiplicity of intersection of two curves. For simplicity, let us assume that $C_1$ and $C_2$ intersect at a point $(1:0:0)$. We will write their equations in the open set $U_0$ as polynomials $f$ and $g$ in $\C[y_1,y_2]$ where $y_i=\frac {x_i}{x_0}$. Consider the subring $R$ of the field of rational functions $\C(y_1,y_2)$ given by
\begin{equation}\label{R}
R=\left\{\frac {F(y_1,y_2)}{G(y_1,y_2)}, {\rm ~such~that~} F,G\in \C[y_1,y_2],\,G(0,0)\neq 0\right\}.
\end{equation}
Consider $R/(f,g)R$. Because $f$ and $g$ have no common factors, it turns out\footnote{ We will actually see it as part of the proof of the Bezout's theorem in the next section.} to be a finite-dimensional vector space over $\C$, and we define
\begin{equation}\label{mult}
{\rm mult}_{(1:0:0)}(\{f=0\}\cap\{g=0\}) = \dim_\C R/(f,g)R.
\end{equation}

Let us now do some examples of intersection multiplicities to confirm that they do what we would like them to do.

\smallskip
First, we will check that two distinct lines intersect with multiplicity one. Without loss of generality, we may assume that these lines are $\{x_1=0\}$ and $\{x_2=0\}$. Then the ideal generated by $y_1$ and $y_2$ in $R$ is precisely the kernel of the evaluation map $R\to \C$ with sends $F(y_1,y_2)/G(y_1,y_2)$ to $F(0,0)/G(0,0)$, and thus the quotient is one-dimensional. 

\smallskip
Now let us look at the intersection at $(1:0:0)$ of $\{x_2=0\}$ and $\{x_0x_2-x_1^2=0\}$. In the affine coordinates $y_1,y_2$ on $U_0$ this becomes the intersection of the line $\{y_2=0\}$ and the standard parabola $\{y_2-y_1^2=0\}$. We will take the quotient of $R$ by $y_2$ and $y_2-y_1^2$ in two steps. At the first step, we mod out by $y_2$ to get the subring of the field $C(y_1)$ which consists of all rational functions with denominators that are nonzero at $y_1=0$. Then we need to further take a quotient by $y_1^2$ (since $y_2-y_1^2$ is now simply $y_1^2$) which gives us a two dimensional space, with basis being the images of $1$ and $y_1$. So we see that the horizontal line and the standard parabola intersect with multiplicity $2$, which makes sense.

\smallskip
{\bf Exercise 1.} Prove that for a sequence of integers $a_d,\,d\geq 0$ the first two conditions of Theorem \ref{Hilbertpoly} are equivalent to the last condition of it. 

{\bf Exercise 2.} Prove \eqref{ex2}.

{\bf Exercise 3.} Prove that the intersection multiplicity of $\{y_2=0\}$ and $\{g(y_1,y_2)=0\}$ at $(0,0)$ is the multiplicity of the zero of 
$g(y_1,0)$ at $y_1=0$.

\section{Proof of Bezout's Theorem on intersection of plane curves.}\label{sec.B2}
In this section, we will give a proof of Bezout's theorem along the lines of the first chapter of \cite{Hartshorne}. Some commutative algebra concepts will be explained as we go along.

\smallskip
As in the statement of Theorem \ref{Bezout}, we have two homogeneous polynomials $f_1(x_0,x_1,x_2)$ and $f_2(x_0,x_1,x_2)$ without common factors, which give plane curves $C_1=\{f_1=0\}$ and $C_2=\{f_2=0\}$. The main idea is to consider the ring $\C[x_0,x_1,x_2]/(f_1,f_2)$  as a graded module 
over  $\C[x_0,x_1,x_2]$.

\smallskip
\begin{prop}\label{filter}
For every nonzero finitely generated graded module $M$ over a graded Noetherian ring $A$ there exists a finite filtration
$$
0=M_0\subsetneq M_1\subsetneq \ldots \subsetneq M_k=M
$$
so that for each $i\in \{1,\ldots,k\}$ the quotient module
$M_i/M_{i-1}$ is isomorphic to $(A/p_i)(d_i)$ where $p_i$ is a homogeneous prime ideal in $A$ and $d_i\in\Z$ is a grading shift.
\end{prop}

\begin{proof}
We first observe that it is enough to show that every nonzero module $M$ has a submodule $M_1$ isomorphic to $(A/p)(d)$ for some homogeneous prime ideal $A$ and some degree shift $d$. Indeed, by the Noetherian property of $M$ there exists a maximum graded submodule $M'$ of $M$ for which the filtration of Proposition \ref{filter} exists. If this submodule is not $M$, then there would be a submodule in $M/M'$ isomorphic to $(A/p)(d)$ and then its preimage in $M$ would  contradict maximality of $M'$.

To find a submodule of $M$ isomorphic to $(A/p)(d)$, let us ask ourselves what this means. Such isomorphism would send $1\in (A/p)(d)$ to some element $m\in M$. The degree of $1$ in $(A/p)(d)$ is, by the definition of the grading shift,\footnote{ Never trust any signs in any mathematical text ever!} equal to $(-d)$, and the ideal $p$ can be recovered as 
$$
p = {\rm Ann}(m):= \{a\in A,{\rm ~such~that~}a\hskip 1pt m=0\}.
$$
Conversely, given any nonzero homogeneous element $m\in M$, the submodule $(m)\subseteq M$  generated by $m$ is
isomorphic to 
$(A/ {\rm Ann}(m))(-\deg m)$.

So now we are simply looking for a nonzero homogeneous element $m\in M$ whose annihilator is prime. As with many things in mathematics, we will try to find such $m$ by imposing on it some maximum or minimum property. As it happens, the right thing to consider is nonzero homogeneous $m$ such that ${\rm Ann}(m)$ is maximum among all such annihilators, in the sense that there are no nonzero homogeneous $m'$ with ${\rm Ann}(m')\supsetneq {\rm Ann}(m)$. Existence of such $m$ is assured  by $A$ being Noetherian. Now suppose that the ideal ${\rm Ann}(m)$ is not prime. It means that there exist $a,b\in A$ with $a\hskip 1pt b\hskip 1ptm=0$, $a\hskip 1pt m\neq 0$, $b\hskip 1pt m\neq 0$. Then ${\rm Ann}(b\hskip 1pt m)$ contains ${\rm Ann}(m)$ and it is strictly larger because it contains $a$ while ${\rm Ann}(m)$ does not. This finishes the proof.
\end{proof}

Filtrations  of Proposition \ref{filter} are typically non-unique, see Exercise 1.
We also remark that the same statement applies to finitely generated modules over usual (non-graded) Noetherian rings, and the argument is essentially the same. 

\smallskip
Returning to the Bezout's theorem,
what primes can appear in a filtration of $M=\C[x_0,x_1,x_2]/(f_1,f_2)$? The annihilator of any element of $M$ contains $(f_1,f_2)$, so we have $p\supseteq (f_1,f_2)$. Dimension theory then implies that these are primes that correspond to intersection points of $\{f_1=0\}$ and $\{f_2=0\}$ or the irrelevant ideal $(x_0,x_1,x_2)$. 

We will now exploit the fact that when we have a filtration of finitely generated graded modules as in Proposition \ref{filter}, the Hilbert polynomial of $M$ is equal to the sum of the Hilbert polynomials of $(A/p_i)(d)$. The Hilbert polynomial of $\C[x_0,x_1,x_2]/(f_1,f_2)$ can be computed from the short exact sequence\footnote{ Exercise 2 shows that the multiplication map by $f_2$ is injective on $\C[x_0,x_1,x_2]/(f_1).$}
\begin{align*}
0\to\C[x_0,x_1,x_2]/(f_1) (-\deg f_2) \stackrel {f_2}\longrightarrow \C[x_0,x_1,x_2]/(f_1) 
\to \C[x_0,x_1,x_2]/(f_1,f_2)\to 0.
\end{align*}
By the computation of the previous section, the Hilbert polynomial $P(k)$ of $\C[x_0,x_1,x_2]/(f_1)$  is $k(\deg f_1) + {\rm c}$ for some constant $c$. Therefore, the Hilbert polynomial of $\C[x_0,x_1,x_2]/(f_1,f_2)$ is 
$$
(k \deg f_1 + c ) - ((k-\deg f_2) \deg f_1 + c ) = (\deg f_1)(\deg f_2).
$$
Let us now compute the Hilbert polynomials for $(\C[x_0,x_1,x_2]/p)(d)$ where $p$ is either the irrelevant ideal or the ideal of a point. For the irrelevant ideal, we get $\C[x_0,x_1,x_2]/(x_0,x_1,x_2)=\C$, located in degree zero. Regardless of the shift, the Hilbert polynomial is zero. For an ideal that corresponds to the point, we can take a linear change of variables and then look at $\C[x_0,x_1,x_2]/(x_1,x_2) \cong \C[x_0]$. The Hilbert polynomial is $1$, which is again unchanged under shifts.

\smallskip
Consequently, to prove Bezout's theorem, it suffices to argue that the ideal $p$ of a point in $\C\P^2$ appears exactly as many times in the filtration of Proposition \ref{filter} for $M = \C[x_0,x_1,x_2]/(f_1,f_2)$ as the multiplicity of intersection at that point. This is a straightforward calculation if one is comfortable with the concept of localization in commutative algebra, which we will now review. 

\smallskip{\blue \hrule}
{\bf Detour: Localization of commutative rings and modules.}
The idea of localization is to look at a ring $A$ and then ask ``what would happen if such and such elements of $A$ were invertible"? It is a common theme in algebra to ask for something and then see if there is some universal way of accomplishing it. For example, the quotient ring is the universal way of setting some elements of a ring to zero. 

It turns out that it is indeed possible for any subset $S$ of $A$ to find a ring $A_S$ together with the ring homomorphism\footnote{ In this setting all rings are commutative and associative with $1$ and ring homomorphisms send $1$ to $1$.}  $A\to A_S$ such that every ring homomorphism $A\to B$ that sends elements of $S$ to invertible elements of $B$ uniquely factors through $A\to A_S$.

For the actual construction, if we are asking for elements of a set $S$ to be invertible, we might as well ask it for their products. So we will assume that the set $S$ is closed under multiplication. Elements of $A_S$ are then formal fractions $\frac as$ with $a\in A$ and $s\in S$, up to an equivalence relation
$$
\frac as \sim \frac {s'a}{s's}
$$
for all $a\in A$ and $s,s'\in S$.

\smallskip
{\blue I lied, but only a little bit.} The above $\sim$ is not an equivalence relation, rather we need to take the smallest equivalence relation that it generates. If $A$ is an integral domain, then we can say $\frac {a_1}{s_1} \sim \frac {a_2}{s_2}$ if and only if $s_1a_2 = s_2a_1$, but if $S$ has some zero divisors, this would still not be enough to ensure transitivity of the relation, and the correct definition of equivalence is
\begin{align*}
\left(\frac {a_1}{s_1} \sim \frac {a_2}{s_2}\right) \iff \left( \exists s_3\in S,~{\rm such~that~}s_3(s_1a_2 - s_2a_1)=0   \right).
\end{align*}

\smallskip
What's even better is that we can localize not only rings but modules over rings. If we have a module $M$ over a ring $A$ with a multiplicatively closed subset $S$, then we define the module $M_S$ as the set of fractions $\frac ms$ with $m\in M, s\in S$ up to the equivalence relation 
\begin{align}\label{locm}
\left(\frac {m_1}{s_1} \sim \frac {m_2}{s_2}\right) \iff \left( \exists s_3\in S,~{\rm such~that~}s_3(s_1m_2 - s_2m_1)=0   \right).
\end{align}
Then $M_S$ is an $A_S$-module, and localization of modules is a functor from the category of $A$-modules to that of $A_S$-modules that preserves exact sequences. These are not entirely trivial statements, and the reader may want to look in \cite{AM} or \cite{Eisenbud} for more details.

\smallskip{\blue \hrule}

\smallskip
We now go back to the proof of the Bezout's theorem. 

\smallskip
We have a filtration of $M = \C[x_0,x_1,x_2]/(f_1,f_2)$ from Proposition \ref{filter} and we want to find the number of times the ideal $p=(x_1,x_2)$ appears in $M_i/M_{i-1} \cong (A/p_i)(d_i)$. Let us consider the multiplicatively closed set $S$
$$
S=\{{\rm homogeneous~elements~of~}\C[x_0,x_1,x_2]{\rm~ which ~are~ not~ zero~at~} (1:0:0)\}.
$$
In particular, $S$ contains $x_0$, and if we localize $\C[x_0,x_1,x_2]$ at it first we obtain $\C[\frac {x_1}{x_0},\frac {x_2}{x_0},x_0^{\pm 1}]$ (Laurent polynomials in $x_0$ over the polynomial ring $\C[y_1,y_2]$ for $y_i=\frac {x_i}{x_0}$). When we localize further by the image of $S$ in $\C[\frac {x_1}{x_0},\frac {x_2}{x_0},x_0^{\pm 1}]$ we end up with the $\Z$-graded ring $R[x_0^{\pm 1}]$ where the ring $R$ is defined in \eqref{R}.
If we localize the module $M$, we end up with $(R/(\tilde f_1,\tilde f_2))[x_0^{\pm 1}]$ where $\tilde f_i = \frac {f_i}{x_0^{\deg f_i}}$ is the dehomogenization of $f_i$. We can then take degree $0$ part of this space to get exactly $R/(\tilde f_1,\tilde f_2)$ that was used to define the multiplicity of the intersection at the point $(1:0:0)$.

\smallskip
So we will take the filtration of $M$ from Proposition \ref{filter} and observe that the operations of taking localization in $S$ and degree $0$ parts preserves the filtration property, with the associated graded objects changing accordingly. If we localize 
$M = \C[x_0,x_1,x_2]/(x_0,x_1,x_2)$
in $S$ we get $0$, because $x_0$  both annihilates $M_S$ and is invertible. Similarly, localization of $\C[x_0,x_1,x_2]/p$ for a prime ideal that corresponds to a point $p$ other than $(1:0:0)$ is zero because there is an element in $S$ which is zero at $p$ but not at $(1:0:0)$ and is therefore both zero and invertible on the localization. Finally, $(\C[x_0,x_1,x_2]/(x_1,x_2))_S = \C[x_0^{\pm 1}]$ and the degree zero of any grading shift of it is a dimension one vector space over $\C$. Thus, the number of occurrences of the ideal of $(1:0:0)$ in the filtration is exactly equal to the dimension of $R/(\tilde f_1,\tilde f_2)$ which is the multiplicity of the intersection at this point by \eqref{mult}. The same is true for all other points in $\C\P^2$, which finishes the proof of Bezout's theorem.

{\bf Exercise 1.} Consider the graded ring $\C[t]$ with $t$ homogenous of degree one. Consider $M=\C[t]$ as a module over itself. For any $k\geq 1$, construct a filtration of Proposition \ref{filter} of length $k$. 

{\bf Exercise 2.} Show that under the assumption that $f_1,f_2\in \C[x_0,x_1,x_2]$ have no common irreducible factors, the multiplication by $f_2$ map is injective on $\C[x_0,x_1,x_2]/(f_1)$.

{\bf Exercise 3.} Show that \eqref{locm} defines an equivalence relation.

\section{Line bundles and vector bundles.}\label{sec.vb}
We will now talk about a very important concept that will follow us throughout these notes, that of a vector bundle. We will mostly do this in the setting of complex manifolds rather than algebraic varieties, with an honest assurance that the algebraic case is essentially the same.

\smallskip
Let $X$ be a complex manifold. Informally, a rank $r$ vector bundle $W$ over $X$ is another complex manifold, with a map $\pi:W\to X$ such that:
\begin{itemize}
\item 
For each $p\in X$ the fiber $\pi^{-1}(x)$ has a structure of a dimension $r$ complex vector space.
\item
This vector space structure varies holomorphically as $x$ changes.
\end{itemize}

Now let us look at the formal definition.

\smallskip
\begin{defn}\label{vector.bundle}
A vector bundle $W$ on a complex manifold $X$ is the a complex manifold equipped with a holomorphic map
$\pi:W\to X$ and a structure of a complex vector space on each fiber, that satisfies the following property. For every $x\in X$ there exists an open set $U\ni x$ and a biholomorphism $\pi^{-1}(U)\cong \C^r\times U$, such that the diagram
\begin{align}\label{vb}
\begin{array}{c}
\pi^{-1}(U) \hskip 5pt \cong \hskip 5pt  \C^r\times U\\
\searrow\hskip 20pt  \swarrow\\
U
\end{array}
\end{align}
commutes and is compatible with the complex vector space structures on the fibers.
\end{defn}

\smallskip
\begin{rem}\label{line.bundle}
A vector bundle $W\to X$ of rank $r=1$ is called a line bundle.
If a bunch of mathematicians are talking, you can recognize algebraic geometers by how frequently they say ``line bundle". It is really  important in algebraic geometry, for reasons we will see shortly.
\end{rem}

\smallskip
\begin{rem}
Morphisms of vector bundles over $X$ are defined as holomorphic maps that commute with projections to $X$ and are linear maps on each fiber.
A vector bundle $W\to X$ of rank $r$ is called  {\blue trivial} if it is isomorphic as a vector bundle to $\C^r\times X\to X$ with the projection map.\footnote{ Nothing is quite as disheartening as not understanding what something ``trivial" is.} One can think of arbitrary $W\to X$ as being locally trivial, i.e. there is a cover of $X$ such that the restriction of $W\to X$ to each open set of the cover is trivial.
\end{rem}

\smallskip
\begin{rem}
If $X$ has a structure of an algebraic variety, then a vector bundle $W\to X$ must be locally trivial in Zariski topology. Otherwise, the same definition applies, with holomorphic maps replaced by maps of algebraic varieties.
\end{rem}

\smallskip
Another way of thinking about a vector bundle is in terms of its {\blue transition functions}. Suppose we have two open sets $U$ and $V$ in $X$ with the biholomorphisms as in \eqref{vb}. Restrictions to $U\cap V$ give a commutative diagram
$$
\begin{array}{c}
 \C^r\times (U\cap V)\cong\pi^{-1}(U\cap V)   \cong\C^r\times (U\cap V)\\
\searrow \hskip 25pt \downarrow \hskip 25pt\swarrow \\ 
U\cap V
\end{array}
$$
and therefore an automorphism of $\C^r\times (U\cap V)$ which is compatible with the projection to $U\cap V$ and the vector space structure on the fibers of this projection. Such automorphisms are precisely of the form
$$
(w,z)\mapsto (A(z)w,z)
$$
where $A$ is a holomorphic function from $U\cap V$ to ${\rm GL}(r,\C)$. These functions $A$ must satisfy a certain cocycle condition for triple intersections. Conversely, given such functions with the appropriate compatibility on triple intersections, we can use them to glue together copies of $\C^r\times U$ to build a vector bundle.

\smallskip
\begin{rem}
In the algebraic setting, the transition functions must be given by matrices whose entries are rational functions with denominators that have no zeros on $U\cap V$.
\end{rem}

\smallskip
If $X$ is a point, then a rank $r$ vector bundle over $X$ is simply a dimension $r$ complex vector space. So it is not too surprising that whatever one can do with vector spaces one can also do with  vector bundles. For instance, we have direct sums, tensor products, symmetric and exterior powers, duality, and fiber-wise homomorphisms. Some basic properties of these constructions are listed below.
\begin{itemize}
\item
 $\rk(W_1\oplus W_2) = \rk W_1 + \rk W_2$

\item
$\rk(W_1\otimes W_2) =( \rk W_1 )( \rk W_2)$

\item
$\rk({\rm Sym}^k W) = \left(\begin{array}{c} {{\rk W}+k-1}\\{k}\end{array}\right)$

\item
$\rk(\Lambda^k W) = \left(\begin{array}{c} {{\rk W}}\\{k}\end{array}\right)$

\item 
$\rk(W^\vee) = \rk (W)$

\item
$Hom(W_1,W_2) = W_2 \otimes W_1^\vee$
\end{itemize}

\smallskip
{\blue \hrule}

Did I mention that line bundles were of special interest? 

\smallskip
\begin{defn}
Picard group ${\rm Pic}(X)$ is the (abelian) group of line bundles on $X$, up to isomorphism (as line bundles). The product is $\otimes$, the identity is the trivial line bundle $\C\times X\to X$, and the inverse is $W\mapsto W^\vee$.
\end{defn}

\smallskip
\begin{rem}
The Picard group of a manifold $X$ is an important invariant of $X$. It is a close relative of the class group of a number field. 
\end{rem}

\smallskip
An important construction that we will use later is that of a {\blue pullback} of vector bundles. Suppose we have a vector bundle $\pi:W\to Y$ and a holomorphic map $\mu:X\to Y$. Then the pullback $\mu^*W\to X$ is defined as the fiber product $X\times_Y W$, i.e. the subvariety in $X\times W$ of pairs $(x,w)$ that map to the same point in $Y$ under $\mu$ and $\pi$. In elementary terms, the fiber of $\mu^* W$ over $x\in X$ is naturally identified with the fiber of $W$ over $\mu(x)$. We have a commutative diagram
$$
\begin{array}{ccc}
\mu^*W &\longrightarrow &W\\
\downarrow&&~\downarrow \pi\\
X&\stackrel \mu\longrightarrow&Y
\end{array}
$$
of complex manifolds. Of course, $\rk \mu^*W = \rk W$. The construction also works in the algebraic setting.

\smallskip
We also need to talk about sections of line bundles.\footnote{The same applies to vector bundles, but we will focus on the line bundles.}
For a line bundle $\pi:L\to X$ we define a section $s$ of $L$ as the holomorphic function $s:X\to L$ such that $\pi\circ s= {\rm id}_X$.
In very simple terms, we are picking one point in each fiber of $L\to X$, in a holomorphic fashion.

\smallskip
The simplest example of a section is the {\blue zero section}. This means that in every fiber we pick the well-defined element $0$ of the corresponding vector space structure. For some line bundles this is the only (holomorphic) section.

\smallskip
When we pull back a line bundle, sections pull back with it. Namely, suppose we have $\pi:L\to Y$, $\mu:X\to Y$ and
the pullback commutative diagram below.
$$
\begin{array}{ccc}
\mu^*L &\longrightarrow &L\\
\downarrow&&~\downarrow \pi\\
X&\stackrel \mu\longrightarrow&Y
\end{array}
$$
Then to any section $s:Y\to L$ we can associate a section $\mu^* s:X\to \mu^* L$ by
\begin{align}\label{mus}
\mu^*s:x\mapsto (x, s(\mu(x)) \in  X\times L.
\end{align}
We observe that $(x,s(\mu(x))$ lies in $\mu^*L\subseteq X\times L$ because $s$ is a section, see Exercise 1.

\smallskip
We can also consider sections not for the whole $L\to X$ but only over some open subset $U\subseteq X$, i.e. the sections of the restriction of $L\to X$ to $\pi^{-1} U \to U$. Spaces of such sections  are sometimes denoted by $\Gamma(U,L)$. If $U$ is all of $X$, the space $\Gamma(X,L)$ is often called the space of {\blue global sections} of $L$.

\smallskip{\blue \hrule}
We will also introduce some natural bundles on smooth manifolds which will be used later. Namely, for any smooth manifold $X$ we can construct its holomorphic tangent bundle $TX\to X$. Sections of this bundle on $U$ are the same as vector fields on $U$. Its dual $TX^{\vee}\to X$ is called the cotangent bundle. Sections of the cotangent bundle over $U$ are precisely the holomorphic $1$-forms on $U$.\footnote{ When I was learning calculus, the symbol $dx$ really confused me. Was it a function of two variables? Of one variable? I only found peace when it became clear that $dx$ is a section of the cotangent bundle.} 

\smallskip
\begin{defn}\label{canonical}
The top exterior power of the cotangent bundle $\Lambda^{\dim X} TX^\vee$ is called the {\blue canonical line bundle} of $X$. Sections of it over $U$ are holomorphic $n$-forms on $U$.
\end{defn}

\smallskip
{\bf Exercise 1.} Verify that \eqref{mus} defines a section of $\mu^*L$.

{\bf Exercise 2.} Show that sections of a vector bundle $W\to X$ have a natural structure of a vector space over $\C$.

{\bf Exercise 3.} Verify that sections of the line bundle $\C\times X\to X$ are in natural bijection to holomorphic functions on $X$.

\section{Line bundle ${\mathcal O}(1)$ on $\C\P^n$. Maps to projective spaces.}\label{sec.maps}
We will now learn exactly why line bundles are so ubiquitous in algebraic geometry -- they are used to define maps to projective spaces.

\smallskip
We start by defining a very important line bundle on $\C\P^n$ which I will denote by $\mathcal O(1)$. Recall that $\C\P^n$ can be viewed as a set of lines through the origin in $\C^{n+1}$. To set our notations, a point $p=(x_0:\ldots:x_n)\in \C\P^n$ corresponds to the line 
$l_p=\{(tx_0,\ldots,tx_n),t\in\C\}$. We define {\blue as a set}
$$
\mathcal O(1) =\{(p,\varphi),~{\rm where~}p\in \C\P^n, \varphi \in l_p^\vee\}
$$
with the projection map  $\pi:\mathcal O(1)\to\C\P^n$ given by $(p,\varphi)\mapsto p$. While we immediately see that fibers of  $\pi$ are $l_p^\vee$ and thus have a natural one-dimensional vector space structure, we have not yet given $\mathcal O(1)$ any complex manifold 
structure. However, observe that over an open set $U_0=\{x_0\neq 0\}$ there is a natural bijection
\begin{equation}\label{O1}
\begin{array}{c}\pi^{-1}U_0 \cong \C\times U_0\\[.2em]
((x_0:\ldots:x_n),\varphi) \mapsto (\varphi(1,\frac {x_1}{x_0},\ldots,\frac {x_n}{x_0}),(x_0:\ldots:x_n)).
\end{array}
\end{equation}
In plain words, we can nicely pick a point in $l_p$ for $p\in U_0$ and evaluate $\varphi$ at that point. Of course, we can do the same for every index $i$ and then verify that the transition functions are holomorphic and rational, see Exercise 1.

\smallskip
We now observe that the geometric meaning of the homogeneous coordinates $x_0,\ldots, x_n$ is that they are sections of the line bundle $\mathcal O(1)\to \C\P^n$. Indeed, each $x_i$ is a linear function on $\C^{n+1}$ and thus restricts to a linear function on $l_p$. It is also clear that $x_i$ is  holomorphic (by looking at it on each $U_j$). 
It is in fact true that 
$$\Gamma(\C\P^n, \mathcal O(1)) = \C x_0\oplus\cdots\oplus\C x_n,$$
see Exercise 2.

\smallskip
Using the group operations of line bundles, we can define $\mathcal O(k)\to\C\P^n$ for all $k\in \Z$ as $k$-th power of $\mathcal O(1)$ in the Picard group ${\rm Pic}(\C\P^n)$. Notably, $\mathcal O(0)\cong \C\times \C\P^n$ is the trivial bundle and $\mathcal O(-1)$ has fiber over $p$ naturally identified with $l_p$.

\smallskip
{\blue \hrule}
We are now ready to discuss maps to $\C\P^n$. If we have a holomorphic map $f:X\to \C\P^n$, we can consider the pullback diagram
$$
\begin{array}{ccc}
f^*\mathcal O(1) &\longrightarrow &\mathcal O(1)\\
\downarrow&&\downarrow\\
X&\longrightarrow &\C\P^n
\end{array}
$$
and note that sections $x_i$ of $\mathcal O(1)\to \C\P^n$ pull back to holomorphic sections $s_i = f^*x_i$ of $f^*\mathcal O(1) \to X$. An important observation is that for every point $p\in X$ at least one of the sections $f^*x_i$ does not equal zero in the fiber of $f^*\mathcal O(1) \to X$.\footnote{ We can not make sense of a value of a section, because identification of the fiber with $\C$ is not canonical, but we can ask whether the value is zero, since the fiber has a natural vector space structure.} Indeed, the fiber of $f^*\mathcal O(1)\to X$  at $p$ is the fiber of $\mathcal O(1)\to\C\P^n$ at $f(p)$ and we know that not all $x_i$ are zero at $f(p)$.

In the other direction, if we have $n+1$ holomorphic sections $s_0,\ldots, s_n$ of a line bundle $L\to X$, which are not simultaneously zero at any point of $X$, we can define a map 
$X\to \C\P^n$ by 
$$
p\mapsto (s_0(p):\ldots :s_n(p)) 
$$
Note that this map is well-defined and holomorphic. The easiest way to see both is to appeal to a local biholomorphism with $\C\times U$. While not unique, it is defined up to a holomorphic invertible scaling of the fiber, so all $s_i(p)$ are multiplied by the same factor.

\smallskip
These two constructions also work in the algebraic setting, and are inverses of each other. I.e., we have the following bijection.
\begin{align*}
\{X\to \C\P^n~{\rm with~coords~}x_0,\ldots,x_n\} \longleftrightarrow \hskip 80pt\\
\{{\rm line~bundle~}L\to X,~{\rm with~sections~} s_0,\ldots,s_n{\rm ~with~no~common~zeros}\}
\end{align*}
We do not prove this claim, because the proof is about as enlightening as filing one's taxes, but a motivated reader is welcome to do so.

\smallskip
{\blue \hrule}

\smallskip
We will now look at some examples.
\begin{itemize}
\item
Consider $\C\P^1$ with the homogeneous coordinates $(y_0:y_1)$. Note that $y_0^2,y_0y_1,y_1^2$ can be viewed as sections of the line bundle $\mathcal O(2) = \mathcal O(1)^{\otimes 2}$. Then we can use them to define a map
$\C\P^1 \to \C\P^2$ by
\begin{equation}\label{conic}
(y_0:y_1)\mapsto (y_0^2 :y_0y_1:y_1^2).
\end{equation}
The image of this map is a smooth curve in $\C\P^2$ given by $x_0 x_2 - x_1^2 = 0$. In fact, any smooth conic\footnote{ a curve given by a quadratic equation} in $\C\P^2$ is isomorphic to this one. (Quadratic forms in three complex variables are uniquely determined by their rank; rank two conics are unions of two lines, all rank three conics are as above.) The fact that a conic curve can be parametrized by a line is something that we have already seen in the rationalizing substitution
$$
(x(t),y(t))=\Big(\frac  {2t}{t^2+1}, \frac {t^2-1}{t^2+1} \Big)
$$
which parametrizes the circle $x^2+y^2=1$. It is useful in computing certain indefinite integrals and in classifying Pythagorean triples.

This map is the simplest case of the so-called Veronese embedding.

\item Consider the map $\CP^1\to \CP^1$ given by
$$
(y_0:y_1)\mapsto (y_0^2:y_1^2).
$$
This map is well-defined, because $y_0^2$ and $y_1^2$ are never simultaneously zero. It can be viewed as the extension of the map $\C\to \C$ given by $z\mapsto z^2$ to $\C\P^1$ by sending $\infty$ to $\infty$. It is a holomorphic map of degree two, i.e. most points have two preimages. 

\item
Consider the product $\C\P^1\times \C\P^1$ with two projections $\pi_1,\pi_2$, and the line bundle
$$
\mathcal O(1,1):=\pi_1^*\mathcal O(1) \otimes\pi_2^*\mathcal O(1).
$$
Holomorphic sections of $\mathcal O(1,1)$  are homogeneous polynomials of bidegree $(1,1)$ in two sets of homogeneous coordinates. They define the map from $\C\P^1\times \C\P^1$ to $\C\P^3$  by
\begin{equation}\label{quadric}
((y_0:y_1),(z_0:z_1))\mapsto
(y_0 z_0: y_0z_1:y_1z_0:y_1z_1)
\end{equation}
The image is a degree two surface in $\C\P^3$ given by $x_0x_3-x_1x_2=0$. As in the $\C\P^2$ case, all maximum rank degree two (called {\blue quadric}) surfaces in $\C\P^2$ are isomorphic to it. 

We can see these two rulings on a {\blue real} quadric in the case of the hyperboloid of one sheet. It can be obtained by rotating a line around the axis, and this provides one of the rulings. There are also such rulings on the hyperbolic paraboloid -- these surfaces are projectively the same after compactification in $\R\P^3$.

\item
More generally, one can use sections of $\mathcal O(1,1)$ to embed $\C\P^m\times \C\P^n$ into $\C\P^{(m+1)(n+1)-1}$. It is called the Segre embedding. 

\end{itemize}

We are now ready to define {\blue ample} and {\blue very ample} line bundles.

\smallskip
\begin{defn}\label{ample}
Let $L\to X$ be line bundle over an algebraic variety $X$. We call $L$ very ample, if its global sections can be used to embed $X$ as a subvariety of some projective space. We call $L$ ample if some positive tensor power $L^{\otimes k},\,k\geq 1$ of $L$ is very ample.
\end{defn}

\smallskip
{\blue \hrule}

\smallskip
There is a coordinate-free version of  $\C\P^n$ and the maps to it. Let $V$ be a complex vector space of dimension $n+1$. Then we denote by $\P V = (V\setminus\{{\mathbf 0}\})/\C^*$ the space of lines in $V$. It comes with natural complex and algebraic structures. 
Of course, $\P V$ is isomorphic to $\C\P^n$, but the isomorphism depends on the choice of a basis of $V$. 

\smallskip
If we have a line bundle $L\to X$ and a map 
$$
\mu:W\to \Gamma(X,L)
$$
from a finite-dimensional complex vector space $W$ to the space of holomorphic sections of $L$, then for any $p\in X$ we consider the subspace $W_p\subseteq W$ 
of $w\in W$ such that $\mu(w)$ is zero at $p$. {\red Provided that the subspace $W_p$ is always proper} (this is the coordinate-free analog of $s_i$ not having common zeros) the codimension of $W_p$ is $1$, and we get a map 
$$
X\to \P W^\vee,~p\mapsto {\rm Ann}(W_x)\in \P W^\vee
$$
which sends $p$ to the line in $W^\vee$ which is the annihilator of $W_x$. This map is always holomorphic in the complex manifold setting and is a map of algebraic varieties in the algebraic variety setting.

\smallskip
\begin{rem}
If the map $W\to \Gamma(X,L)$ is not injective, then the image of $X$ in $\P W^{\vee}$ sits in a projective subspace of $\P W^{\vee}$ and vice versa. For example, the map
$$
(y_0:y_1)\mapsto (y_0^2:y_1^2:y_1^2)
$$
gives a degree two map from $\C\P^1$ to a line $\{x_1-x_2=0\}$ in $\C\P^2$.
\end{rem}

\smallskip
Later we will use the following result about automorphisms of $\C\P^n$.
\begin{prop}\label{autoP}
The automorphism group of $\C\P^n$ either as a complex manifold or as an algebraic variety is naturally identified with 
${\rm PGL}(n+1,\C)$. In the coordinate-free language, automorphisms of $\P V$ are identified with the group ${\rm PGL}(V)$
of vector space automorphisms of $V$ up to scaling.
\end{prop}

\begin{proof}
An invertible $(n+1)\times (n+1)$ complex matrix $A$ gives a linear change of coordinates and thus an automorphism of $\C\P^n$. This gives a group homomorphism
$${\rm GL}(n+1,\C)\to{\rm Aut}(\C\P^n)$$
whose kernel  are the constant diagonal matrices $\lambda \,{\rm Id},\,\lambda \neq 0$. Analogously, we get an injective group homomorphism ${\rm PGL}(V)\to{\rm Aut}(\P V)$.
To prove surjectivity, we need to argue that for every automorphism  $\mu:\C\P^n\to \C\P^n$ the pullback of $\mathcal O(1)$ is isomorphic to $\mathcal O(1)$. We delay the verification of this fact until Remark \ref{O1O1}.
\end{proof}

{\bf Exercise 1.} Verify that transition functions for the identification of \eqref{O1} (but for all $U_i$) are holomorphic and rational on $U_i\cap U_j$ and thus give $\mathcal O(1)$ the structure of a line bundle over $\C\P^n$ in both holomorphic or algebraic settings.

{\bf Exercise 2${}^*$.} Prove that every holomorphic section of $\mathcal O(1)\to\C\P^n$ is a linear combination of homogeneous coordinates. Hint: for any such section define the map $f:(\C^{n+1}\setminus\{{\mathbf 0}\})\to \C$ and show $f$ is holomorphic. Use Hartogs' Lemma to extend $f$ to a holomorphic function on $\C^{n+1}$ and use Taylor expansion and $f(\lambda z) = \lambda z$. The same argument shows that holomorphic sections of $\mathcal O(k)\to \C\P^n$ are zero for $k<0$ are are homogeneous polynomials of degree $k$ for $k\geq 0$.

{\bf Exercise 3.} Prove that the maps \eqref{conic} and \eqref{quadric} are embeddings of complex manifolds.

\section{Weil and Cartier divisors and class groups.}
In this section, we will talk about certain abelian groups associated to a complex algebraic variety, known as groups of {\blue Weil and Cartier divisors}. The term ``divisor" goes back to nineteenth century and is ultimately related to divisors of natural numbers, although this will not 
be immediately apparent.

\smallskip
We start with the Weil divisors.
Let $X$ be an algebraic variety. We consider the group of Weil divisors  ${\rm WeilDiv}(X)$, defined as the free group generated by elements $[Y]$  {\blue for each codimension one subvariety} $Y\subset X$. In other words, elements of ${\rm WeilDiv}(X)$ are finite formal linear combinations 
$$
\sum_{i} a_i [Y_i]
$$
where $a_i\in \Z$ and $Y_i$ are codimension one subvarieties of $X$.

\smallskip
If that were all there was to it, we would just have some huge free group, with an uncountable generating set, with nothing interesting to show for it. However, the key to this is that to any rational function $f$ on $X$ one can associate a Weil divisor ${\rm div} f$. 
The concept of a rational function, introduced in Definition \ref{rationalfunctions} and Remark \ref{rationalfunctionsprojective} 
is an algebraic analog of the concept of a meromorphic function\footnote{ Meromorphic functions of several variables are defined as functions which can be locally written as ratios of holomorphic functions.} and the two notions coincide for smooth complex projective varieties, although this is far from obvious. We will construct\footnote{\blue ~sort of} ${\rm div}f$ in a second but will now make two important definitions.

\smallskip
\begin{defn}\label{classgroup}
A Weil divisor on $X$ is called {\blue principal} if it is equal to ${\rm div}f$ for some rational function $f$. It is easy to see that principal divisors form a subgroup of ${\rm WeilDiv}(X)$. The quotient group   ${\rm Cl}(X)$ of the group of all Weil divisors by the subgroup of principal Weil divisors is called the {\blue class group} of $X$.
\end{defn}

\smallskip
For a meromorphic function of one variable, we can look at its zeros and poles, with multiplicities, and this is exactly what ${\rm div}f$ is meant to generalize. We will now get simultaneously vague and technical. 

\smallskip
For simplicity, we will talk about Weil divisors on projective varieties $X\subseteq \C\P^n$.
Let $I$ be a prime homogeneous ideal in $\C[x_0,\ldots,x_n]$, other than the irrelevant ideal. Consider (cf. Remark \ref{rationalfunctionsprojective}) a nonzero element $f$ of the field of degree zero fractions in the quotient field of $A=\C[x_0,\ldots,x_n]/I$, i.e.
$$
f = \frac {g_1}{g_2} 
$$
where $g_i$ are homogeneous elements of $A$ of the same degree. We would like to define 
$${\rm div}f = \sum_k a_k [Y_k]$$
under certain assumptions on $Z(I)$.

\smallskip
By a commutative algebra result known as Krull's Principal Ideal theorem, the minimal prime ideals $I_k$ of $A$ that contain $g_i$ correspond to codimension one subvarieties of $Z(I)$. These will be the $Y_k$ for ${\rm div}f$, but the question remains how to define $a_k$. It turns out that if $Z(I)$ is smooth in codimension one\footnote{ There is a Zariski open subset  $U$ of $Z(I)$ where the equations define a complex submanifold, such that $\dim_\C Z(I) -\dim_\C(Z(I)\setminus U) \geq 2$.} then the ring $A_k$ which is the degree zero part of the localization of $A$ by homogeneous elements not in $I_k$ is what's known as a DVR\footnote{ I am old enough to remember when this abbreviation acquired another meaning, by now already archaic.} ({\blue discrete valuation ring}). This implies that there is a group homomorphism $\nu_k$ from $QF_0(A)^*$ to the additive group of $\Z$, and we define 
$$
{\rm div}f = \sum_k \nu_k(f) [Y_k].
$$
If you feel that you are getting a run-around here, you are not entirely wrong. Still, it is good to know that a prototypical case of a DVR is the ring $\C[[t]]$ of formal power series in one variable, and the corresponding valuation $\nu$ from its field of fractions, the field of Laurent formal power series $\C((t))$, sends $f(t)$ to the degree of its leading term. So in this case it picks up the orders of zeros and poles at $t=0$ of a meromorphic function in $t$.

\smallskip
{\blue \hrule}

Let us look more specifically at the case $I=\{0\}$, i.e. $Z(I)=\C\P^n$. Codimension one subvarieties of $\C\P^n$ are in one-to-one correspondence to irreducible positive degree homogeneous polynomials in $\C[x_0,\ldots,x_n]$ up to scaling. It is not an obvious 
statement, rather it follows from $\C[x_0,\ldots,x_n]$ being a unique factorization domain. Any rational function on $\C\P^n$ can be written, uniquely up to a nonzero multiplicative constant, as 
$$
f=\displaystyle\prod_i F_i(x_0,\ldots,x_n)^{a_i}
$$
where $a_i$ are nonzero integers and $F_i$ are irreducible homogeneous polynomials, such that 
$$
\sum_{i} a_i \deg F_i = 0.
$$
Then the Weil divisor of $f$ works out to be 
\begin{equation}\label{zerosum}
{\rm div} f = \sum_i a_i [\{F_i=0\}].
\end{equation}
It is easy to see that the condition $\sum_i a_i \deg Y_i=0$ is the only condition on $\sum_i a_i [Y_i]$ being a divisor of a rational function on $\C\P^n$. Therefore, we have ${\rm Cl}(\C\P^n)\cong \Z$, with the long exact sequence
$$
1 \to \C^* \to {\rm Rat}(\C\P^n)^*\to {\rm WeilDiv}(\C\P^n)\stackrel{\deg}{\longrightarrow}\Z\to 0
$$
where the degree map sends $[Y]$ to $\deg Y$. Indeed, a divisor of a rational function is zero if and only if the function is a nonzero constant.

\smallskip
\begin{rem}
If we pick a hyperplane $H$, for example $\{x_0=0\}$, we get a natural splitting of the surjective map $\deg$ which sends $k$ to $k[H]$.
\end{rem}

\smallskip
\begin{rem}\label{O1O1}
In the proof of Proposition \ref{autoP} we needed the statement that for an automorphism $\mu:\C\P^n\to \C\P^n$ the pullback of $\mathcal O(1)$ is isomorphic to $\mathcal O(1)$. Since $\mathcal O(1)$ generates ${\rm Pic}(\C\P^n)$, the 
pullback  of $\mathcal O(1)$ must be  $\mathcal O(\pm 1)$. However,  $\mathcal O(-1)$ does not have any nonzero holomorphic sections and is thus excluded.
\end{rem}

\smallskip
{\blue \hrule}

\smallskip
There is another kind of divisors, namely {\blue Cartier divisors} which we will now describe. At a first glance, they have little to do with Weil divisors, but we will soon see that the two notions often coincide.

\smallskip
\begin{defn}
A Cartier divisor on $X$ is a collection of $(U_\alpha,f_\alpha)_{\alpha\in I}$ where $X = \bigcup_{\alpha\in I} U_\alpha$ is a Zariski open cover of $X$ and $f_\alpha$ are rational functions on $U_\alpha$ which are holomorphic on $U_\alpha$, such that for any $\alpha,\beta$ the function $\frac {f_\alpha }{f_\beta}$ is invertible on $U_\alpha\cap U_\beta$. In particular, for every $x\in X$ the functions $f_\alpha$  for $U_\alpha$ that contain $x$ differ by an invertible function. We moreover identify Cartier divisors $(U_\alpha,f_\alpha)_{\alpha\in I}$ and $(U_\alpha,f_\alpha)_{\alpha\in J}$ if they are compatible, i.e. the union $(U_\alpha,f_\alpha)_{\alpha\in I\sqcup J}$
is a Cartier divisor on $X$. Cartier divisors form a group, with the operation given by products of $f_\alpha$ on the common refinement of the covers.
\end{defn}

Here we are assuming that all $U_\alpha$ are some Zariski open subvarieties in $\C^n$, and the concept of rational function from Definition \ref{rationalfunctions} is applicable.

\smallskip
Compared to the definition of Weil divisor, the definition of a Cartier divisor is rather complicated. 
However, there is a natural way to associate to any Cartier divisor a Weil divisor. Namely, to any divisor $(U_\alpha,f_\alpha)_{\alpha\in I}$ we assign the unique Weil divisor on $X$ which restricts to ${\rm div} f_{\alpha}$ on each $U_\alpha$. In other words, we are reading off the zeros and poles of the rational functions $f_\alpha$, with multiplicities. The condition that $\frac {f_\alpha }{f_\beta}$ is invertible on $U_\alpha\cap U_\beta$ assures that these Weil divisors on $U_\alpha$ coincide on the intersections. Therefore, they can be glued together to a unique Weil divisor on $X$.

\smallskip
\begin{rem}
If we are working with {\red smooth} varieties $X$,
we can also go from a Weil divisor to a Cartier divisor. Indeed,\footnote{ It is not at all obvious, but true!} every codimension one subvariety $Y\subset X$ can be Zariski locally written as a divisor of a rational function. We can take these local descriptions and put them together to form a Cartier divisor. Then ratios $\frac {f_\alpha }{f_\beta}$ will have no poles or zeros on $U_\alpha\cap U_\beta$, which will make them invertible on them. I am suppressing the technical details.
\end{rem}

\smallskip
As with the Weil divisors, there is a concept of a principal Cartier divisor. It is simply given by $(X,f)$ for a rational function $f$. Clearly, these are identified with principal Weil divisors. In short, we have the following diagram of abelian groups\footnote{ I have debated whether to write $1$ or $0$ on the left, but chose $1$ to honor the multiplicative nature of $ {\rm InvRat}(X) $.}
\begin{equation}\label{CaWe}
\begin{array}{c}
1 \to {\rm InvRat}(X) \to {\rm Rat}(X)^* \to  {\rm CartierDiv}(X) \to {\rm CaCl}(X)\to  0\\
      \Arrowvert  \hskip 60pt  \Arrowvert     \hskip 70pt  \downarrow   \hskip 60pt   \downarrow  \\
1  \to  {\rm InvRat}(X)  \to  {\rm Rat}(X)^*  \to \hskip 7pt {\rm WeilDiv}(X)\hskip 6pt \to \hskip 7pt{\rm Cl}(X) \hskip 7pt\to 0
\end{array}
\end{equation}
where the maps $ {\rm CartierDiv}(X)\to {\rm WeilDiv}(X)$ and ${\rm CaCl}(X)\to {\rm Cl(X)}$ are isomorphisms for smooth $X$. Here the rows are exact sequences, the notation  ${\rm InvRat}(X)$ stands for invertible functions on $X$ (equal to $\C^*$ for projective $X$), with the operation of multiplication, and $ {\rm Rat}(X)^* $ means nonzero rational functions.

\smallskip
\begin{rem}
We can similarly talk about Weil and Cartier divisors on complex manifold, with Zariski topology replaced by the usual one.
\end{rem}

\smallskip
When $X$ is not smooth, but has some singularities, the concepts of Weil and Cartier divisors are no longer equivalent. 
Both groups are well-defined for varieties $X$ with so-called normal 
singularities,\footnote{ Mathematicians clearly have some unresolved psychological issues, hence the word ``normal" is used in far too many unrelated settings: normal vector, normal distribution, normal subgroup, normal linear operator, etc.}  which means that $X$ is built out of integrally closed integral domains (see \cite{AM}). However, when $X$ is only normal, but not smooth, the maps $ {\rm CartierDiv}(X)\to {\rm WeilDiv}(X)$ and ${\rm CaCl}(X)\to {\rm Cl(X)}$ are only injective. Surjectivity fails if there are subvarieties of codimension one which are {\blue not given locally by a single equation}. 

\smallskip
{\bf Exercise 1.} Find the Weil divisor ${\rm div} f$ for the rational function on $\C^2$
$$f(x_1,x_2)=\frac {x_1^3x_2}{x_1^2+x_2^3-x_2}.$$
\emph{Hint:} It is a very simple problem, do not overthink it.

{\bf Exercise 2.} Find a Cartier divisor on $\C\P^1$ so that the corresponding Weil divisor is $5[0]-[\ii]$. \emph{Hint:} Use the cover of $\C\P^1$ by $U_0$ and $U_1$.

{\bf Exercise 3.} Consider the quotient of the ring of formal power series in three variables
$$R=\C[[x_1,x_2,x_3]]/(x_1x_2-x_3^2).$$
Show that the ideal $I=(x_1,x_3)\subset R$ is not principal.\footnote{ This is an infinitesimal version of the statement that the Weil divisor on the surface $x_1x_2=x_3^2$ given by $[x_1=x_3=0]$ can not be locally given as ${\rm div} f$.}
\emph{Hint:} Look at the dimension of $I/(x_1,x_2,x_3)I$.

\section{Cartier divisors as meromorphic sections of line bundles.}\label{sec.CD}
In this section we will connect the concepts of Cartier and Weil divisors with the previously considered concept of line bundles. 
The main idea is that one can interpret every Cartier divisor as a meromorphic\footnote{ or rational, in the algebraic category} section of an appropriate line bundle, up to isomorphism. For reasons that are not very clear to me, it is difficult to find this viewpoint in the literature without the language of sheaves. I am choosing to avoid the language of sheaves, and to muddle through anyway. 

\smallskip
Let $X$ be a smooth complex manifold. Recall from Section \ref{sec.vb} that a line bundle $\pi:L\to X$ is a holomorphic map from 
a smooth complex manifold $L$, with the fibers $\pi^{-1}(x)$ given a structure of a one-dimensional complex vector space. The simplest example is that of a trivial line bundle $\C\times X\to X$ and all line bundles are glued from $\C\times U\to U$ for some open subsets $U\subseteq X$.

\smallskip
Recall that holomorphic sections of $\pi:L\to X$ are holomorphic maps $s:X\to L$ such that $\pi\circ s = {\rm id}_X$. If $L = \C\times X\to X$ is a trivial line bundle, then holomorphic sections of it are in bijection with holomorphic functions $X\to \C$.
Similarly, we can talk about {\blue meromorphic sections} of line bundles. Indeed, locally every line bundle is (non-canonically) isomorphic to $\C\times U\to U$ and we would be looking at meromorphic functions on $U$ (ratios of two holomorphic functions).
If we change to another identification of $\pi^{-1}U$ with $\C\times U$, we simply multiply by an invertible holomorphic function, so meromorphicity of sections  is well defined. 

\smallskip
The analog of meromorphic sections in the case of smooth complex algebraic varieties is that of rational sections. It is basically the same thing, but now everything is algebraic. The key observation is that to
any nonzero\footnote{ i.e. not identically zero} rational section $s$ of a line bundle $L\to X$ 
one can naturally associate a Cartier divisor as follows.

\smallskip
\begin{defn}\label{mertoCa}
Let $s$ be a nonzero rational section of a line bundle $\pi:L\to X$. Pick an open cover $X=\bigcup_{\alpha\in I}U_\alpha$ such that $L$ is trivial over each $U_{\alpha}$ and choose one isomorphism $\pi^{-1}U_{\alpha}\cong \C\times U_\alpha$ of line bundles over each $U_\alpha$. Under this isomorphism, $s$ becomes a rational function $f_{\alpha}$ and we can see that 
$(U_\alpha,f_\alpha)$ form a Cartier divisor (see Exercise 1.)
\end{defn}

\smallskip
\begin{rem}
Observe that if we multiply a rational section $s$ of a line bundle $L\to X$ by an invertible rational function on $X$ (which we can do because fibers of $L\to X$ have a vector space structure), the corresponding Cartier divisor from Definition \ref{mertoCa} is unchanged. One can think about it as a particular case of the statement that if an isomorphism of vector bundles $L_1$ and $L_2$ sends a rational section $s_1$ of $L_1$ into a rational section $s_2$ of $L_2$, then the corresponding Cartier divisors are the same. Therefore, we see that rational sections of line bundles, up to isomorphism, give Cartier divisors.
\end{rem}

\smallskip
\begin{prop}
If we have two line bundles $L_1\to X$ and $L_2\to X$, with two nonzero  rational sections $s_1$ and $s_2$, then we can define the tensor product $s_1\otimes s_2$ as a  rational section of the tensor product line bundle $L_1\otimes L_2$. This gives an abelian group structure on the set of line bundles with nonzero  rational sections up to isomorphism, and the map of Definition \ref{mertoCa} is a group homomorphism.
\end{prop}

\begin{proof}
The identity of the group structure is given by the trivial line bundle $\C\times X\to X$ with the section $x\mapsto (1,x)$, and it is easy to check group properties. When checking that the map is a homomorphism, we can work on a cover $X=\bigcup_\alpha U_\alpha$ for which both $L_1$ and $L_2$ have trivializations. Then tensor product of sections gives rise to the product of the corresponding functions $f_{\alpha,1}$ and $f_{\alpha,2}$, which corresponds to the product of the Cartier divisors.
\end{proof}

\smallskip
We can in fact reverse the construction of Definition \ref{mertoCa}. To every Cartier divisor we will associate a line bundle, with a  rational section, unique up to an isomorphism. This will provide Cartier divisors with a nice geometric meaning.

\smallskip
\begin{defn}\label{Catomer}
Suppose we are given a Cartier divisor $(U_\alpha,f_\alpha)_{\alpha\in I}$. Take a copy of $\C\times U_\alpha$ for each $\alpha$ and glue them together by the following identification. We consider the equivalence relation of the set $\bigsqcup_\alpha (\C\times  U_\alpha)$ with  
$$
(\C\times U_\alpha)\ni(t,x)\sim\left( \frac {f_\beta(x)}{f_{\alpha}(x)} \,t, x\right)\in (\C\times U_\beta)
$$
for all $x\in U_\alpha\cap U_\beta$ and all $t\in \C$. It is easy to see that this is an equivalence relation. The set of equivalence classes $L$ then has a structure of a complex manifold (or algebraic variety), with a natural projection to $X$, and it has a line bundle structure.
Also observe that since 
$$
(f_\alpha(x),x)\sim(f_\beta (x),x)
$$
we can use $f_\alpha$ to define a  rational section of $L\to X$. 
\end{defn}

Note that if two Cartier divisors are equivalent, then the resulting line bundles are isomorphic, with the isomorphism identifying the natural  rational sections. It is also easy to see that the constructions of Definitions \ref{mertoCa} and \ref{Catomer} are inverses of each other.

\smallskip
\begin{rem}
For a  rational function $f$ on $X$, the principal Cartier divisor ${\rm div }f$ corresponds to the trivial line bundle $\C\times X\to X$ with the  rational section identified with $f$, i.e.
$$
s(x) = (f(x),x).
$$
\end{rem}

\begin{prop}
For a smooth complex manifold  $X$ the class group ${\rm Cl}(X)$ is isomorphic to the group ${\rm Pic}(X)$ of line bundles on $X$.
\end{prop}

\begin{proof}
Two different nonzero  rational sections $s_1$ and $s_2$ of the same line bundle $L\to X$ differ by a global  rational function $g$ on $X$, i.e. $s_2(x) = g(x)s_1(x)$. Therefore, forgetting the section $s$ and just remembering the line bundle gives a map
$$
{\rm Pic}(X)\to {\rm CaCl(X)}
$$
from the group of isomorphism classes of line bundles to the group of Cartier divisor classes on $X$, which is easily seen to be an isomorphism. Then we use the identification of Carter and Weil divisor class groups from the previous section.
\end{proof}

\smallskip
{\blue \hrule}
We will now illustrate this construction with an example. Consider the line bundle $\mathcal O(1)\to \C\P^2$ where we use $(x_0:x_1:x_2)$ to denote the homogeneous coordinates on $\C\P^2$ and a holomorphic section $x_0$ of $\mathcal O(1)$. Recall the standard cover $\C\P^2 = U_0\cup U_1\cup U_2$ where  $U_i=\{x_i\neq 0\}$. Each of the $U_i$ is isomorphic to $\C^2$, with coordinates 
$$
(\frac {x_1}{x_0},\frac {x_2}{x_0}),~(\frac {x_0}{x_1},\frac {x_2}{x_1}),~(\frac {x_0}{x_2},\frac {x_1}{x_2})
$$
respectively. Recall that $\mathcal O(1)$ is defined as the set of linear functions on lines $\lambda (x_0:x_1:x_2)$. For the points on $U_0$, we can identify the fiber of  $\mathcal O(1)$ over them  with $\C$ by looking at the values at the standard point $(1:\frac {x_1}{x_0}:\frac {x_2}{x_0})$. This means that the first coordinate (=section $x_0$) gives $f_0=1$. For the points on $U_1$, the standard point is given by $(\frac {x_0}{x_1}:1:\frac {x_2}{x_1})$, so the section gives $f_1=\frac {x_0}{x_1}$. We similarly get $f_2=\frac {x_0}{x_2}$.
To summarize, the Cartier divisor that corresponds to the section $x_0$ is therefore
$$
\{(U_0,1),(U_1,\frac {x_0}{x_1}),(U_2,\frac {x_0}{x_2})\}.
$$
We can easily verify that the above description fits the definition of Cartier divisor. For example, the ratio $\frac {f_1}{f_2} = \frac {x_2}{x_1}$ is an invertible holomorphic function on $U_1\cap U_2$. The corresponding Weil divisor is built from the zeros and poles of $f_i$ and is the line $[\{x_0=0\}]$ on $\C\P^2$, with coefficient $1$.

\smallskip
{\blue \hrule}
We will now compute the {\blue canonical} line bundle $K$ of $\C\P^2$, which we constructed in Definition \ref{canonical}.  Rational sections of $K$ are  rational differential $2$-forms, so we can pick one of them and compute the corresponding Cartier divisor. More precisely, a form that locally looks like $f(z_1,z_2) \,dz_1\wedge dz_2$ will correspond to the function $f(z_1,z_2)$, and these functions together will give us a Cartier divisor.
Let us start with the form 
$$
w = d\left(\frac {x_1}{x_0}\right) \wedge d\left(\frac {x_2}{x_0}\right).
$$
It is precisely $dz_1\wedge dz_2$ for the coordinates on $U_0$, so the corresponding function is $f_0=1$. We now wish to compute it on $U_1$, where the local coordinates are $\frac {x_0}{x_1}$ and $\frac {x_2}{x_1}$. We compute
\begin{align*}
w &= d\left(\frac {x_0}{x_1}\right)^{-1} \wedge  d\left(\frac  {x_2}{x_1} (\frac {x_0}{x_1})^{-1}\right)
\\
&=-\left( \frac {x_0}{x_1}\right)^{-2}  d\left(\frac {x_0}{x_1}\right) 
\wedge \left(\left(\frac {x_0}{x_1}\right)^{-1}d\left(\frac  {x_2}{x_1}\right) - \left(\frac  {x_2}{x_1}\right) \left( \frac {x_0}{x_1}\right)^{-2}  d\left(\frac {x_0}{x_1}\right) \right)
\\
&
=-\left( \frac {x_0}{x_1}\right)^{-3}  d\left(\frac {x_0}{x_1}\right) \wedge d\left(\frac  {x_2}{x_1}\right).
\end{align*}
This gives us $f_1=\left( \frac {x_0}{x_1}\right)^{-3}$ (signs, or more generally invertible functions do not matter) and similarly $f_2=\left(\frac {x_0}{x_2}\right)^{-3}$. Then the corresponding Weil divisor is $-3[\{x_0=0\}]$, and therefore the canonical line bundle over $\C\P^2$ is isomorphic to $\mathcal O(-3)$. 

\smallskip
We finish this section with an extremely useful result, known as the {\blue adjunction formula}.
\begin{prop}\label{adjform}
Let $X$ be a smooth complex algebraic variety and let $D\subset X$ be a smooth subvariety of codimension one. 
Let $K_X$ be the canonical line bundle on $X$ and let $\mathcal O(D)$ be the line bundle on $X$ that corresponds to the 
Weil divisor $[D]$. Then the canonical line bundle $K_D$ is isomorphic to
$$
\mu^*(K_X\otimes \mathcal O(D))
$$
where $\mu^*$ is the pullback of line bundles via the embedding map $\mu:D\to X$.
\end{prop}

\begin{proof}
We will not present the actual proof,\footnote{ Shocking, I know.}
but rather the main idea behind it. Holomorphic sections of $K_X\otimes \mathcal O(D)$ correspond to meromorphic top differential forms on $X$ which have pole of order at most $1$ along $D$ and no other poles. Then taking a residue at $D$ gives top differential forms on $D$.
\end{proof}

\smallskip
{\bf Exercise 1.}  Verify that Definition \ref{mertoCa} defines a Cartier divisor. Further verify that different choices of isomorphisms $\pi^{-1}(U_\alpha)\cong \C\times U_\alpha$ define the same Cartier divisor. 

{\bf Exercise 2.} Verify that the relation $\sim$ in Definition \ref{Catomer} is indeed an equivalence relation.

{\bf Exercise 3.}  Check that the canonical line bundle of $\C\P^n$ is isomorphic to $\mathcal O(-n-1)$, for any $n\geq 1$.

\section{Cubic curves in $\C\P^2$. Group law.}\label{cubicP2}
We are switching gears now, as we will do a lot in these notes, to go on a long detour into the wonderful world of elliptic curves.

\smallskip
Recall that an irreducible curve in $\C\P^2$ is the zero locus of an irreducible homogeneous polynomial $f(x_0,x_1,x_2)$ in the homogeneous coordinates $x_0,x_1,x_2$. 
If the degree of $f$ is equal to $1$, then $\{f=0\}=\{a_0x_0+a_1x_1+a_2x_2=0\}$ is a line. After a coordinate change, we may assume that it is $\{x_2=0\}$, so it is isomorphic to $\C\P^1$.
If the degree of $f$ is equal to $2$, then $f$ is a nonzero quadratic form in three complex variables. If the rank is $1$ or $2$, then the corresponding polynomial $x_0^2$ or $x_0^2+x_1^2$ is not irreducible. If the rank is three, then after a change of basis we get $\{x_0x_2-x_1^2=0\}$, which is isomorphic to $\C\P^1$ in its Veronese embedding, which we considered in Section \ref{sec.maps}.

\smallskip
The next case is ${\rm deg}\,f =3$, which is what we will focus on now. A general cubic curve is given by
$$
\sum_{i+j+k=3,\,i,j,k\geq 0}c_{ijk} \,x_0^ix_1^jx_2^k=0
$$
and is thus determined by its ten coefficients $c_{ijk}$.

\smallskip
\begin{prop} Generic choice of coefficients $c_{ijk}$ gives a smooth Riemann surface.
\end{prop}

\begin{proof}
Let us try to understand what it means for $\{f=0\}$ to be singular at $(1:0:0)$. This point lies in the open set $U_0$ and we can dehomogenize $f=0$ to 
\begin{equation}\label{cubU0}
\sum_{j+k\leq 3,\,j,k\geq 0}c_{(3-j-k)jk} \,\left(\frac {x_1}{x_0}\right)^j \left(\frac {x_2}{x_0}\right)^k = 0
\end{equation}
in $U_0$. In order for the curve \eqref{cubU0} to pass through $(1:0:0)$, we need to have no constant term, i.e. we need $c_{300}=0$. Moreover, if any of the linear terms is nonzero, we get smoothness at $(1:0:0)$ by the inverse function theorem. This means that we need $c_{210}=c_{201}=c_{300}=0$. Thus the subspace of the coefficient space $\C^{10}$ for which  the curve is singular at $(1:0:0)$ is of complex codimension three. 

As we vary the point $p\in \C\P^2$, we get a family of $7$-dimensional subspaces of  $\C^{10}$, and their union is then at most $9$-dimensional (everything is algebraic here, so dimension works like it should). Picking $c_{ijk}$ generically simply means staying away from this ``bad" locus.
\end{proof}

\smallskip
\begin{rem}
If a complex curve is smooth, then it is necessarily irreducible by Bezout's theorem. Indeed, a union of two curves would have intersection points, and will be singular there.
\end{rem}

\smallskip
\begin{rem}
All smooth cubic curves are diffeomorphic {\blue as real manifolds}. To see that, observe that the ``bad" locus in the coefficient space is of real codimension at least two, so the complement is connected. Thus any two collections of coefficients that give a smooth cubic can be connected by a path so that all of the intermediate Riemann surfaces are smooth. Locally, small changes in coefficients clearly don't affect the topology. 
\end{rem}

\smallskip
\begin{rem}
There is nothing special about the degree $3$ here. In fact, everything we discussed so far works for hypersurfaces of any degree in any $\C\P^n$. We leave the parameter count as Exercise 1. We also remark that it further generalizes to the statement known as the Bertini's theorem -- given a smooth subvariety of $\C\P^n$, a generic hypersurface cut of it is also smooth.
\end{rem}

\smallskip
{\blue \hrule}
So what is the topology of a smooth complex cubic curve? This is a  Riemann surface, which means that it is an orientable real surface. It is also compact, since $\C\P^2$ is compact. Any such surface is homeomorphic to a sphere with $g$ handles, where $g$ is called the {\blue genus} of the surface. So we can rephrase the question as asking what is the genus of a smooth cubic curve in $\C\P^2$. 

\smallskip
We will be using homogeneous coordinates $(x:y:z)$ on $\C\P^2$, to match some traditional notations. 
We pick\footnote{ Every smooth cubic can be written in this form after a coordinate change, but we do not need this fact here.} a cubic curve $E$ of the form
\begin{equation}\label{Esp}
y^2 z = (x- \alpha_1 z) (x-\alpha_2z) (x-\alpha_3z)
\end{equation}
where $\alpha_i$ are distinct complex numbers.
We can try to visualize $E$ by sketching the corresponding real curve in the case of $\alpha_i$ being real numbers, in the open subset $z\neq 0$. 

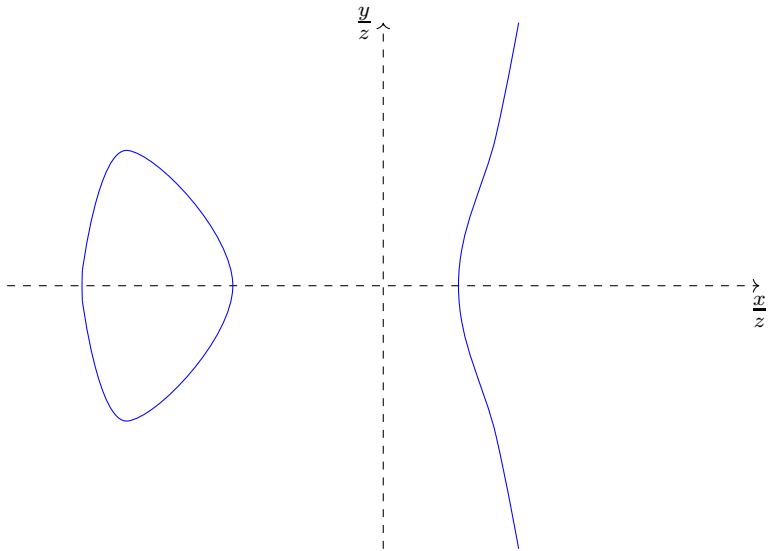
\begin{tikzpicture}
\draw[dashed, ->] (0,0) -- (10,0) node [anchor=north]{$\frac xz$} ;
\draw[dashed, ->] (5,-3.5) -- (5,3.5) node [anchor=east]{$\frac yz$};
\draw [blue] plot[semithick, smooth, tension=.7] coordinates {(1.,0.2)(1.6,1.8) (3,0) (1.6,-1.8)(1.,-0.2)};
\draw  [blue] plot[semithick, smooth, tension=.7] coordinates {(1.,0.2)(.995,0)(1.,-0.2)};
\draw  [blue] plot[semithick, smooth, tension=.7] coordinates {(6.8,3.5)(6.5,2)(6,0)(6.5,-2)(6.8,-3.5)};
\end{tikzpicture}

\smallskip
Let us check that $E$ is smooth. Note that the restriction of the above equation to the line $z=0$ gives $0=x^3$, so $(0:1:0)$ is an inflection point of $E$. In the local coordinates $(\frac xy, \frac zy)$ at $(0:1:0)$ the curve $E$ is given by
$$
0 = \frac zy - {\rm ~higher~degree~terms}
$$
which shows that $E$ is smooth at $(0:1:0)$. 

\smallskip
For the points with $z\neq 0$, we might as well dehomogenize by setting $z=1$. Then singularity is governed by $f(x,y,1)=0$, $\frac {\partial }{\partial x} f(x,y,1) = 0$ and $\frac {\partial }{\partial y} f(x,y,1) = 0$ which gives
$$
\left\{
\begin{array}{ll}
y^2-(x-\alpha_1)(x-\alpha_2)(x-\alpha_3) = 0,&\\
-(x-\alpha_2)(x-\alpha_3)-(x-\alpha_1)(x-\alpha_2)-(x-\alpha_1)(x-\alpha_2)=0,&\\
2y=0.&
\end{array}
\right.
$$
We then see that $y=0$, so by the first equation $x=\alpha_i$ for some $i$ which then does not fit the second equation since $\alpha_i$ are all different.

\smallskip
We can think of our Riemann surface $E$ given by \eqref{Esp} as 
$$
\frac yz = \sqrt{
\left(\frac xz-\alpha_1\right)\left(\frac xz-\alpha_2\right)\left(\frac xz-\alpha_3\right)
}
$$
which we view as a double cover of $\C\P^1$ with coordinates $(x:z)$. 
This double cover is ramified at $\alpha_1$, $\alpha_2$, $\alpha_3$ and $\infty$, which means that a small loop around these points leads to an opposite branch of the square root. Thus, if we make two cuts, one between $\alpha_1$ and $\alpha_2$ and another between
$\alpha_3$ and $\infty$, we will have two disjoint branches of the square root on the complement. 
We visualize the two branches and the rules of going from one to another as you cross the cut as follows. 

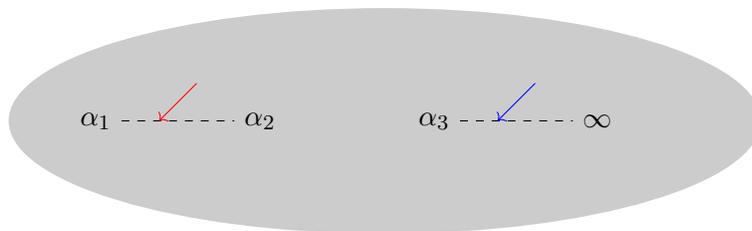
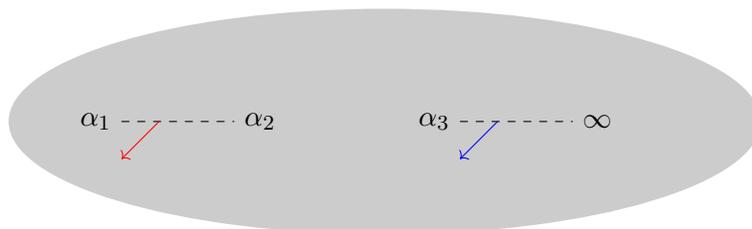
\begin{tikzpicture}
\fill[gray!40!white] (6.5, 5) ellipse (5 and 1.5); 
\draw[dashed](3,5)--(4.5,5) node[anchor=west]{$\alpha_2$};
\draw(3.,5) node[anchor=east]{$\alpha_1$};
\draw[dashed](7.5,5)--(9.,5) node[anchor=west]{$\infty$};
\draw(7.5,5) node[anchor=east]{$\alpha_3$};
\fill[gray!40!white] (6.5, 1) ellipse (5 and 1.5); 
\draw[dashed](3,1)--(4.5,1) node[anchor=west]{$\alpha_2$};
\draw(3.,1) node[anchor=east]{$\alpha_1$};
\draw[dashed](7.5,1)--(9.,1) node[anchor=west]{$\infty$};
\draw(7.5,1) node[anchor=east]{$\alpha_3$};
\draw [ red,->] (4,5.5)--(3.5,5);
\draw [ red,->] (3.5,1)--(3,0.5);
\draw [ blue,->] (8.5,5.5)--(8,5);
\draw [ blue,->] (8,1)--(7.5,0.5);
\end{tikzpicture}

\noindent
We then flip the top part and extend the shores of the cuts.

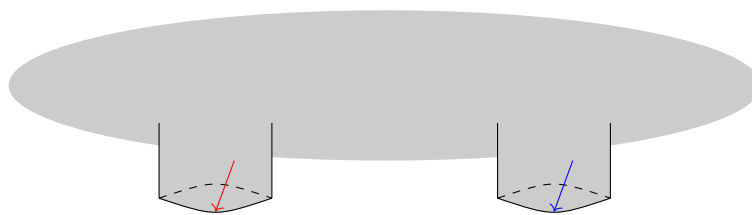
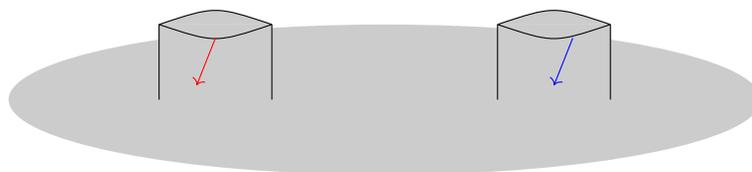
\begin{tikzpicture}
\fill[gray!40!white] (6.5, 5.5) ellipse (5 and 1.0); 
\fill[gray!40!white] (3.5,5) rectangle(5,4.0);
\filldraw[gray!40!white] (3.5,4.0)..controls(4.25,4.25)..(5,4.0);
\filldraw[gray!40!white](3.5,4.0)..controls(4.25,3.75)..(5,4.0);
\draw[dashed] (3.5,4.0)..controls(4.25,4.25)..(5,4.0);
\draw(3.5,4.0)..controls(4.25,3.75)..(5,4.0);
\draw[thin](3.5,5)--(3.5,4.0);
\draw[thin](5,5)--(5,4.0);
\fill[gray!40!white] (8,5) rectangle(9.5,4.0);
\filldraw[gray!40!white] (8,4.0)..controls(8.75,4.25)..(9.5,4.0);
\filldraw[gray!40!white](8,4.0)..controls(8.75,3.75)..(9.5,4.0);
\draw[dashed] (8,4.0)..controls(8.75,4.25)..(9.5,4.0);
\draw(8,4.0)..controls(8.75,3.75)..(9.5,4.0);
\draw[thin](8,5)--(8,4.0);
\draw[thin](9.5,5)--(9.5,4.0);
\fill[gray!40!white] (6.5, 1) ellipse (5 and 1.0); 
\fill[gray!40!white] (3.5,1.5) rectangle(5,2);
\filldraw[gray!40!white] (3.5,2)..controls(4.25,2.25)..(5,2);
\filldraw[gray!40!white](3.5,2)..controls(4.25,1.75)..(5,2);
\draw (3.5,2)..controls(4.25,2.25)..(5,2.0);
\draw(3.5,2.0)..controls(4.25,1.75)..(5,2.0);
\draw[thin](3.5,1)--(3.5,2.0);
\draw[thin](5,1)--(5,2.0);
\fill[gray!40!white] (8,1) rectangle(9.5,2.0);
\filldraw[gray!40!white] (8,2)..controls(8.75,2.25)..(9.5,2);
\filldraw[gray!40!white](8,2)..controls(8.75,1.75)..(9.5,2);
\draw (8,2)..controls(8.75,2.25)..(9.5,2);
\draw(8,2)..controls(8.75,1.75)..(9.5,2);
\draw[thin](8,1)--(8,2);
\draw[thin](9.5,1)--(9.5,2);
\draw [ red,->] (4.5,4.5)--(4.25,3.82);
\draw [ blue,->] (9,4.5)--(8.75,3.82);
\draw [ red,->] (4.25,1.82)--(4.0,1.19);
\draw [ blue,->] (9,1.82)--(8.75,1.19);
\end{tikzpicture}

\noindent
After we glue the tentacles together,
we get the real torus, 
a.k.a. sphere with one handle. So the genus of the cubic curve is $g=1$.


\smallskip
{\blue \hrule}
What's fascinating about the cubic curves is that they come with a group structure! More precisely, let $O$ be a point on $E$. It is common to pick $O$ to be an inflection point of $E$, but then one would have to prove that such point exists, and I can not be bothered to do so. For points $A,B\in E$ consider the line $AB$ (this would be the unique tangent line to $E$ at $A$ if $A=B$). By Bezout's theorem\footnote{ This is an overkill here, because the restriction of a cubic polynomial to a line will have three roots.} this line intersects $E$ at three points, counted with multiplicities. Two of these intersection points are $A$ and $B$, and we call the third one $X$. Then we similarly consider the line $OX$ and denote the third intersection point with $E$ by $Y$. We declare 
$$
Y=A+B
$$
for the newly defined addition operation $+$. We claim that it gives $E$ a group structure, which is of course abelian by construction, since lines $AB$ and $BA$ are the same.

\smallskip
We will {\red almost} prove associativity, which is rather miraculous. Suppose we want to prove 
$(A+B)+C=A+(B+C)$.
It is easy to see that the operation is continuous, so it suffices to assume that $A$, $B$ and $C$ are general points on the curve and then just take limits. In particular, we never have to worry about tangent lines. We denote  $Y= A+B$ and $W = B+C$. Then the following lines ``compute" $L=(A+B)+C$ and
$R=A+(B+C)$. 
\begin{align*}
ABX,~XOY,~YCZ,~OZL,\\
BCW,~OWV,~AVU,~OUR.
\end{align*}
Clearly, $L=R$ is equivalent to $Z=U$. 

\smallskip
Consider homogeneous cubic polynomials in $x,y,z$ which vanish on the eight points
\begin{equation}\label{8p}
A,B,C,X,Y,W,V,O.
\end{equation}
Since this vanishing is $8$ linear conditions on the space $\C^{10}$ of the polynomials, it is reasonable to expect that these conditions are linearly independent and cut out a $2$-dimensional space of polynomials. {\blue This is why this is only an ``almost" proof -- it takes a fair bit of annoying arguments to establish this linear independence.} Now observe that the three cubic curves 
$$
E, ~ABX\cup OWV\cup YCZ,~BCW\cup OXY\cup AVU
$$
pass through all eight points \eqref{8p}. This means that there is a linear relation among the corresponding polynomials. Since these curves are pairwise distinct, this linear relation involves all three polynomials. Therefore, if two of these curves pass through a point, so does the third one. In other words, the nine intersection points of $E$ with $ABX\cup OWV\cup YCZ$ coincide with the nine intersection points of $E$ with $BCW\cup OXY\cup AVU$, and this implies $Z=U$.

\smallskip
It is easy to check (Exercise 3) that $O$ is the identity of the group law. When $O$ is an inflection point, then the inverse of a point $A\in E$ is given by looking at the third intersection point of the line $A O$ with $E$. More generally, the inverse is found as follows. Consider the tangent line to $E$ at $O$. It intersects $E$ at $O$ twice and at some other point $O'$ (which is $O$ in the inflection case). Then the inverse of $A$ is the third point of intersection of $AO'$ with $E$, see Exercise 3.

{\bf Exercise 1.} Extend the count of the codimension of the bad locus to arbitrary hypersurfaces of degree $d$ in $\C\P^n$ by again figuring out what it takes to be singular at $(1:0:\ldots:0)$.

{\bf Exercise 2.} For the curve $E$ given by \eqref{Esp} and $O=(0:1:0)$ compute explicitly the sum $A+B$ for $A=(x_1:y_1:1)$ and $B=(x_2:y_2:1)$ (you can assume $x_1\neq x_2$). Verify that the operation is associative.
\emph{Hint:} write a parametric equation of the line $AB$, restrict the equation of $E$ to it and then use the Vieta's formula for the sum of three roots of a cubic polynomial.  Don't be shy about using software for algebraic manipulations.

{\bf Exercise 3.} Verify that $O$ is the identity of the group law on $E$. Verify that the inverse $B$ of $A\in E$ is the third intersection point of the line $AO'$ with $E$.

\section{Cubic curves as complex Lie groups. Lattice description. Elliptic functions.}
Recall that a cubic curve $E$ was shown to be a compact Riemann surface, equipped with an abelian group structure. It is not hard to see that the group operation is holomorphic, which makes $E$ into a complex compact Lie group. Universal cover of $E$ is then a one-dimensional simply-connected Lie group and is therefore isomorphic to $(\C,+)$.\footnote{ If you are not familiar with Lie groups, either complex or real, this paragraph probably doesn't make much sense to you. Hopefully, it will motivate you to learn these topics.
I learned this stuff from \cite{VO}, but there are many other viable sources.} Then $E$ must be a quotient of $\C$ by a discrete additive subgroup. Compactness of $E$ implies that it must be given, {\blue as a compact Riemann surface}, by
$$
E\cong \C/(\Z\gamma_1 + \Z\gamma_2)
$$
where $\gamma_1$ and $\gamma_2$ are nonzero complex numbers such that $\frac {\gamma_1}{\gamma_2}\notin \R$. We will denote the subgroup $\Z \gamma_1 + \Z\gamma_2$ by $L$ and call it a {\blue lattice}.

\smallskip
\begin{rem}
It is possible to scale $\C$ and pick the generators of the lattice so that $L=\Z + \Z\tau$ with ${\rm Im}(\tau)>0$. 
\end{rem}

We will now work on the complex manifold $\C/L$. We want to show that for any lattice $L$ we can embed $\C/L$ into $\C\P^2$, as a complex submanifold. Our treatment is very similar to that of \cite{Koblitz}.

\smallskip
We are interested in meromorphic functions on $\C/L$, which we can think of as meromorphic functions $f$ on $\C$ which are periodic with respect to $L$:
$$
f(z+l)= f(z),~{\rm for~all~}l\in L.
$$
These functions are called {\blue elliptic}. Before finding meaningful examples of such functions, we will investigate their properties.

\smallskip
Let $f$ be an elliptic function for the lattice $L$. Since poles of $f$ have no accumulation points, there must be a finite number of them modulo $L$. 

\smallskip
\begin{prop}\label{sumres}
For any elliptic function $f$ there holds
$$
\sum_{w\in \C/L} {\rm Res}_{z=w} f(z) = 0.
$$
\end{prop}

\begin{proof}
Consider the fundamental domain $D$ of $L$ given by a parallelogram (see picture below). It is shifted by some complex number $z_0$ so that it does not pass through the poles of $f$.

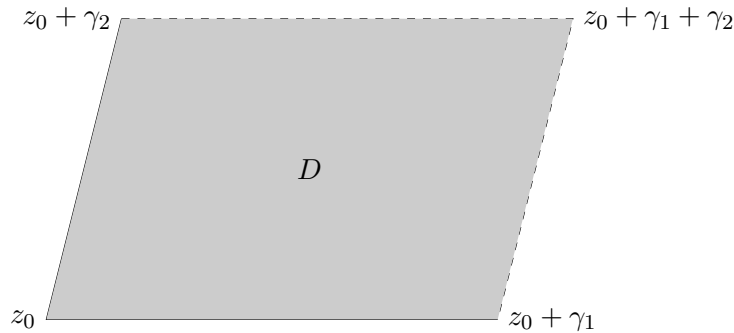
\begin{tikzpicture}
\draw (1,4)--(0,0)--(6,0);
\draw[dashed] (1,4)--(7,4)--(6,0);
\fill [gray!40!white](1,4)--(0,0)--(6,0)--(7,4)--(1,4);
\draw (0,0) node [anchor=east]{$z_0$};
\draw (6,0) node [anchor=west]{$z_0+\gamma_1$};
\draw (1,4) node [anchor=east]{$z_0+\gamma_2$};
\draw (7,4) node [anchor=west]{$z_0+\gamma_1+\gamma_2$};
\draw (3.5,2) node {$D$};
\end{tikzpicture}

\noindent
The sum of residues in $D$ is the counterclockwise integral $\int_{\partial D} f(z)\,dz$. By periodicity of $f$, the contributions of the opposite sides of $D$ cancel.
\end{proof}

\begin{prop}\label{zp}
For a nonzero elliptic function $f$, the number of zeros modulo $L$ equals the number of poles modulo $L$, counted with multiplicities.
\end{prop}

\begin{proof}
If $f(z)$ is a nonzero elliptic function, then so is $g(z) = \frac {f'(z)}{f(z)}$. It remains to observe that poles of $g$ occur at zeros and poles of $f$, and the residues of $g$ are equal to the order of the zero or pole (the latter with a negative sign). Then we apply Proposition \ref{sumres} to $g$ and get the desired result. 
\end{proof}

\begin{prop}\label{sumzerospoles}
Let $w_i$ be the zeros/poles of a nonzero elliptic function $f$, one for each coset modulo $L$. We denote by $m_i$ the multiplicity of the zero (positive) or pole (negative) of $f$ at $w_i$. Then
\begin{equation}\label{sum}
\sum_{i} m_i w_i \in L.
\end{equation}
\end{prop}

\begin{proof}
Before starting the proof, we remark that changing  $w_i$ to $w_i+l$ for $l\in L$ does not affect the conclusion. So we may again pick a fundamental parallelogram $D$ and write the sum in \eqref{sum} as
$$
\frac 1{2\pi{\rm i}} \oint_{\partial D} z\, \frac {f'(z)}{f(z)}\, dz.
$$
Because of the term $z$ which is not periodic, we will not quite have cancellation for the opposite sides of $D$. Rather, the sum 
of the integral over the top and the bottom segments will be
\begin{align*}
&
\frac 1{2\pi{\rm i}} \int_{z=z_0}^{z=z_0+\gamma_1} z \frac {f'(z)}{f(z)}\, dz -
\frac 1{2\pi{\rm i}} \int_{z=z_0+\gamma_2}^{z=z_0+\gamma_1+\gamma_2} z \frac {f'(z)}{f(z)}\, dz
\\
&=
\frac 1{2\pi{\rm i}} \int_{z=z_0}^{z=z_0+\gamma_1}\left( z \frac {f'(z)}{f(z)}-(z+\gamma_2) \frac {f'(z+\gamma_2)}{f(z+\gamma_2)}\right)\, dz 
=-\frac 1{2\pi{\rm i}} 
\int_{z=z_0}^{z=z_0+\gamma_1} \gamma_2 \frac {f'(z)}{f(z)}\, dz
\\
&=- \frac {\gamma_2}{2\pi{\rm i}} 
\int_{z=z_0}^{z=z_0+\gamma_1}\,d \log f(z) \in \frac { \gamma_2}{2\pi{\rm i}} (2\pi{\rm i} \,\Z) \in \Z \gamma_2 \subset L
\end{align*}
and similarly for the other pair.
\end{proof}

Having discovered these wonderful properties of elliptic functions,  it would be great to actually construct some non-trivial examples. An urban legend claims that once there was a PhD student who has been finding the most fascinating properties of some class of functions, introduced by their advisor. And just before the defense, they finally proved that all of these functions were identically zero. Are we going to meet the same fate? No!

\smallskip
If we have an elliptic function $f$ which is holomorphic, then it must be constant, by virtue of being a holomorphic function on a compact manifold (otherwise, the image of $f$ is compact, but it also has to be an open set).
Of course, constant functions are elliptic, but that's not very interesting.

\smallskip
So we are forced to allow some poles. If we allow only one pole of order $1$, up to lattice shifts, then by Proposition \ref{sumres}, the residue at this pole would be zero, so it would not be a pole at all! So we should either have at least two different poles of order one or one pole of order two. This makes it reasonable to look for an elliptic function $f$ on $\C$  that 
has a pole of order $2$ at $z=0$ and no other poles modulo $L$.

\smallskip
We know that the residue at $z=0$ must be $0$, so our first attempt at finding such $f$ is 
$$
f(z) = \frac 1{z^2}.
$$
This obviously fails, because $\frac 1{z^2}$ is not $L$-periodic.

\smallskip
Undeterred, we are going make it periodic. Consider
$$
f(z) = \sum_{l\in L} \frac 1{(z-l)^2}.
$$
Great, this looks periodic. But, we have a slight problem -- the series does not have absolute convergence. Indeed, the number of lattice points $l$ of size approximately $k$ is on the order of $c\, k$ for some constant $c$, so we have roughly $c\, k$ terms of size $\frac 1{k^2}$, and the harmonic series diverges (although rather slowly).

\smallskip
There is a nice trick that we can use to deal with it. Namely, for a fixed $z$ and  $l$ large enough, $\frac 1{(z-l)^2}$ is approximately $\frac 1{l^2}$. So we may reasonably hope that 
\begin{equation}\label{almostW}
f(z) = \sum_{l\in L} \Big(\frac 1{(z-l)^2}-\frac 1{l^2}\Big)
\end{equation}
will be an absolutely convergent series. And since we are just subtracting constants, we should still have the periodicity.

\smallskip
We are very close. The only problem is that we have a $\frac 1{0^2}$ term in the series, which prompts the following definition.

\smallskip
\begin{defn}The Weierstrass elliptic function\footnote{ We will usually drop the adjective ``elliptic", but it should not be confused with the continuous nowhere differentiable monstrosity also known as the Weierstrass function.} $\mathcal P(z)$ is defined by
\begin{equation}\label{Weier}
\mathcal P(z) = \frac 1{z^2} + \sum_{0\neq l\in L} \Big(\frac 1{(z-l)^2}-\frac 1{l^2}\Big).
\end{equation}
\end{defn}

\smallskip
We will now work to verify that $\mathcal P(z)$ is an elliptic function with pole of order two at $z\in L$ and no other poles. It is actually rather straightforward. We first claim that $\mathcal P(z)$ converges uniformly and absolutely on any compact set $A$ in $\C\setminus L$. Indeed, for $0\neq l\in L$  and $z\in A$ we have
$$
\Big\vert
\frac 1{(z-l)^2}-\frac 1{l^2} 
\Big\vert=
\Big\vert
\frac {z(2l-z)}{(z-l)^2l^2}  \Big\vert \leq  \Big\vert \frac {\rm const} {l^3}\Big\vert
$$ 
and since the number of $l\in L$ of size roughly $k$ is at most ${\rm const}\,k$ (see Exercise 1), we get absolute and uniform convergence. 

\smallskip
It is easy to see that the poles of $\mathcal P(z)$ are of order two and at $z\in L$ only. For example, all of the terms other than $\frac 1{z^2}$ together converge absolutely and uniformly in a small neighborhood of $z=0$, by the same argument as above. Thus, it remains to verify that $\mathcal P$ is $L$-periodic. This is not exactly surprising, given how we gradually constructed $\mathcal P$, but let's do it. Pick a nonzero lattice element $m$. Then for $z\not\in L$ we have
\begin{align*}
&\mathcal P(z+m)
 = \frac 1{(z+m)^2} + \sum_{0\neq l\in L} \Big( \frac 1{(z+m-l)^2 }-\frac 1{l^2}\Big)
 \\&
 =\frac 1{(z+m)^2} + \frac 1{z^2} - \frac 1{m^2}  + \sum_{l\neq 0,m} \Big( \frac 1{(z+m-l)^2 }-\frac 1{l^2}\Big)
 \\&
 =\frac 1{(z+m)^2} + \frac 1{z^2} - \frac 1{m^2}  + \sum_{l'\neq -m,0} \Big( \frac 1{(z-l')^2 }-\frac 1{(l'+m)^2}\Big)
 \\&
 =
 \frac 1{(z+m)^2} + \frac 1{z^2} - \frac 1{m^2}  + \sum_{l'\neq -m,0} \Big( \frac 1{(z-l')^2 }-\frac 1{(l')^2}\Big)
  \\&
  + \sum_{l'\neq -m,0} \Big( \frac 1{(l')^2 }-\frac 1{(l'+m)^2}\Big)
 =\mathcal P(z) 
 + \sum_{l'\neq -m,0} \Big( \frac 1{(l')^2 }-\frac 1{(l'+m)^2}\Big) = \mathcal P(z).
\end{align*}
In the last equality we used that the change of variables $l'\to -m-l'$ sends 
$S= \sum_{l'\neq -m,0} \Big( \frac 1{(l')^2 }-\frac 1{(l'+m)^2}\Big)$ to $(-S)$, which implies $S=0$.

We also observe that $\mathcal P(z)$ is an even function, i.e. 
$$\mathcal P(-z)=\mathcal P(z),$$
which follows from making a change of summation index $l\to (-l)$ in the definition of $\mathcal P$.

\smallskip
\begin{prop}\label{valP}
For every $a\not\in \frac 12 L$ the even elliptic function $\mathcal P(z) -\mathcal P(a)$ has zeros at $a \,{\rm mod} \,L$ and $-a\, {\rm mod} \,L$ with multiplicity one, and no other zeros. For $a\in \frac 12 L \setminus L$ the function $\mathcal P(z) -\mathcal P(a)$ has zeros of order two at  $a\,{\rm mod}\, L$ and no other zeros.
\end{prop}

\begin{proof}
The function $\mathcal P(z) -\mathcal P(a)$ has a pole of order two at $L$ and no other poles. Therefore, by Propositions \ref{zp} and \ref{sumzerospoles} it has two zeros (with multiplicity, up to $L$) which add up to $0\,{\rm mod} \,L$. One of these zeros is $a$, so the other is $(-a)$. The situation with $2a\in L$ is special, because then we have $a=-a\,{\rm mod}\,L$ and a double zero.
\end{proof}

\begin{thm}\label{evenelliptic}
Every even elliptic function $f(z)$ is a ratio of two polynomials in $\mathcal P(z)$ with constant coefficients.
\end{thm}

\begin{proof}
For a nonzero $f(z)$, its zeros and poles come in pairs $(a,-a)$ modulo $L$. Moreover, the multiplicity of zero/pole at $a$ with $2a\in L$ is even, see Exercise 2.
Then we consider 
$$
g(z) = f(z)  \prod_{2a\in L,a\not\in L} (\mathcal P(z) -\mathcal P(a))^{-\frac{m_a}2}
\prod_{(a,-a),2a\not\in L} (\mathcal P(z) -\mathcal P(a))^{-m_a}.
$$
By construction, $g(z)$ has no zeros or poles, except for $0\,{\rm mod}\,L$. Because the total number of zeros and poles is zero, this means that $g(z)$ has neither, and is therefore a nonzero constant. This implies that $f(z)$ is a ratio of polynomials in $\mathcal P(z)$. 
\end{proof}

{\bf Exercise 1.} Prove that for every lattice $L$ there exists a constant $c$ such that the number of $l\in L$ with 
$k\leq |l|< k+1$ is at most $c\, k$ for all $k\geq 1$. \emph{Hint:} Argue that the union of fundamental parallelograms of $L$ centered at such $l$ lies in a certain annulus and thus has smaller area.

{\bf Exercise 2.} Prove that for any even elliptic function $f$ and any $a\in \frac 12 L$ the order of pole or zero of $f$ at $a$ is even.
\emph{Hint:} Consider the Laurent power series of $f$ at $a$ and use $f(2a-z) = f(-z)=f(z)$.

{\bf Exercise 3.} Prove that every elliptic function can be uniquely written as a sum of an even and an odd elliptic functions.

\section{Weierstrass embedding of $\C/L$ into $\C\P^2$.}
Theorem \ref{evenelliptic} provides a good handle on even elliptic functions, but can we find an odd elliptic function? No problem at all. 
Since $\mathcal P(z)$ is even, its derivative $\mathcal P'(z)$ is odd. Since as we saw in Exercise 3 of the previous section, every elliptic function is a sum of an even and an odd one, we now have a pretty good understanding of all elliptic functions. Specifically, we see that every elliptic function can be written as 
$$
f(\mathcal P(z)) + \mathcal P'(z) g(\mathcal P(z))
$$
where $f$ and $g$ are rational functions of one variable.

\smallskip
Observe that since $\mathcal P'(z)^2$ is even, it can be written as a rational function in $\mathcal P(z)$. But we can be a lot more precise than that, it is a polynomial of degree $3$ in $\mathcal P(z)$. We will prove it by looking at the Laurent expansions at $z=0$.

\smallskip
We start with the Laurent power series expansion 
$$
\mathcal P(z) = \frac 1{z^2}+ 0 + a z^2 + bz^4 + \ldots
$$
where we can only have even powers because $\mathcal P(z)$ is even and the $z^0$ term is zero by Exercise 1. The constants $a$ and $b$ depend on the lattice $L$, but we will not need explicit formulas for them. We differentiate and square to get the following.
\begin{align*}
&\mathcal P'(z) = -\frac 2{z^3} + 2a z + 4bz^3 + \ldots
\\
&\mathcal P'(z)^2 = \frac 4{z^6} - \frac {8a}{z^2} - 16b + \ldots
\end{align*}
We also compute $\mathcal P(z)^3$ as
$$
\mathcal P(z)^3 = \frac 1{z^6} + \frac {3a}{z^2} + 3b + \ldots
$$
which gives
$$
\mathcal P'(z)^2-4\,\mathcal P(z)^3 = -\frac {20a}{z^2} - 28 b + \ldots
$$
and 
$$
\mathcal P'(z)^2-4\,\mathcal P(z)^3 + 20a\,\mathcal P(z) = -28b + \ldots.
$$
Now observe that $\mathcal P'(z)^2-4\,\mathcal P(z)^3 + 20a\,\mathcal P(z)$ is an elliptic function with no poles, so it must be a constant, which implies 
\begin{equation}\label{CLcubic}
\mathcal P'(z)^2 =4\,\mathcal P(z)^3 - 20a\,\mathcal P(z) - 28b.
\end{equation}
\smallskip
Doesn't it look suspiciously like the cubic curve \eqref{Esp} in $\C\P^2$? 

\smallskip
We will in fact show that the Riemann surface $\C/L$ is biholomorphic to a smooth cubic curve in $\C\P^2$.
\begin{defn}
Consider the map $\mu:\C/L\to \C\P^2$ given by
$$
\mu(z) = \left\{
\begin{array}{ll}
(\mathcal P(z):\mathcal P'(z):1),&{\rm if~} z\neq 0\,{\rm mod}\,L,\\
(0:1:0),&{\rm if~}z=0\,{\rm mod}\,L.
\end{array}
\right.
$$
\end{defn}

\begin{prop}
The map $\mu$ is holomorphic.
\end{prop}

\begin{proof}
Clearly, $\mu$ is well-defined, because $\mathcal P$ and $\mathcal P'$ are $L$-periodic. Holomorphicity outside of $z=0\,{\rm mod}\, L$ is immediate, since $\mathcal P$ has no poles there. To show holomorphicity at $z=0$ observe that near $z=0$ we can rewrite the top line of the definition of $\mu$ as 
\begin{equation}\label{near0}
\mu(z) = \left(\frac{\mathcal P(z)}{\mathcal P'(z)}:1:\frac 1{\mathcal P'(z)}\right)
\end{equation}
and the elliptic functions  $\frac{\mathcal P(z)}{\mathcal P'(z)}$ and $\frac 1{\mathcal P'(z)}$ are holomorphic near $z=0$ and have limit $0$ as $z\to 0$.
\end{proof}

\begin{thm}\label{mainCL}
The map $\mu$ is an embedding of complex manifolds. The image is a smooth cubic curve given by
$$
x_1^2 x_2 =4 x_0^3 - 20a \,x_0 x_2^2 - 28b\, x_2^3.
$$
\end{thm}

\begin{proof}
The first order of business is to prove that $\mu$ is injective. It is clear that $z=0\,{\rm mod}\,L$ and $z\neq 0\,{\rm mod}\,L$ map to different parts of $\C\P^2$, since the former has nonzero last coordinate in its image.
Suppose now that $\mu(z_1)=\mu(z_2)$ with nonzero $z_1\neq z_2\,{\rm mod}\,L$. This implies
$$
\mathcal P(z_1)=\mathcal P(z_2),~\mathcal P'(z_1)=\mathcal P'(z_2).
$$
Observe that by Proposition \ref{valP} we must have $z_2=-z_1 \,{\rm mod}\,L$ and $z_1$ can not be a $2$-torsion point, i.e. $2z_1\not\in L$. Then since $\mathcal P'$ is odd, we get $\mathcal P'(z_2)=-\mathcal P'(z_1)$ which means $\mathcal P'(z_1)=0$. However, one can easily show that $\mathcal P'(z_1)=0$ occurs only at the three nonzero $2$-torsion points of $\C/L$ (this is Exercise 2), which proves 
$\mu$ is injective.

\smallskip
This injectivity is not enough to prove that we have an embedding.\footnote{ The map $t\mapsto (t^2,t^3)$ from $\C$ to $\C^2$ is injective but is not an embedding of complex manifolds.} We also need to show that tangent vectors do not map to zero. For $z\not\in L$, this means to show that at least one of the derivatives of $\mathcal P$ and $\mathcal P'$ is nonzero at $z$. To see that, observe that $\mathcal P'(z) = 0$ implies $2z\in L$ by Exercise 2. Then $\mathcal P''(z)$ is nonzero, again by Exercise 2, because the 
multiplicity of the zero of $\mathcal P'(z)$ is one. Finally, the map $\mu$ is an embedding near $z=0$, because in \eqref{near0} the function
$$
\frac{\mathcal P(z)}{\mathcal P'(z)} = \frac {\frac 1{z^2}+  a z^2 + bz^4 + \ldots}{-\frac 2{z^3} + 2a z + 4bz^3 + \ldots}
= -\frac 12 z + \ldots
$$
has a nonzero derivative at $z=0$. 

\smallskip
We now know that the image is smooth. It also satisfies the above cubic equation because of \eqref{CLcubic}.
\end{proof}

\smallskip{\blue \hrule}
We will now compare the group law on $\C/L$ with that on the cubic curve of Theorem \ref{mainCL} where we use $O=\mu(0)=(0:1:0)$ to define the group structure geometrically. The idea is simple: if we have a line $a_0 x_0 + a_1 x_1 + a_2 x_2 = 0$ (for simplicity assume that $a_1\neq 0$) then the intersection points of it with $\mu(\C/L)$ are images under $\mu$ of the three roots of the function
$$
a_0 \mathcal P(z) + a_1 \mathcal P'(z) + a_2.
$$
Since the three poles of this function are all at $0\,{\rm mod}\,L$, the sum of these zero is $0\,{\rm mod}\,L$ by Proposition \ref{sumzerospoles}. Similarly, for $\alpha_1=0$, but $a_0\neq 0$, the line intersect $\mu(\C/L)$ at $\mu(0)$ and two other points that add up to zero. If we start with $A=\mu(z_1)$, $B=\mu(z_2)$, then the third intersection point of $AB$ with $\mu(\C/L)$ is $\mu(-z_1-z_2)$, and then the point $A+B$ corresponds to $\mu(z_1 +z_2)$, so the two group laws are matched under $\mu$.\footnote{ We are ignoring various issues of multiple roots, or one of $A$ or $B$ or $A+B$ being equal to $(0:1:0)$. These can be handled by continuity, or directly.}
 
\smallskip
{\blue \hrule}
Let us now interpret the map $\mu$ in terms of sections of line bundles. If we start with a Weil divisor $3[0\,{\rm mod}\,L]$, (a multiple of the codimension one subvariety which is the origin in $\C/L$), then it gives us a line bundle $\mathcal O(3[0\,{\rm mod}\,L])$ and a global holomorphic section $s$ of it (unique up to nonzero scalar), which has locally a zero of order three at $0\,{\rm mod}\,L$ under a trivialization. All other meromorphic sections of $\mathcal O(3[0\,{\rm mod}\,L])$  are of the form $s f$ for an elliptic function $f$; they are holomorphic if and only if $f$ has a pole of order at most three at $0\,{\rm mod}\,L$ and no other poles. Such $f$ are uniquely determined by the coefficients $a_{-3},a_{-2},a_{-1},a_0$ in their Laurent expansion $f(z)=\sum_{k\geq -3} a_kz^{k}$, and we also have $a_{-1}=0$, as this is the only residue of $f$. Thus $f$ lies in the linear span of  
$$
1,~\mathcal P(z), ~\mathcal P'(z)
$$
Then sections $s \mathcal P(z), s \mathcal P'(z), s$ give us the map $\mu$. See also Exercise 3 for what happens for the Weil divisor $2[0\,{\rm mod}\,L]$.

\smallskip
\begin{rem}
One should think of the Weierstrass embedding $\mu$ as providing $\C/L$, which is initially only a complex manifold, with the structure of 
a smooth algebraic variety.
\end{rem}

\smallskip
\begin{rem}
The term ``elliptic" comes from the fact that the ``inverse" of $\mu$ is obtained by integrating the form 
$$
dz = \frac  {d\mathcal P(z)}{\mathcal P'(z)} = \frac {dx}{\sqrt{4x^3-20a\,x-28b}}
$$ 
where $x=\frac {x_0}{x_2}$ and $y=\frac{x_1}{x_2}$. These types of the integrals, with square roots of cubic (or degree four) polynomials historically occurred in the problem of computing the arc length of ellipses.
\end{rem}

\smallskip

{\bf Exercise 1.} Use the definition of $\mathcal P(z)$ to prove that it has zero constant term in its Laurent series expansion at $z=0$.

{\bf Exercise 2.} Prove that $\mathcal P'(z)$ has zeros exactly at the three points $z\,{\rm mod}\,L$ such that $2z\in L$ but $z\not\in L$, each with multiplicity one. \emph{Hint:} Use that $\mathcal P'$ is odd to prove that it has zeros at these points. Use Proposition \ref{zp} to prove that there are no other zeros or higher multiplicities.

{\bf Exercise 3.} Prove that the space of global holomorphic sections of the line bundle $\mathcal O(2[0\,{\rm mod}\,L])$  is two-dimensional and that the resulting map to $\C\P^1$ is a $2:1$ map ramified at points $z$ with $2z= 0\,{\rm mod}\, L$.

\section{Jacobi theta function and divisors on elliptic curves. Elliptic curves over $\Q$.}
We have proved in Proposition \ref{sumzerospoles} that that if $f$ is an elliptic function for a lattice $L\subset \C$, then for zeros/poles $z_i\in \C/L$ with multiplicities $m_i$ there holds 
\begin{equation}\label{zpell}
\sum_{i}m_i = 0,~\sum_i m_i z_i =0\,{\rm mod}\,L.
\end{equation}
We will argue in this section that the converse is true, i.e. for a set of points $z_i\in \C/L$ and integers $m_i$ that satisfy \eqref{zpell} there is a unique up to scaling elliptic function $f$ with poles and zeros at $z_i$ with multiplicity $m_i$.

\smallskip
Uniqueness is clear because the ratio of two functions with the same multiplicities of zeros and poles is holomorphic on $\C/L$ and thus constant. Existence is harder to prove -- we will do it with the help of the  {\blue Jacobi theta function} which is a holomorphic function on $\C\times \mathcal H$, where $\mathcal H= \{\tau,{\rm~Im}\tau>0\}$ is the upper half plane.

\smallskip
\begin{defn}\label{theta}
We define a function of two variables $(z,\tau)\in \C\times \mathcal H$ by
$$
\theta(z,\tau) = \ee^{\pi \ii \tau /2} (2\sin \pi z) \prod_{l=1}^{\infty}(1-\ee^{2\pi\ii l\tau})
(1-\ee^{2\pi\ii (l\tau+z)})(1-\ee^{2\pi\ii (l\tau-z)}).
$$
\end{defn}

\smallskip
\begin{rem}
There is no universally agreed upon set of notations when it comes to the Jacobi theta functions. One can find other formulas in the literature. However, they are all related to $\theta$ that we define by some simple coordinate changes.
\end{rem}

It is easy to see that the infinite product in the definition of $\theta$ converges uniformly on compacts in $\C\times \mathcal H$ and thus defines a holomorphic function of two variables. The key to this is the equality $|\ee^{2\pi\ii(l\tau +z)}| = |\ee^{-2\pi\, {\rm  Im} z }||\ee^{-2\pi\, {\rm Im} \tau }|^l$. 
It is also easy to prove that 
\begin{equation}\label{proptheta}
\theta(-z,\tau) = -\theta(z,\tau),~
\theta(z+1,\tau) = -\theta(z,\tau),~
\theta(z+\tau,\tau) = -\ee^{-2\pi\ii z - \pi \ii \tau} \theta(z,\tau),
\end{equation}
see Exercise 1.

\smallskip
Because of the convergence of the product, zeros of $\theta(z,\tau)$ occur precisely where the factors vanish.
\begin{align*}
&\{\theta(z,\tau)=0\} = \{\sin \pi z = 0\} \cup \bigcup_{l>0} \{1=\ee^{2\pi\ii(l\tau \pm z)}\} 
\\
&= \Z \cup \bigcup_{l>0} \{z \in \mp l\tau+\Z\} =\Z+\Z\tau.
\end{align*}
We also see that the vanishing is of first order, so in particular if we fix $\tau$, the function $\theta(z) = \theta(z,\tau)$ has simple zeros at $L=\Z+\Z\tau$. This can be also proved using the transformation properties \eqref{proptheta}, see Exercise 2. 

\smallskip
\begin{prop}\label{findf}
Suppose that we have a finite set of points $z_k \in \C$  and integers $m_k$ such that 
$$
\sum_k{m_k}=0,~\sum_k{m_k z_k} = 0.
$$
Then
$f(z) = \prod_{k} \theta(z-z_k,\tau)^{m_k}$ is an elliptic function for $L=\Z+\Z\tau$.
\end{prop}

\begin{proof}
By \eqref{proptheta} and assumptions on $z_k$ and $m_k$,
\begin{align*}
&f(z+1) = f(z) \prod_k (-1)^{ m_k}  = f(z),\\
&f(z+\tau) = f(z) \prod_k \ee^{(\pi\ii - 2\pi\ii z + 2\pi\ii z_k -\pi\ii\tau)m_k}
=f(z).
\end{align*}
\end{proof}

\begin{thm}\label{mainzpell}
For any set of points  $z_k\in \C/L$ with integers $m_k$ that satisfy 
$$
\sum_k m_k = 0,~\sum_k m_k z_k = 0\,{\rm mod}\,L
$$
there exists a unique up to scaling meromorphic function on $\C/L$ whose Weil divisor ${\rm div} f$ is given by $\sum_k m_k z_k$.
\end{thm}

\begin{proof}
By scaling $\C$ and $L$ by a complex number, we may assume that $L=\Z+\Z\tau$.

We lift $z_k$ to $\C$. We would like to be able to do it so that $\sum_k m_k z_k = 0$, but a priori we can only get 
$\sum_k m_k z_k \in L$, if none of the $m_k$ is equal to $\pm 1$. However, we can also split some $m_k$ into $m_k-1$ and $1$ and do a separate lift for $1$ to get the sum to be exactly zero. Now the elliptic function $f$ from Proposition \ref{findf} defines a meromorphic function on $\C/L$ with prescribed zeros and poles. (Note that we did not assume in Proposition \ref{findf} that $z_k$ are pairwise distinct modulo $L$.)
\end{proof}

\begin{rem}
The Jacobi theta function $\theta(z,\tau)$ has a number of miraculous properties. For example, for any $\left(\begin{array}{cc}a&b\\
c&d\end{array}\right) \in {\rm SL}(2,\Z)$ there holds
$$
\theta \left(\frac z{c\tau+d},\frac {a\tau+b}{c\tau+d}\right) = \xi (c\tau+d)^{\frac 12} \ee^{\frac {\pi \ii c z^2}{c\tau+d}}\theta (z,\tau)
$$
where $\xi$ is some $8$-th root of $1$ which depends only on $a,b,c,d$.
Another great result is the Jacobi triple product identity which leads to
$$
\theta(z,\tau) = -\ii \sum_{n\in\Z} (-1)^{n} \ee^{\pi\ii((n+\frac 12)^2 \tau + (2n+1)z)}.
$$
Proofs of these results, while not terribly difficult, lie outside of the scope of these notes, see \cite{Chandrasekharan}.

It is also worth mentioning that for a fixed $\tau$ the second logarithmic derivative 
$$
\frac {d^2}{dz^2} \log \theta(z)
$$
is elliptic. By looking at the Laurent expansion at $z=0$, we see that it equals a constant minus the Weierstrass function.
\end{rem}

\smallskip
{\blue \hrule}
Now that we know that conditions \eqref{zpell} precisely describe the location and multiplicity of zeros and poles of meromorphic functions on $\C/L$, we can use it to find ${\rm Pic}(\C/L)$, which we know to equal the Weil class group ${\rm Cl}(\C/L)$.

\begin{thm}
The class group of $\C/L$ is isomorphic to $\C/L\oplus \Z$.
\end{thm}

\begin{proof}
The isomorphism is induced by the map ${\rm WeilDiv}(\C/L)\to \C/L\oplus \Z$ given by
$$
\sum_{k}m_k[z_k]\mapsto \left(\sum_k m_kz_k,\sum_k m_k\right).
$$
By Theorem \ref{mainzpell} the kernel is exactly the subgroup of principal divisors on $\C/L$.
\end{proof}

\smallskip
{\blue \hrule}
When we have a cubic curve $E$ whose coefficients are rational numbers, it is meaningful to ask about divisors on it that are defined over $\Q$ (it is equivalent to looking at the group of points on $E$ with rational coordinates). By a theorem of Mordell (1922) this group has finite rank, i.e. it is 
$$\Z^{n}\oplus{\rm Torsion}.$$
The torsion is well-understood, but the rank is harder. For example, it is unknown whether the rank $n$ is bounded from above. The current world record of a curve whose rank is at least $29$ is due to Elkies and Klagsbrun, achieved in 2024. The previous world record of $n\geq 28$ was achieved by Elkies in 2006, see \cite{Elkies}.

\smallskip
Another interesting topic is generalizations of $\C/L$ to higher dimensions. For example, we may want to consider the complex manifolds $\C^2/L$ where $L\cong \Z^4$ is a discrete subgroup. Remarkably, for a generic choice of $L$ the quotient $\C^2/L$ is {\red not an algebraic variety!} For some, well-understood, conditions on $L$ the quotients are algebraic, and are referred to as {\blue abelian surfaces}.

\smallskip
We would be amiss not to mention the Frey elliptic curves which play a crucial role in the proof of Fermat's Last Theorem. For a putative integer solution to 
$$
a^n+b^n=c^n,
$$
Hellegouarch \cite{Hellegouarch} and then Frey \cite{Frey1,Frey2} considered the elliptic curve 
$$
y^2 z = (x - a^n z) x (x +b^n z)
$$
and proved that it would have some weird properties. This put  FLT squarely into a well-studied area of elliptic curves over $\Q$ and eventually lead to Wiles' proof.\footnote{ I was still in high school when Frey's paper discussing the connection between FLT and modularity conjecture appeared, but a mathematician I knew correctly predicted that because of it FLT will now be proved sooner rather than later.}

{\bf Exercise 1.} Prove \eqref{proptheta}.

{\bf Exercise 2.} Prove that $\theta(z)$ has zeros of first order at $L=\Z+\Z\tau$ by integrating $\frac {\theta'(z)}{\theta(z)}$ along the boundary of a fundamental parallelogram of $L$.

{\bf Exercise 3.} Prove that the function
$$
 -\ii \sum_{n\in\Z} (-1)^{n} \ee^{\pi\ii((n+\frac 12)^2 \tau + (2n+1)z)}
$$
is well-defined and holomorphic, and satisfies the transformation properties \eqref{proptheta}.

\section{Genus of complex algebraic curves and Riemann-Hurwitz formula.}\label{sec.RH}

Let $X$ be a compact connected Riemann surface. 
Since it is an oriented real surface, it is homeomorphic to a sphere with $g$ handles.

\smallskip
For $g=0$, the only\footnote{ Uniqueness is not obvious.} such surface is $\C\P^1$. For $g=1$ these are elliptic curves. They form a one-parameter family (basically $\tau$ in $L=\Z+\Z\tau$). A general picture for $g\geq 2$ is something like:

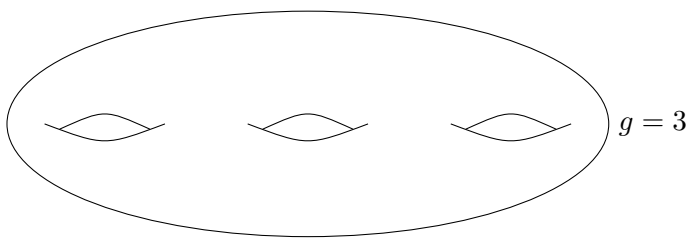
\begin{tikzpicture}
\draw (6,0) ellipse (4 and 1.5);
\draw (2.5,0)..controls(3.3,-.3)..(4.1,0);
\draw (2.7,-0.07)..controls(3.3,.2)..(3.9,-0.07);
\draw (5.2,0)..controls(6,-.3)..(6.8,0);
\draw (5.4,-0.07)..controls(6,.2)..(6.6,-0.07);
\draw (7.9,0)..controls(8.7,-.3)..(9.5,0);
\draw (8.1,-0.07)..controls(8.7,.2)..(9.3,-0.07);
\draw (10,0) node[anchor=west]{$g=3$}; 
\end{tikzpicture}

\noindent
A nontrivial theorem implies that $X$ has a unique structure of a complex algebraic curve. The difficult part is to show that there is a meromorphic function $X\to \C\P^1$.

\smallskip
{\blue \hrule}
Let us discuss {\blue nonconstant} holomorphic maps between Riemann surfaces. Suppose we have one such map $f:X\to Y$. For a point $p\in X$ consider $q=f(p)$. If we consider local charts at $p$ and $q$ with coordinates $z$ and $t$ respectively (so that $z=0$ corresponds to $p$ and $t=0$ corresponds to $q$) then the map is given by
$$
t=f(z) = c_1 z + c_2 z^2 + \ldots
$$
and we will call the smallest $k$ such that $c_k\neq 0$ the {\blue multiplicity} of $f$ at $p$. One can think of the multiplicity of $f$ at $p$ as the number of solutions to $f(z) = \varepsilon$ in the neighborhood of $z=0$ for small $\varepsilon$, see Exercise 1.
We also note that there are only finitely many points in $X$ where the multiplicity of $f$ is larger than one. Indeed, by compactness of $X$, we would otherwise have an accumulation point which would lead to an accumulation point of the zeros of the derivative, thus making $f$ constant by analytic continuation.

\smallskip
\begin{prop}\label{deg}
For $q\in Y$ the number of preimage points $p\in X$, counted with multiplicities
\begin{equation}\label{totalmult}
\sum_{p,f(p)=q} m_p 
\end{equation}
is finite and is independent of $q$.
We call it the  {\blue degree} of the map $f$.
\end{prop}

\begin{proof}
We will sketch the proof and the reader is welcome to fill in the details.

Finiteness is obvious due to compactness of $X$, since an accumulation point would imply that $f$ is locally (and hence globally) constant. 
Now suppose that $p_1,\ldots,p_k$ map to $q\in Y$ with multiplicities $m_1,\ldots,m_k$. Let $U$ be a small neighborhood of $q$ and consider connected components of the open set $f^{-1}(U\setminus \{q\})$. By compactness of $X$, every such connected component must have one of $x_i$ in its closure, since we can find an accumulation point of a sequence of points in $X$ whose values approach $q$. We can pick a small enough neighborhood $U$ of $q$ so that preimages of $U$ near $x_i$ are close enough to $x_i$ (in some metric) to be disjoint. This means that that for $q$ in this $U$ all of the preimages are close to $x_i$ and by Exercise 1, the total number in \eqref{totalmult} is locally constant and hence constant.
\end{proof}

\begin{rem}\label{degalg}
The invariant $\deg f$ has an algebraic meaning. The fields of meromorphic functions on $X$ and $Y$ are transcendence degree one field extensions of $\C$ and the pullback under $f$ induces a field extension
$$
{\rm Mer}(Y)\subseteq {\rm Mer}(X).
$$
Then this extension is finite and $\deg f$ is its degree, i.e. $\dim_{{\rm Mer}(Y)}{\rm Mer(X)}$. We do not prove this claim, but see Exercise 2.
\end{rem}

Proposition \ref{deg} has an important corollary.
\smallskip
\begin{cor}\label{cor.zp}
For any nonzero meromorphic function $f$ on $X$ with zeros/poles at points $x_i\in X$ of multiplicities $m_i$, there holds
$\sum_i m_i=0$.
\end{cor}

\begin{proof}
If $f$ is a nonzero constant, then the sum is empty. Otherwise, we can think of $f$ as a nonconstant holomorphic map $X\to \C\P^1$. The sum of positive $m_i$ is precisely the degree of $f$ computed via preimages of $0\in \C\P^1$ and the opposite of the sum of negative $m_i$ is the degree of $f$ computed via preimages of $\infty\in \C\P^1$.
\end{proof}

Another proof of Corollary \ref{cor.zp} can be obtained by integrating 
$$
\frac 1{2\pi\ii} \int_\gamma d \log f = \frac 1{2\pi\ii} \int_\gamma  \frac {df}f 
$$
over a contour $\gamma$ which is the boundary of a region $U$ that contains all zeros and poles of $f$. Now observe that $\gamma$ is also the boundary (with the wrong sign) of the complement of $U$ which now has no zero or poles of $f$ and thus the integral is zero.

\smallskip
Corollary \ref{cor.zp} implies that the map ${\rm WeilDiv}(X)\to \Z$ which sends 
$$
\sum_k m_k [p_k] \mapsto \sum_k m_k
$$
induces a map ${\rm Pic}(X)\to \Z$, which is understandably called the {\blue degree map.} Isn't it nice how there are terms ``the degree of a map" and ``the degree map" which are quite different but sufficiently related to cause confusion?

\smallskip
\begin{rem}
The kernel of the degree map is called ${\rm Pic}^0(X)$. For $g=0$ we have ${\rm Pic}^0(\C\P^1) \cong \{0\}$. For $g=1$ we saw that 
$$
 {\rm Pic}^0(\C/L) \cong \C/L.
$$
More generally, for any compact Riemann surface $X$ of genus $g$ the group ${\rm Pic}^0(X)$ has a natural structure of an algebraic variety of dimension $g$. In fact it is isomorphic to $\C^g/L$ for some lattice $L\cong \Z^{2g}$. We do not prove these statements.
\end{rem}

\smallskip
{\blue \hrule} 
We will now talk about a wonderful result known as the {\blue Riemann-Hurwitz formula}. 
\begin{thm}\label{RH}
Let $f:X\to Y$ be a nonconstant holomorphic map of compact connected Riemann surfaces. Then
$$
2g_X-2 = (\deg f)(2g_Y-2) + \sum_{p\in X}(m_p-1)
$$
where $m_p$ is the multiplicity of $f$ at $p$.
\end{thm}

\begin{proof}
First of all, $\sum_p$ is really just a finite sum since we only need to worry about the points where $m_p>1$. We call
points in $X$ with $m_p>1$ the {\blue ramification points} of $f$.
We pick a triangulation $T_Y$ of $Y$ in such a way that all images of ramification points of $X$ are among the vertices of $T_Y$.
Then we compare the triangulation $T_Y$ with its preimage $T_X$ on $X$. Note that each triangle $\Delta$ in $T_Y$ gives rise to $\deg f$ triangles in $T_X$ because the map $f^{-1}(\Delta)\to\Delta$ is unramified, and the same is true for the edges. For vertices, a vertex $q\in Y$ of $T_Y$ gives rise to 
$$
\sum_{p,f(p)=q}1=\sum_{p,f(p)=q}m_p - \sum_{p,f(p)=q} (m_p-1)
=
\deg f - \sum_{p,f(p)=q} (m_p-1)
$$
vertices in $T_X$. We know that the Euler characteristics of $X$ (resp. $Y$) is the number of triangles minus the number of edges plus the number of vertices of $T_X$ (resp. $T_Y$). It remains to combine the observation about the numbers of triangles, edges and faces of $T_X$ and $T_Y$ with well-known formulas for the Euler characteristics
$
\chi(X) = 2-2g_X,~\chi(Y)= 2- 2g_Y.
$
\end{proof}

We will illustrate the Riemann-Hurwitz formula with a couple of examples.

\begin{itemize}
\item
Let $X$ be an elliptic curve and $X\to \C\P^1$ be the degree two map ramified at $4$ points (see Section \ref{cubicP2}).
The Riemann-Hurwitz formula reads
$$
2\cdot 1 -2 = 2(2\cdot 0 - 2) + 4 \times (2-1).
$$
\item
Let $f:\C\P^1\to \C\P^1$ be the map $(x_0:x_1)\mapsto (x_0^n:x_1^n)$ which is a compactification of the map $z\mapsto z^n$. Then the degree of $f$ is $n$ and there are two ramification points $(0:1)$ and $(1:0)$ with $m_p=n$ each. The Riemann-Hurwitz formula reads
$$
2\cdot 0 -2 = n(2\cdot 0 -2) + (n-1)+(n-1).
$$
\end{itemize}

\smallskip{\blue \hrule}
For a dimension one complex manifold $X$, the holomorphic cotangent bundle is the canonical bundle of $X$. If $X$ is compact, it makes sense to ask about the degree of the corresponding Weil divisor class $K_X$, and Riemann-Hurwitz formula allows us to compute it.

\smallskip
First, we note that in the case of $X=\C\P^1$ the degree of $K_{X_{\C\P^1}}$ is $(-2)$. Indeed, if we look at the differential form $dz$ which has no zeros/poles in $\C$, it has a second order pole at infinity, because $dz = -\left(\frac 1z\right)^{-2}\, d\left(\frac 1z\right)$ near $\infty$.

\smallskip
\begin{prop}\label{degcan} For a smooth compact Riemann surface $X$ we have
$\deg K_X = 2g_X-2$.
\end{prop}

\begin{proof}
For any nonconstant map $f:X\to Y$ we have 
\begin{equation}\label{KRH}
\deg K_X = (\deg f)  \deg K_Y + \sum_{p\in X} (m_p-1).
\end{equation}
Indeed, if we take a meromorphic form on $Y$ and compare its zeros and poles with zeros and poles of its preimage, we will get additional
$(m_p-1)$ for every ramification point, because $dz^k = kz^{k-1}\,dz$.

Because every $X$ has a holomorphic map  $f:X\to \C\P^1$, we combine Riemann-Hurwitz formula for $f$ with \eqref{KRH} to get the desired result.
\end{proof}

\smallskip{\blue \hrule}
Let us now talk about automorphisms of compact Riemann surfaces.

\smallskip
For $g(X)=0$, we have ${\rm Aut}(\C\P^1)={\rm PGL}(2,\C)$, given by the M\"obius transformations $z\mapsto \frac {az+b}{cz+d}$.

\smallskip
For $g(X)=1$, the automorphism group of $X=\C/L$ is a semidirect product of $\C/L$ with a finite group, which could be $\Z/2\Z$, $\Z/4\Z$ or $\Z/6\Z$. The normal subgroup $\C/L$ is obtained by parallel transports $\varphi_w:z\mapsto z+w$, and the finite group comes from scalar multiplications that preserve the lattice $L$. It is usually just $\{\pm 1\}$, but it is $\{1,\ii,-1,\ii\}$ if $L=\Z+\Z \ii$ and 
it is the group of sixth roots of $1$ for the hexagonal lattice $L=\Z + \Z(\frac 12 +\frac {\sqrt 3}2\ii)$. We do not prove this, but it is not too hard to derive from what we know about elliptic curves.

\smallskip
For $g(X)\geq 2$, the automorphism group of $X$ is always finite. One way to prove it (we omit details) is to argue that $X$ can be embedded into a projective space in a way that automorphisms of $X$ extend to the ambient space. Then the automorphism group of $X$ will be an algebraic subgroup of some ${\rm PGL}(n,\C)$, so it will have a finite number of connected components. Finally, the identity component must be just a point, otherwise $X$ would admit a nonzero holomorphic vector field, which would contradict $\deg K_X>0$.

\smallskip
There is an even stronger statement with a surprisingly simple proof.
\begin{thm}
Let $X$ be a compact connected Riemann surface of genus $g\geq 2$. Then
$$
\vert {\rm Aut}(X)\vert \leq 84(g-1).
$$
\end{thm}

\begin{proof}
Let $G$ be the automorphism group of $X$. One can show\footnote{ We will indeed show this later in Corollary \ref{dim1fixed}, although a motivated reader can try to do it now.} that the set of $G$-orbits of $X$ can be given the structure of a Riemann surface so that $X\to Y= X/G$ is a degree $\vert G\vert$ holomorphic map.
An orbit of $G$ with a stabilizer of size $n$ contributes $\vert G\vert/n$ points of multiplicity $n$ to the set of ramification points of $X\to Y$. Therefore, the Riemann-Hurwitz formula gives
$$
2g - 2 = |G| (2g_Y-2) + |G|\frac {n_1-1}{n_1} +  |G|\frac {n_2-1}{n_2} + \ldots +  |G|\frac {n_k-1}{n_k} 
$$
for some integers $n_1,\ldots,n_k$ that encode stabilizer sizes of the action of $G$.

\smallskip
If $g_Y\geq 2$, then we have $2g-2 \geq |G| \cdot 2$, and the claim follows. 
If $g_Y=1$, then we have 
$$
2g-2 = |G| \sum_{i=1}^k (1-\frac 1{n_i}).
$$
Since the left-hand side is positive, the sum on the right is nonempty. Each of the terms is at least $|G|\frac 12$, so 
$|G|\leq 4(g-1)$. 

\smallskip
If $g_Y=0$, we see that 
$$
\frac {2g-2}{|G|} = -2 + \sum_{i=1}^k  (1-\frac 1{n_i}).
$$
So it remains to observe that the smallest possible positive value of the right hand side is $\frac 1{42}$, achieved by 
$$-2 + \frac 12 + \frac 23 + \frac 67,$$
see Exercise 3.
\end{proof}

{\bf Exercise 1.} Prove that for a nonconstant holomorphic function $f$ with $f(0)=0$ the number of solutions to $f(0)=\varepsilon$ is equal, for all small nonzero $\varepsilon$, to the smallest $k$ such that the $k$-th derivative of $f$ at $z=0$ is nonzero. \emph{Hint:} Since $f'(z)$ is holomorphic, there is a neighborhood where $f'(z)$ has no other zeros than perhaps $z=0$. Thus any zeros of $f(z)-\varepsilon$ will be simple. Then you can count the their number in a circle $|z|<\alpha$ by 
$$
\frac 1{2\pi\ii} \oint_{|z|=\alpha} \frac {f'(z)}{f(z)-\varepsilon}\,dz.
$$
It is continuous in $\varepsilon$ for $\varepsilon$ close to $0$ and is equal to $
\frac 1{2\pi\ii} \oint_{|z|=\alpha} \frac {f'(z)}{f(z)}\,dz.
$

{\bf Exercise 2.} Check the claim of Remark \ref{degalg} for the map $f:\C\P^1\to\C\P^1$ that sends $(x_0:x_1)$ to $(x_0^n:x_1^n)$ for some positive integer $n$.

{\bf Exercise 3.} Prove the last claim of the section. \emph{Hint:} First argue that for $k\geq 5$ the right hand side is at least $\frac 12$. Then for $k=4$ the right hand is positive, so at least one of $n_i$ is $3$ or more, which gives at least $-2+\frac 12 +\frac 12 +\frac 12 + \frac 23 = \frac 16$. If $k\leq 2$, the right hand side is negative, so we must have $k=3$, etc.

\section{Blowup of a point in $\C^2$. Birational equivalence of algebraic varieties.}\label{sec.blowup}
We are now switching gears to talk about phenomena that appear in complex dimension two and higher.
We start with a wonderful and initially puzzling construction of the {\blue blowup} of a point in $\C^2$.

\smallskip
Consider the subset $S\subset \C^2 \times \C\P^1$ defined by
$$
S = \{(x,y)\times(u:v) \in  \C^2 \times \C\P^1, {\rm~such~that~}xv-yu = 0\}.
$$
\begin{prop}
The set $S$ is a smooth complex algebraic surface.
\end{prop}

\begin{proof}
Recall that $\C\P^1$ is covered by two copies of $\C^1$, with points $(\frac uv:1)$ and $(1:\frac vu)$. This gives a covering of 
$\C^2\times \C\P^1$ by two copies of $\C^3$. In the first open set, when using $z=\frac uv$ as the third coordinate, we get
$$
S\cap \C^3=\{(x,y,z)\in \C^3,{\rm ~such~that~}x - y z =0\}
$$
which is  isomorphic to $\C^2$ (we can solve for $x$). Similarly, in the second open set we get
$$
x \left(\frac vu\right) - y
$$
and we can solve for $y$.
It is also clearly algebraic, for example because it is given by a polynomial equation in $\C^2\times \C\P^2$.
\end{proof}

Whenever we have a subset in a product, it is natural to consider the two projections of it. We will first look at the projection 
$$
S\to \C\P^1
$$
that forgets $(x,y)$. The fiber of $S$ over $(u:v)$ is the set of all $(x,y)$ such that $xv=yu$. Since $(u,v)\neq 0$, this is 
equivalent to $(x,y) = \lambda (u,v)$ for some $\lambda \in \C$. We can see from this that $S$ is the line bundle $\mathcal O(-1)$ over $\C\P^1$.

\smallskip
The other projection
$$
\pi:S\to\C^2
$$
is even more interesting. We get
$$
\pi^{-1}(a,b) = \left\{\begin{array}{ll}\{(a,b)\times (a:b)\},&{\rm if~}(a,b)\neq (0,0);\\
\{(0,0)\}\times \C\P^1,&{\rm if~}(a,b)=(0,0).\end{array}\right.
$$
So we have drastically different fibers -- a single point everywhere except $(0,0)$ and a whole $\C\P^1$ over $(0,0)$.

\smallskip
How should we visualize the blowup surface $S$? It is important to realize that the naive picture on the left in
the figure below

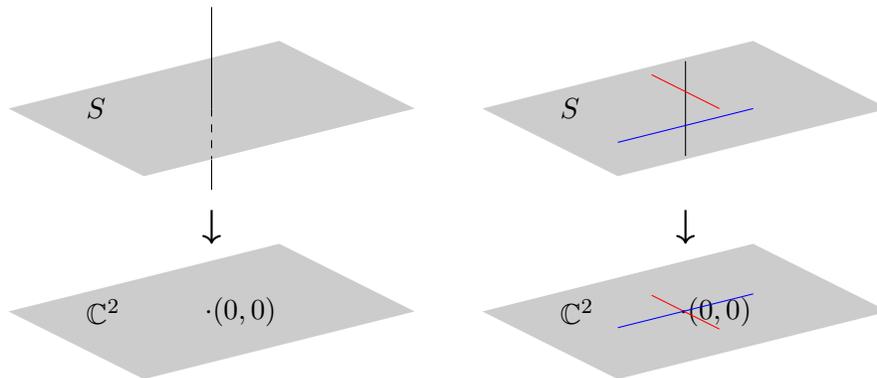
\begin{tikzpicture}[scale = .9]
\fill[gray!40!white] (0,0.0)--(4,1)--(6,0)--(2,-1)--(0,0);
\draw (3,1.5)--(3,0);
\draw [dashed](3,0)--(3,-.75);
\draw (3,-.76)--(3,-1.2);
\draw (1,0) node [anchor=west]{$S$};
\draw[thick,->] (3,-1.5)--(3,-2);
\fill[gray!40!white] (0,-3)--(4,-2)--(6,-3)--(2,-4)--(0,-3);
\draw (3,-3) node [anchor=west]{$\hskip -6.5pt \cdot (0,0) $};
\draw (1,-3) node [anchor=west]{$\C^2 $};
\fill[gray!40!white] (7,0.0)--(11,1)--(13,0)--(9,-1)--(7,0);
\draw (10,.7)--(10,-.7);
\draw (8,0) node [anchor=west]{$S$};
\draw[thick,->] (10,-1.5)--(10,-2);
\fill[gray!40!white] (7,-3)--(11,-2)--(13,-3)--(9,-4)--(7,-3);
\draw (10,-3) node [anchor=west]{$\hskip -6pt \cdot (0,0) $};
\draw (8,-3) node [anchor=west]{$\C^2 $};
\draw[red] (9.5,-2.76)--(10.5,-3.26);
\draw[blue] (9,-3.24)--(11,-2.74);
\draw[red] (9.5,0.5)--(10.5,0);
\draw[blue] (9,-0.5)--(11,0);
\end{tikzpicture}

\noindent
is wrong, because $S$ would not be smooth. The correct way of thinking about it is the picture on the right,
where I also indicate that approaching the origin along different directions gives different limiting points on $S$.

\smallskip
{\blue Terminology:} The map $S\to \C^2$ is called the {\blue blowdown map} or the {\blue blowdown morphism}, and $S$ is called the blowup of $\C^2$ at $(0,0)$ because that point is ``blown up" to a $\C\P^1$.\footnote{ Algebraic geometers waiting for their flight in the airport should refrain from discussing blowing up the plane at any number of points.}

\smallskip
The construction generalizes, of course. First of all, we can blow up a point in $\C^n$. Consider 
$S\subset \C^n \times \C\P^{n-1}$ given by
$$
S=\{(x_1,\ldots,x_n) \times (y_1:\ldots:y_n),{\rm ~such~that~rk}\hskip -2pt
\left(
\begin{array}{ccc}
x_1& \cdots &x_n \\
y_1& \cdots &y_n 
\end{array}
\right)
=1\}.
$$
Since not all $y_i$ are zero, the rank of the matrix is at least one, so $S$ is cut out by all $2\times 2$ minors of the matrix, i.e.
$$
x_i y_j -x_j y_i =0,~ 1\leq i<j\leq n.
$$
We claim that $S$ is smooth. Indeed, suppose we are working in the chart of $\C\P^{n-1}$ where $y_n\neq 0$. Then
the rank condition is equivalent to 
$$
(x_1,\ldots,x_n) = \lambda(\frac {y_1}{y_n},\ldots,\frac {y_{n-1}}{y_n},1).
$$
This means that $\lambda =x_n$ and then in this chart $S$ is isomorphic to $\C^n$ with coordinates 
$$x_n,\frac  {y_1}{y_n},\ldots, \frac{y_{n-1}}{y_n}.$$

\smallskip
Looking at the fibers of $\pi :S \to \C^n$, we see that 
$$
\pi^{-1}(x_1,\ldots,x_n) = \left\{\begin{array}{ll}\{(x_1,\ldots,x_n)\times (x_1:\ldots:x_n)\},&{\rm if~}(x_1,\ldots,x_n)\neq (0,\ldots,0),\\
\{(0,\ldots,0)\}\times\C\P^{n-1},&{\rm if~}(x_1,\ldots,x_n)= (0,\ldots,0),\end{array}\right.
$$
so the point $(0,\ldots,0)\in\C^n$ is blown up to $\C\P^{n-1}$.

\smallskip
\begin{rem}
In general, it is possible to blow up a smooth subvariety $Z$ in a smooth variety $X$ to get a blowdown morphism
$$
\pi:{\rm Bl}(Z\subset X)\to X
$$
with the following properties.
\begin{itemize}
\item
${\rm Bl}(Z\subset X)$ is a smooth algebraic variety.

\item
Fibers of $\pi$ over points in $X\setminus Z$ are points, i.e. $\pi$ induces an isomorphism 
$\pi^{-1}(X\setminus Z)\to X\setminus Z$.

\item
The morphism $\pi^{-1}Z\to Z$ can be naturally identified with the map $\P N(Z\subset X) \to Z$.
Here $N(Z\subset X)$ is the normal bundle\footnote{ A fiber of this bundle at a point $z\in Z$ is the quotient of the tangent space to $X$ at $z$ by the tangent space to $Z$ at $z$.}
 to $Z$ in $X$ and $\P W$ means the space of lines in the fibers of 
a vector bundle $W$.
\end{itemize}
In other words, a subvariety $Z$ gets blown up to the projectivization of its normal bundle.
\end{rem}

\smallskip
Even more generally, one can blow up arbitrary (i.e. singular) subvarieties of arbitrary varieties, but then there is no guarantee that the result is smooth. In fact, it usually isn't.

\smallskip
{\blue \hrule}

The blowup construction inspires the concept of {\blue birational equivalence} of algebraic varieties.

\smallskip
\begin{defn}
We call complex algebraic varieties $X$ and $Y$ birationally equivalent, if their fields of rational functions are isomorphic as field extensions of $\C$.
\end{defn}

\smallskip
\begin{rem}
An equivalent, more geometric, definition is that $X$ and $Y$ are birational if and only if there exist Zariski open subsets
$U\subseteq X$ and $V\subseteq Y$ with $U$ isomorphic to $V$ as a complex algebraic variety. We also call a morphism of algebraic varieties $X\to Y$ a birational morphism if it induces an isomorphism of Zariski open subsets.
\end{rem}

The second definition makes it clear that blowups (and blowdowns) do not change the birational equivalence class.
Study of birational equivalence is known as {\blue birational geometry} and is a big active area of Algebraic Geometry, with multiple Fields medals awarded for it (Hironaka, Mori, Birkar). A major part of birational geometry is the so-called Minimal Model Program
which aims to find a ``nice" representative in each birational class. Some singular varieties naturally occur in the Minimal Model Program.

\smallskip
I will state some accessible results of birational geometry, without proof.
\begin{itemize}

\item
For every, possibly singular, complex algebraic variety $X$ there exists a smooth complex algebraic variety $Y$ with a birational morphism $Y\to X$ (Hironaka). This also holds over other fields of characteristic zero, but is a very famous open problem in positive characteristics!

\item 
 $\dim = 1$. Two smooth complex projective  algebraic varieties $X$ and $Y$ of dimension one are birational to each other if and only if they are isomorphic.
 
 \item 
 $\dim = 2$. Two smooth complex projective algebraic varieties $X$ and $Y$ of dimension two are birational to each other if and only if they are connected by a sequence of blowups of points followed by a sequence of blowdowns, as indicated below.
\begin{equation}\label{strongfac}
X\leftarrow\cdots\leftarrow Z \rightarrow \cdots\rightarrow Y
\end{equation}

\item 
 $\dim = {\rm any}$. ({\blue Weak Factorization Theorem}, by Abramovich, Karu, Matsuki, W\l odarczyk, see \cite{AKMW})
Two smooth complex projective algebraic varieties $X$ and $Y$ of any dimension are birational to each other if and only if they are connected by a sequence of blowups and blowdowns of smooth subvarieties as indicated below
\begin{align*}
X = X_0 \dashleftarrow \hskip -5pt\dashrightarrow X_1 \dashleftarrow \hskip -5pt\dashrightarrow \cdots \dashleftarrow \hskip -5pt\dashrightarrow X_{n-1}
\dashleftarrow \hskip -5pt\dashrightarrow X_n=Y
\end{align*}
where each $\dashleftarrow\hskip -5pt \dashrightarrow$ is either a blowup or a blowdown morphism.\footnote{ The precise statement is stronger but a bit more technical.} 
It is an open problem, even in dimension $3$, whether the weak factorization statement can always be upgraded to strong factorization, as in \eqref{strongfac}.
\end{itemize}

\smallskip
{\blue \hrule}

When one has a birational morphism $\pi:X\to Y$, one can talk about the {\blue exceptional locus} ${\rm Exc}$ of $\pi$, which are the points $x\in X$ where $\pi$ is not an isomorphism in any neighborhood of $x$. One is also interested in the image $\pi({\rm Exc})$ in $Y$. If $Z$ is any closed subset of $Y$ which is not contained in $\pi({\rm Exc})$ then the {\blue proper preimage} of $Z$ is defined as the closure
\begin{equation}\label{properpreimage}
\overline {\pi^{-1}\Big(C\setminus (C\cap \pi({\rm Exc}))\Big)}.
\end{equation}
This construction could be used to desingularize algebraic varieties by looking at their proper preimages in carefully chosen blowups of some ambient varieties. 

{\bf Exercise 1.} Check that the blowup $S$ of $\C^n$ at the origin can be identified with the line bundle $\mathcal O(-1)\to \C\P^{n-1}$.

{\bf Exercise 2.} Let $W$ be the subset of $\C\P^2\times \C\P^2$ given by the conditions that the points $(x_0:x_1:x_2)\times (y_0:y_1:y_2)$ 
satisfy
$$
x_0 y_0 = x_1 y_1 = x_2 y_2.
$$
Prove that $W$ is a smooth surface. Verify that the projection maps $W\to \C\P^2$ are birational and find the image of the exceptional locus (i.e. points in $\C\P^2$ with more than one preimage point). \emph{Remark:} $W$ is the closure of the graph of the Cremona transformation $(x_0:x_1:x_2)\mapsto (\frac 1{x_0}:\frac 1{x_1}:\frac 1{x_2})$.

{\bf Exercise 3.} Let $x^2-y^3=0$ be a singular curve on $\C^2$. Compute its proper preimage \eqref{properpreimage} on the blowup $S$ of $\C^2$ at $(0,0)$ and verify that it is a smooth curve. \emph{Hint:} Work in the coordinate charts on $S$.

\section{Finite group actions on affine algebraic varieties and complex manifolds. Subgroups of $SL_2(\C)$.}
Let $X\subseteq \C^n$ be an affine complex algebraic variety. Suppose that a finite group $G$ acts on $X$, 
i.e. 
$$
G\subseteq {\rm Aut}_{\C-{\rm alg}}(\C[x_1,\ldots,x_n]/I)
$$
where $I$ is the defining ideal of $X$.
What can we say about the set of $G$-orbits $X/G$? The following beautiful result goes back to at least Hilbert.

\smallskip
\begin{thm}\label{XG}
The set $X/G$ has a natural structure of an affine algebraic variety. More precisely, the subring of $G$-invariant elements
$$
(\C[x_1,\ldots,x_n]/I)^G
$$
is a finitely generated $\C$-algebra, and its maximal ideals are in a natural bijection with $G$-orbits on $X$.
\end{thm}

\begin{proof}
Consider the elements $y_{ik} = g_i(x_k)$ for all $g_i\in G$ and all $1\leq k\leq n$. We can use this set of generators to write $X$ as an algebraic subvariety of $\C^{n|G|}$, so that the action of $G$ comes from a permutation action on the coordinates of $\C^n$.
Then we have 
$$
(\C[x_1,\ldots,x_n]/I)^G \cong (\C[y_1,\ldots,y_l]/J)^G\cong \C[y_1,\ldots,y_l]^G/J^G.
$$
In the second isomorphism we used the fact that every element $r$ of the invariant ring  $(\C[y_1,\ldots,y_l]/J)^G$
can be lifted to an element $\hat r$ of  $\C[y_1,\ldots,y_l]$ and then 
$$\frac 1{|G|}\sum_{g\in G} g(\hat r)=r\,{\rm mod}\,J.$$
Thus, to prove that $(\C[x_1,\ldots,x_n]/I)^G$ is a finitely generated $\C$-algebra, it suffices to prove it for
$ \C[y_1,\ldots,y_l]^G$, with the permutation action of $G$.
We actually don't care too much that the action is a permutation action, but we will use a weaker statement that the action of 
$G$ on $A= \C[y_1,\ldots,y_l]$ preserves the natural grading on $A$ by the total degree in $y_i$.  

\smallskip
Consider the ideal $A A_+^G\subset A$ which is generated by all homogeneous elements of $A^G$ of positive degree. Since $A$ is Noetherian, we can find a finite set of generators $f_1,\ldots, f_k$ of $A A_+^G$, and we can even assume that all $f_i$ are $G$-invariant homogeneous elements of $A$. We will now prove that
$f_i$ generate the ring $A^G$ as a $\C$-algebra, by induction on degree.

\smallskip
In degree zero, we have $A_0 = A^G_{0}=\C$, and there is nothing to prove. Suppose we have proved that every homogeneous element of $A^G$ of degree at most $d$ can be written as a polynomial in $f_1,\ldots, f_k$ with constant coefficients. Then for a homogeneous element $f\in A^G$ of degree $d+1$ we use $f\in AA_+^G$ to write 
$$
f = \sum_{i=1}^k a_i f_i
$$
where $a_i$ are some elements of $A$ of degree $d+1-\deg f_i$. We will now use the averaging trick
\begin{align*}
f = \frac 1{|G|} \sum_{g\in G}g(f) =\frac 1{|G|} \sum_{i=1}^k  \sum_{g\in G}g(a_i f_i) 
=\frac 1{|G|} \sum_{i=1}^k  \sum_{g\in G}g(a_i) f_i
\\
=\sum_{i=1}^k  f_i \left(\frac 1{|G|} \sum_{g\in G}g(a_i) \right) = \sum_{i=1}^k f_i h_i
\end{align*}
where $h_i$ are $G$-invariant elements of degree $d+1-\deg f_i$. By induction assumption, $h_i$ can be written as polynomials in $f_1,\ldots,f_k$, therefore so can be $f$.

\smallskip
Now that we have established that $(\C[x_1,\ldots,x_n]/I)^G$ is finitely generated, let us prove the bijection between its maximal ideals and the set of orbits of $G$-action on the set of maximal ideals of $\C[x_1,\ldots,x_n]/I$.\footnote{ Recall that points in $X$ correspond to maximal ideals in $R=\C[x_1,\ldots,x_n]/I$ by Hilbert's Nullestellensatz.} In one direction, to any maximal ideal $m\subset R$ we associate the ideal $m^G = m\cap R^G\subset R^G$. The ideal $m^G$ is maximal because for any 
$x\in R^G\setminus m^G$ there exists $y\in R$ such that $xy=1\,{\rm mod}\,m$. Then
the element 
$$
\prod_{g\in G} (1 - g(x y))  = \prod_{g\in G} (1 - x g(y))  = 1 - x\hat y
$$
(with $G$-invariant $\hat y$) lies in $m^G$, which shows that $x \,{\rm mod}\,m^G$ is invertible. It is clear that different $m$ in the same $G$-orbit give the same $m^G$.

\smallskip
In the other direction, for a maximum ideal $I\subset R^G$ consider the ideal $RI$ that it generates in $R$. If $RI$ contains $1$
then there holds
$$
1 = \sum_i r_i x_k
$$
with $x_k\in I$. Averaging gives $1 =\sum_i \left (\frac 1{|G|}r_i\right)x_k\in I,$ contradiction. Therefore, $RI$ is contained in some maximal ideal $m$. It is clear that $m^G$ contains $I$ and is therefore equal to it, since $I$ is maximal.

\smallskip
It now suffices to show that if $m_1$ and $m_2$ satisfy $m_1^G=m_2^G$, then $m_1$ and $m_2$ are in the same $G$-orbit.
Consider $I_1 = \bigcap_{g\in G} g(m_1)$ and $I_2= \bigcap_{g\in G} g(m_2)$. If $I_1+ I_2$ were contained in a maximum ideal $m$,
we would have 
$$
m\supseteq \bigcap_{g\in G} g(m_1),~m\supseteq \bigcap_{g\in G} g(m_2).
$$
If $m$ is not equal to any $g(m_1)$, then there are elements $x_g\in g(m_1),x_g\not\in m$, and their product is in $ \bigcap_{g\in G} g(m_1)$ but not in $m$. Similarly, $m$ must be equal to one of $g(m_2)$, and this is impossible, since $m_1$ and $m_2$ are in different $g$-orbits. Therefore, $I_1+I_2=R$ and\footnote{ You have been warned.}
$$
1 = a_1+a_2,~a_i\in I_i.
$$
Moreover, since $I_i$ are $G$-invariant, when we average over $G$, we may assume that $a_i$ are $G$-invariant. However, this means that $a_i\in m_i^G$, which contradicts $1\not\in m_1^G=m_2^G$.
\end{proof}

\begin{rem}
The invariant rings are also finitely generated when the group $G$ is a {\blue linear reductive group}. However, they don't have to be finitely generated in the non-reductive case, see a famous counterexample to the Hilbert's fourteenth problem given by Nagata \cite{Nagata}. 
\end{rem}

\smallskip
{\blue \hrule}
We will now discuss finite group actions on smooth complex manifolds. 
Suppose a finite group $G$ fixes a point $p\in X$. Then locally we have a surjective map of $G$-representations 
$$
{\rm Hol}^0(p) \stackrel {d}{\to} \left(TX \right)^\vee_p\to 0
$$
where ${\rm Hol}^0(p)$ denotes the infinite-dimensional space of holomorphic functions near $p$ which are zero at $p$.
This surjection has a $G$-invariant splitting. To see this, a person nervous around infinite-dimensional vector spaces, can lift a basis of $\left(TX \right)^\vee_p$, look at the span of the $G$-translates of it and thus reduce the problem to that of surjective finite-dimensional representations of $G$. The $G$-invariant splitting then endows $X$ near $p$ with analytic coordinates on which $G$ acts {\blue linearly}.

\smallskip
\begin{cor}\label{dim1fixed}
In dimension one, all actions that fix a point $p$ locally analytically look like $z\mapsto \alpha z,~\alpha^n=1$, and the quotient is given by 
$z\mapsto z^n=w$.
\end{cor}

\begin{proof}
Locally near $p$, we have an action of ${\rm GL}(1,\C) =\C^*$, and all finite subgroups of $ \C^*$ are given by groups of $n$-th roots of $1$.
\end{proof}

\begin{cor}
For any finite group $G$ of automorphisms of a complex manifold $X$, the fixed locus $X^G$ is a disjoint union of  complex submanifolds of $X$.
\end{cor}

\begin{proof}
Since the action is locally linear, it suffices to look at $G\subseteq {\rm GL}(n,\C)$ acting on $X=\C^n$. Then $X^G$ is the linear subspace $(\C^n)^G$.
\end{proof}

\smallskip
{\blue \hrule}

We will now focus on dimension two, and more precisely on the action of finite subgroups
$G\subset {\rm SL}(2,\C)$ on $\C^2$. (We could have considered ${\rm GL}(2,\C)$ but it is more complicated.)

\smallskip
We first remark that any such $G$ can be conjugated to lie inside the unitary group ${\rm SU}(2)$. Indeed,
$G\subset SU(2)$ iff $G$ fixes the standard Hermitean on $\C^2$. We can take a positive definite Hermitean form on $\C^2$ and average it over $G$ to prove that $G$ preserves {\blue some} positive definite Hermitean form. Since all such forms can be written as the standard form in some basis, this change of basis makes a conjugate of $G$ unitary.\footnote{ This is a particular case of the Weyl's unitarian trick.} We put a matrix version of this  statement in Exercise 2.

\smallskip
Recall that 
$$
{\rm SU}(2) = \left\{\left(\begin{array}{rr} a&-\bar b\\
b&\bar a\end{array}\right),~|a|^2 + |b^2|=1\right\}
$$
is a real $3$-sphere. We can also think of it as the length one quaternions $\gamma\in \H$ acting on $\H\cong \C^2$ by 
left multiplication $\gamma :w\mapsto \gamma w$. 

\smallskip
Now consider the conjugation action
$$
w\mapsto \gamma w \gamma^{-1}
$$
of ${\rm SU}(2)$ on $\H$.
It preserves the trace zero part of $\H$ (see Exercise 3) and the norm
$$
|w|^2 = w \bar w
$$
on it. This gives a group homomorphism ${\rm SU(2)} \to {\rm SO}(3,\R)$, whose kernel is $\{\pm 1\}$. Finite subgroups of ${\rm SO}(3,\R)$ are classically known. They include two infinite series (rotation symmetries of a regular pyramid and a regular bipyramid) as well as rotational symmetries of the regular tetrahedron, cube/octahedron, and icosahedron/dodecahedron. 
This allows one to classify all finite subgroups $G\subset {\rm SL}(2,\C)$. 

\smallskip
In the future sections we will study in detail the singular quotients $\C^2/G$ for these groups $G$ and their resolutions of singularities.

\smallskip
{\bf Exercise 1 (Molien series.)} Let $V$ be a complex finite-dimensional representation of a finite group $G$. It induces a grading-preserving action of $G$ on ${\rm Sym}^*(V)\cong \C[x_1,\ldots,x_n]$. Prove that the Hilbert series of the ring of invariants ${\rm Sym}^*(V)^G$
can be computed by
$$
\sum_{d=0}^{\infty} \dim_\C\left({\rm Sym}^d(V)\right)^G\, t^d =
\frac 1{|G|}
\sum_{g\in G} \prod_{i=1}^n \frac {1}{1-\lambda_{g,i}t}
=\frac 1{|G|}
\sum_{g\in G} \det (1-tg)^{-1}
$$
where $\lambda_{g,1},\ldots,\lambda_{g,n}$ are eigenvalues of $g$ on $V$.  \emph{Hint:} For a finite-dimensional $G$-representation $W$ we have $\dim_\C W^G = \frac 1{|G|}\sum_{g\in G} {\rm Tr}(g,W)$. Then use the result that all finite order elements in ${\rm GL}(V)$ are diagonilizable.

{\bf Exercise 2.}
Let $M$ be a Hermitean form on $\C^n$ given by $\langle v,w\rangle = \bar w^T M v$. Prove that an endomorphism
$$
g: v\to Av
$$
satisfies 
$
\langle Av, Aw \rangle = \langle v, w \rangle 
$
for all $v$ and $w$ if and only if $\bar A^TMA = M$. Prove that if for some invertible matrix $S$ we have $M = \bar S^TS$, then
$S A S^{-1}$ is a unitary matrix.

{\bf Exercise 3.} 
Prove that the conjugation action of quaternions preserves the subspace $\R\ii + \R{\rm j} + \R{\rm k}$ and the norm on it.

\section{Resolutions of $A_n$ and $D_4$ singularities.}\label{sec.ADE1}
Let us now look at some specific examples of finite subgroups $G$ of ${\rm SL}(2,\C)$ and the corresponding quotients $\C^2/G$.

\smallskip
The first series of examples comes from the action of the cyclic group of order $n\geq 2$ which necessarily\footnote{ Exercise 1.}
looks like
\begin{equation}\label{cyclic}
G=<g>,~g(x,y) =(\xi x, \xi^{-1} y)
\end{equation}
where $\xi = \ee^{2\pi \ii/n}$.
As a consequence of the definition, the action\footnote{ There is an evergreen issue of the action on points versus the action on functions. We consider the action on the ring $\C[x,y]$.} of the generator $g$ on the monomial $x^ay^b$ is 
$$
g(x^a y^b) = \xi^{a-b} x^ay^b
$$
which implies that the invariant ring $\C[x,y]^G$ is the linear span of monomials $x^ay^b$ with the property $a=b\,{\rm mod}\,n$.
Some prominent examples of such monomials are $u=x^n$, $v=y^n$ and $w=xy$. We claim that every $G$-invariant monomial $x^ay^b$ can be written as a product of these three. Indeed, we can use $u$ and $v$ to reduce to the case $a,b<n$, and then we must have $a=b$. It is also clear that $u,v,w$ satisfy
\begin{equation}\label{An}
uv-w^n = 0,
\end{equation}
so the ring of invariants $\C[x,y]^G$ is isomorphic to $\C[u,v,w]/(uv-w^n)$, and the quotient $\C^2/G$ is a surface in $\C^3$
given by the equation \eqref{An}.
By looking at partial derivatives, we see that the only singular point of this surface occurs at the origin $(0,0,0)$. It is known as the $A_{n-1}$ singularity, and we will soon see why.

\smallskip{\blue \hrule}
We will first look in detail at the $A_1$ singularity given by
\begin{equation}\label{A1}
uv-w^2=0
\end{equation}
This is the simplest hypersurface singularity in complex dimension two, in the following sense. If we have a function
$f(u,v,w)$ with $f(0,0,0)=0$, we can consider its Taylor series
$$
f(u,v,w) = 0 + {\rm ~linear~} + {\rm~ quadratic~} + {\rm higher~order.}
$$ 
If the linear term is nonzero, then $f(u,v,w)=0$ is smooth at $(0,0,0)$. If the the linear term is zero, then we can ask about the quadratic term. If the rank of the corresponding form is the maximum possible $3$, then we get an $A_1$ singularity.

\smallskip
We will now resolve the surface \eqref{A1} by blowing up the ambient $\C^3$ at $(0,0,0)$ and looking at the proper preimage of the surface $uv=w^2$. Here is the intuitive picture of what will happen, where we indicate in blue the preimage of the singular point.

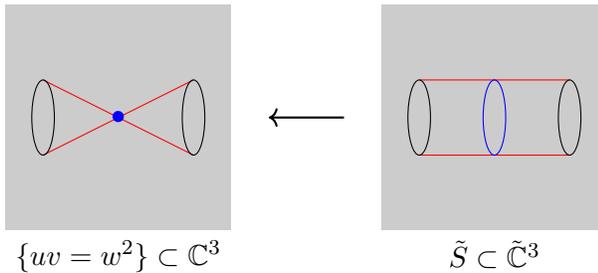
\begin{tikzpicture}
\fill[gray!40!white](4,0) rectangle  (7,3);
\draw[red] (4.5,2) -- (6.5,1);
\draw[red] (4.5,1) -- (6.5,2);
\draw[blue] node at (5.5,1.5){$\bullet$};
\draw (4.5,1.5) ellipse (.15 and .5);
\draw (6.5,1.5) ellipse (.15 and .5);
\draw(5.5,0) node[anchor=north]{$\{uv = w^2\}\subset \C^3$};
\fill[gray!40!white](9,0) rectangle  (12,3);
\draw[thick,<-] (7.5,1.5)--(8.5, 1.5);
\draw[red] (9.5,2) -- (11.5,2);
\draw[red] (9.5,1) -- (11.5,1);
\draw (9.5,1.5) ellipse (.15 and .5);
\draw (11.5,1.5) ellipse (.15 and .5);
\draw [blue](10.5,1.5) ellipse (.15 and .5);
\draw(10.5,0) node[anchor=north]{$\tilde S\subset \tilde \C^3$};
\end{tikzpicture}

Let us see this with all the gory details, in the coordinate charts. Recall that the blowup $\tilde \C^3$ is a submanifold of $\C^3 \times \C\P^2$, given by $\{(u,v,w),(u_1:v_1:w_1)\}$ such that there exists $\lambda$ with $(u,v,w) = \lambda (u_1,v_1,w_1)$.
As a result, $\tilde \C^3$ is covered by three open charts, each one isomorphic to $\C^3$.

\smallskip
In the first open chart $w_1\neq 0$, we have $\C^3$ with coordinates $\frac {u}{w}, \frac {v}{w},w$ and the 
map to the original $\C^3$ is given by
\begin{equation}\label{blowcoord}
(u,v,w) = (w \left(\frac {u}{w}\right), w \left(\frac {v}{w}\right), w).
\end{equation}

\smallskip
\begin{rem}
You may be concerned with such a bizarre notation for the coordinates. Shouldn't they be variables? Yes, of course. Let's say that we have $\C^3$ with coordinates $t_1,t_2,t_3$ which maps to the original $\C^3$ by $(t_1,t_2,t_3)\mapsto (u,v,w)=(t_1 t_3,t_2 t_3, t_3)$.
But then we think about what the coordinate functions $t_i$ look like as rational functions in $u,v,w$ and we recover \eqref{blowcoord}. At which point we forget about $t_i$-s and accept a somewhat weird notation. It does have advantages of being easily able to write the transition functions, although one has to be careful figuring out which subsets in the three open charts are glued together. Fortunately, we will not worry too much about it.
\end{rem}

\smallskip
What is the preimage of $S=\{uv-w^2=0\}$ in the chart with coordinates $\frac {u}{w}, \frac {v}{w},w$? We rewrite the equation of $S$ as
$$
0=\left(\frac uw\right) w \left(\frac vw\right) w  - w^2 = w^2 \left(\left(\frac uw\right) \left(\frac vw\right) -1\right)
$$
and see that the preimage has two components: $w=0$ and $  \left(\frac uw\right)\left(\frac vw\right) -1 = 0$. The first component is the preimage of $(u,v,w)=0$, a.k.a. the exceptional divisor, and the second component is the proper preimage of $S$. Or, more precisely, these are the intersections of the exceptional divisor and the proper preimage of $S$ with this coordinate chart.
Observe now that the proper preimage of $S$ is smooth in this chart.

\smallskip
Let us now look at the coordinate chart with coordinates $u, \frac vu, \frac wu$. We similarly get
$$
0=u^2\left(\frac vu \right)- u^2 \left(\frac wu\right)^2 = u^2 \left(\left(\frac vu\right) - \left(\frac wu\right)^2\right)
$$
and again we see that the proper preimage $0= \left(\frac vu\right) - \left(\frac wu\right)^2$ is smooth.
Obviously the $v_1\neq 0$ chart is handled similarly.

\smallskip
We thus see that the proper preimage $\tilde S$ of $S$ is smooth. It is also instructive to look at the fibers of $\tilde S\to S$. Outside of $(0,0,0)$, the fibers are of course points. The preimage of $(0,0,0)$ is a subvariety of the exceptional divisor, given 
by $(u_1:v_1:w_1)\in \C\P^2$ with $u_1 v_1 - w_1^2 = 0$, so it is a $\C\P^1$. For example, in the coordinate chart $w_1\neq 0$,
we have $\frac uw = \frac {u_1}{w_1},~\frac vw = \frac {v_1}{w_1}$ so $0=\left(\frac uw\right) \left(\frac vw\right) -1$ matches $u_1 v_1 - w_1^2 = 0$ in the appropriate chart on $\C\P^2$.

\smallskip
{\blue \hrule}
Now let us look at the $A_{n-1}$ case, i.e. 
$$
S=\{uv-w^n=0\}.
$$
As before, we blow up $\C^3$ and see what happens on the three coordinate charts, two of which are similar due to $u\leftrightarrow v$ symmetry. On the chart with coordinates $\frac uw, \frac vw, w$ we get
$$
uv-w^n = w^2 \left(\left(\frac uw\right) \left(\frac vw\right) -w^{n-2}\right)
$$
so we may still have a singular point at the origin, but the value of $n$ has decreased by $2$. Also note that the intersection of the proper preimage of $S$ with the exceptional divisor is given in this chart by
\begin{equation}\label{Anlines}
w=0,~ \left(\frac uw\right) \left(\frac vw\right) -w^{n-2}= 0,
\end{equation}
so we get two lines $w=\frac uw=0$ and  $w=\frac vw=0$. In the other two charts the proper preimage is smooth (Exercise 2).
They also compactify the lines \eqref{Anlines} from $\C$ to $\C\P^1$. In short, we get a map $\tilde S\to S$ with the preimage of $(0,0,0)$ being a union of two $\C\P^1$-s, which intersect as coordinate lines at an $A_{n-3}$ singularity, where $A_0$ is by definition smooth. 

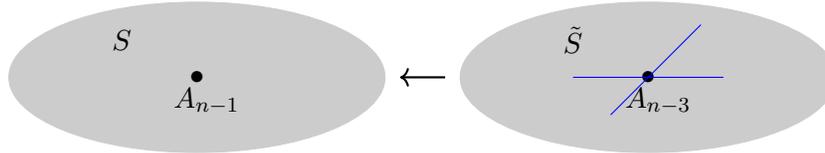
\begin{tikzpicture}
\filldraw[gray!40!white] (3,0) ellipse (2.5 and 1);
\filldraw[gray!40!white] (9,0) ellipse (2.5 and 1);
\node at (3,0){$\bullet$};
\node at (3,0)[anchor=north]{$~~A_{n-1}$};
\node at (2,.5){$S$};
\node at (9,0){$\bullet$};
\node at (9,0)[anchor=north]{$~~A_{n-3}$};
\node at (8,.5){$\tilde S$};
\draw[thick, <-]( 5.7,0)--(6.3,0);
\draw[blue] (9.7,.7)--(8.5,-.5);
\draw[blue] (10,0)--(8,0);
\end{tikzpicture}

\smallskip
We repeat the process. Each time, the power $n$ decreases by $2$ and two new $\C\P^1$-s spring out. A paranoid reader can verify that under the blowup the previous $\C\P^1$-s separate and two new lines intersect them transversely. The last step is either unnecessary (if $n$ is odd, and we reduce to $uv=w$) or we only have one $\C\P^1$ (if $n$ is even and we reduce to $A_1$-singularity). In both cases we end up with a resolution of singularities
$\tilde S\to S$ such that the preimage of the singular point is a chain of $(n-1)$ transversal $\C\P^1$-s.

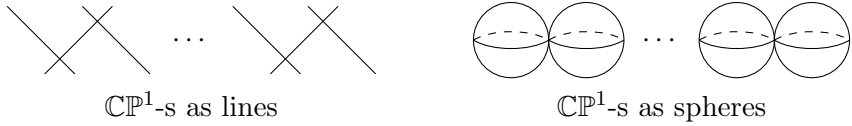
\begin{tikzpicture}
\draw (0.3,.7)--(1.2,-.2);
\draw (.8,-.2)--(1.7,.7);
\draw (1.3,.7)--(2.2,-.2);
\draw (3.3,.7)--(4.2,-.2);
\draw (3.8,-.2)--(4.7,.7);
\draw (4.3,.7)--(5.2,-.2);
\node at (2.75,.25){$\cdots$};
\node at (2.75,-.3)[anchor=north]{$\C\P^1$-s as lines};
\draw (7.0,0.25) ellipse (0.5 and 0.5);
\draw (6.5,0.25) ..controls(6.75,0.1) and (7.25,0.1)  ..(7.5,0.25);
\draw [dashed] (6.5,0.25) ..controls(6.75,0.4) and (7.25,0.4)  ..(7.5,0.25);
\draw (8.0,0.25) ellipse (0.5 and 0.5);
\draw (7.5,0.25) ..controls(7.75,0.1) and (8.25,0.1)  ..(8.5,0.25);
\draw [dashed] (7.5,0.25) ..controls(7.75,0.4) and (8.25,0.4)  ..(8.5,0.25);
\node at (9,0.25) {$\cdots$};
\draw (10.0,0.25) ellipse (0.5 and 0.5);
\draw (9.5,0.25) ..controls(9.75,0.1) and (10.25,0.1)  ..(10.5,0.25);
\draw [dashed] (9.5,0.25) ..controls(9.75,0.4) and (10.25,0.4)  ..(10.5,0.25);
\draw (11.0,0.25) ellipse (0.5 and 0.5);
\draw (10.5,0.25) ..controls(10.75,0.1) and (11.25,0.1)  ..(11.5,0.25);
\draw [dashed] (10.5,0.25) ..controls(10.75,0.4) and (11.25,0.4)  ..(11.5,0.25);
\node at (9,-.3)[anchor=north]{$\C\P^1$-s as spheres};
\end{tikzpicture}

\noindent
Neither of the above pictures is entirely correct, since the Riemann spheres intersect transversely. At any rate, we will consider
the dual graph that encodes the components of the exceptional locus of $\tilde S\to S$ as vertices and their intersections as edges. It is
$$
\circ - \hskip -5pt - \circ - \hskip -5pt - \cdots - \hskip -5pt - \circ - \hskip -5pt - \circ 
$$
with $n-1$ vertices, which is the $A_{n-1}$ Dynkin diagram.\footnote{ There is, of course, more to this notation. Appropriately understood intersection pairing matrix of the components is up to the sign the Cartan matrix of $A_{n-1}$. We will not try to define it now, but will comment on it in Section \ref{sec.int}.}

\smallskip
{\blue \hrule}
Now let us look at the next case. Here we have a subgroup $G$ of ${\rm SL}(2,\C)$ generated by
\begin{equation}\label{Dn}
(x,y)\mapsto (\xi x, \xi^{-1} y),~(x,y)\mapsto (-y,x).
\end{equation}
If we apply the second operation twice, we get $(x,y)\mapsto (-x,-y)$, so we might as well assume that $\xi$ is the primitive $(2n)$-th root of $1$. 

\smallskip
Let us first compute the ring of invariants. The cyclic subgroup is normal, so we will first take the quotient by that to get
$u=x^{2n},v= y^{2n}, w= xy$ with the ring
$$
\C[u,v,w]/(uv-w^{2n}).
$$
The symmetry $(x,y)\mapsto (-y,x)$ now becomes the involution $(u,v,w) \mapsto (v,u,-w)$. We can take the invariants first and then the quotient. The action of this involution on $\C^3$ has positive eigenspace $u+v$ and the $A_1$-type action on $u-v$ and $w$, so the ring of invariants is generated by
$$
t_1 = u+v,~t_2= (u-v)^2,~t_3 =w^2,~t_4=(u-v)w
$$
with relations $t_2 t_3 = t_4^2$. We then have 
\begin{align*}
&\C[x,y]^G \cong \C[t_1,\ldots,t_4]/(t_2 t_3 - t_4^2, \frac 14( t_1^2 - t_2)- t_3^n)\\
&\cong  \C[t_1,t_3,t_4]/((t_1^2 -4t_3^n) t_3 - t_4^2).
\end{align*}
We rescale and rename the variables to get this into the standard form of 
$$
\{u^2 + wv^2 + w^{n+1} = 0\} \subset \C^3
$$
called the $D_{n+2}$ singularity. Note that the extra $w$ in the second term makes it different from the $A_n$ singularities.

\smallskip
Looking at the group generators or the equation of the quotient, we see that $n=1$ case simply gives $A_3$ singularity, see Exercise 3.
We will now consider the simplest new case $n=2$, i.e. the $D_4$ singularity
$$
u^2 + wv^2 + w^3=0.
$$

\smallskip
It is easy to see that the only singular point occurs at $(0,0,0)$. This is also apparent from the quotient description, since nonidentity diagonalizable elements of ${\rm SL}(2,\C)$ can not have eigenvalue $1$ and thus act freely away from the origin.
Now we blow up $\C^3$ at $(0,0,0)$ and see what happens. 

\smallskip
In the open chart with coordinates $(z_1,z_2,z_3) =(\frac uv , v, \frac wv)$ we get\footnote{ So having raised a stink about using letters for the variables, we are now doing it. Enjoy!}
$$
0 = u^2 + w v^2 + w^3 = z_1^2 z_2^2  + z_2^3  z_3 + z_2^3 z_3^3 = z_2^2(z_1^2 + z_2 z_3 + z_2 z_3^3)
$$
and the proper preimage is given by $0=z_1^2 + z_2 z_3 + z_2 z_3^3$. Vanishing of partial derivatives gives
$$
z_1 = 0, z_3 +z_3^3 =0, z_2 (1 + 3z_3^2) = 0,
$$
which gives three singular points $(0,0,0),(0,\ii,0),(0,-\ii,0)$. Local analysis of these points show that all three are of the type $A_1$.
We skip the calculations for the two other charts, but we claim that we do not see any new singular points.  
So the $D_4$ singularity blows up to a $\C\P^1$ with three type $A_1$ singular  points on it.

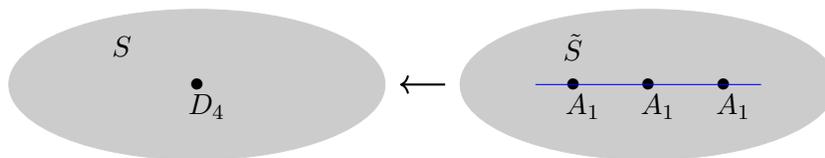
\begin{tikzpicture}
\filldraw[gray!40!white] (3,0) ellipse (2.5 and 1);
\filldraw[gray!40!white] (9,0) ellipse (2.5 and 1);
\node at (3,0){$\bullet$};
\node at (3,0)[anchor=north]{$~~D_{4}$};
\node at (2,.5){$S$};
\node at (10,0){$\bullet$};
\node at (10,0)[anchor=north]{$~~A_{1}$};
\node at (9,0){$\bullet$};
\node at (9,0)[anchor=north]{$~~A_{1}$};
\node at (8,0){$\bullet$};
\node at (8,0)[anchor=north]{$~~A_{1}$};
\node at (8,.5){$\tilde S$};
\draw[thick, <-]( 5.7,0)--(6.3,0);
\draw[blue] (10.5,0)--(7.5,0);
\end{tikzpicture}

\noindent
When we further blow up these three points, we obtain a resolution of singularities $\tilde S\to S$ with four $\CP^1$-s in it which intersect transversally in the $D_4$ Dynkin diagram (one central Riemann sphere intersects three others).
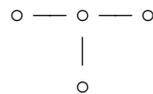
\begin{align*}
\begin{array}{c}
\circ - \hskip -5pt - \circ - \hskip -5pt - \circ \\
\vert \\
\circ 
\end{array}
\end{align*}

{\bf Exercise 1.} Prove that every cyclic subgroup of ${\rm SL}(2,\C)$ can be conjugated to the one in \eqref{cyclic}. 

{\bf Exercise 2.} Verify that the proper preimage of $uv-w^n$ in the chart with coordinates $u, \frac vu, \frac wu$ is smooth.

{\bf Exercise 3.} Check that the $D_3$ case is the same as the $A_3$ case in two ways, by looking at the group action, and by
making a change of variables in $u^2 + w v^2 + w^2=0$.

\section{More resolutions of $ADE$ singularities.}
Let us now discuss the general $D_{n+2}$ case 
$$
S=\{u^2+ wv^2 + w^{n+1} = 0\}\subset \C^3
$$
for $n>2$.

\smallskip
As always, we blow up the origin in $\C^3$. On the open chart with coordinates $\frac uv ,v ,\frac wv$ we get
$$
0=v^2 \left(
\left( \frac uv \right)^2 + v \left(\frac wv\right) + \left(\frac wv\right)^{n+1} v^{n-1}
\right)
$$
with the second factor giving the proper preimage of $S$.
Let us find singular points of $x^2 + yz + z^{n+1}y^{n-1}=0$. We have to solve the equations\footnote{ Yes, algebraic geometry is about solving polynomial equations.}
$$
2x = 0,~ z + (n-1) z^{n+1} y^{n-2} = 0,~y + (n+1) z^{n}y^{n-1}=0,~x^2 + yz + z^{n+1}y^{n-1}=0.
$$
We can multiply the second equation by $(n+1)y$ and the third one by $(n-1)z$, subtract them and get $yz=0$. Since  $n>2$, we use the second and third equations to deduce that both $y$ and $z$ are $0$. So the only singular point in this chart is $(0,0,0)$,
which corresponds to $(0:1:0)$ on the exceptional $\C\P^2$ of $\tilde \C^3\to \C^3$ (we are in the chart $v_1\neq 0$).
For $n>2$, we see that the singularity at this point is of type $A_1$, because we can ignore higher degree terms.
The intersection with the exceptional locus $v=0$ is given by $\frac u v=0$, i.e. we have the line $(0:*:*)$ on the exceptional $\C\P^2$ of $\tilde \C^3\to \C^3$.

\smallskip
Let us look at the $u,\frac vu, \frac wu$ chart. There, the proper preimage is given by
$$
0=1 + \left(\frac wu\right) \left(\frac vu\right)u +  \left(\frac wu\right)^{n+1} u^{n-1}
$$
and the exceptional locus is $u=0$. Therefore, this chart of $S$ does not intersect the exceptional locus and is automatically smooth (one can also check it directly).

\smallskip
The last chart has coordinates $\frac uw,\frac vw,w$. There we get
$$
0=\left(\frac uw\right)^2 w^2 + \left(\frac vw\right)^2w^3 + w^{n+1} = z^2\left(\left(\frac uw\right)^2 + w\left(\frac vw\right)^2 + w^{n-1}\right).
$$
Clearly, we have a $D_{n-2}$ singularity at the origin, which corresponds to the point $(0:0:1)$.

\smallskip
Combining the information from the three charts, we see that a single blowup gives us a map from the proper preimage $\tilde S$ to $S$ so that the exceptional locus is $\C\P^1$ with two singular points on it. One of these points of type $A_1$ and the other is of type $D_{n-2}$. By resolving the $A_1$ singularity, we replace $D_n$ by two transversal $\C\P^1$-s with a $D_{n-2}$ point on one of them. 

\hskip -15pt
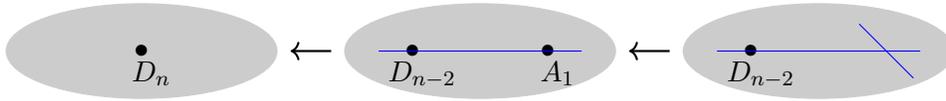
\begin{tikzpicture}[scale = 0.9]
\filldraw[gray!40!white] (3,0) ellipse (2. and .7);
\filldraw[gray!40!white] (8,0) ellipse (2 and .7);
\node at (3,0){$\bullet$};
\node at (3,0)[anchor=north]{$~~D_{n}$};
\node at (9,0){$\bullet$};
\node at (9,0)[anchor=north]{$~~A_1$};
\node at (7,0){$\bullet$};
\node at (7,0)[anchor=north]{$~~D_{n-2}$};
\draw[thick, ->]( 5.8,0)--(5.2,0);
\draw[blue] (9.5,0)--(6.5,0);
\filldraw[gray!40!white] (13,0) ellipse (2 and .7);
\node at (12,0){$\bullet$};
\node at (12,0)[anchor=north]{$~~D_{n-2}$};
\draw[blue] (13.6,0.4)--(14.4,-0.4);
\draw[thick, ->]( 10.8,0)--(10.2,0);
\draw[blue] (14.5,0)--(11.5,0);
\end{tikzpicture}

\noindent
When we repeat the process, skipping details, we eventually get to $D_4$ or $D_3=A_3$ and the resulting dual graph is given by the $D_{n+2}$ Dynkin diagram.
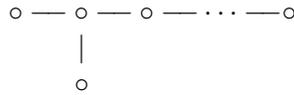
\begin{align*}
\begin{array}{c}
\circ - \hskip -5pt - \circ - \hskip -5pt - \circ  - \hskip -5pt - \cdots  - \hskip -5pt -\circ\\
\vert \hskip 50pt ~\\
\circ \hskip 50pt~\\
\end{array}
\end{align*}

As you may deduce from the title of the section, there are three more cases to consider, which correspond to the remaining three Dynkin diagrams $E_6$, $E_7$ and $E_8$, so that cumulatively these singularities are referred to as $ADE$ surface singularities.\footnote{ They are also called du Val singularities, canonical singularities, simple surface singularities, Kleinian singularities, and rational double points. This should tell you how ubiquitous these singularities are.} The corresponding groups are $\Z/2\Z$ central extensions of the groups of rotational symmetries of the regular tetrahedron, cube and dodecahedron, respectively.

\smallskip
We will not write down the group generators, but will list the equations below.
$$
\begin{array}{ll}
E_6:& u^2 + v^3 + w^4 = 0\\
E_7:& u^2 + v^3 + vw^3 = 0\\
E_8:& u^2 + v^3 + w^5 = 0
\end{array}
$$

\smallskip
{\blue \hrule}
There is another way of looking at the resolutions of $D_{n+2}$ singularities. The group $G$ given by \eqref{Dn} has a cyclic subgroup $G_1 = \Z/2n\Z$. We can first take the quotient of $\C^2$ by $G_1$, then resolve the resulting $A_{2n-1}$ singularity
to get a smooth surface $S_1$, then take the quotient by $G/G_1 \cong \Z/2\Z$ to get a singular surface $S_2$ and then further resolve to get a resolution $S_3$ of $\C^2/G$, see the commutative diagram below.
$$
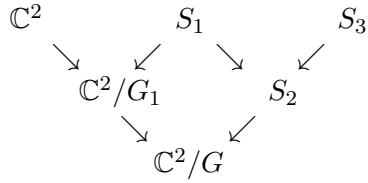
\begin{array}{c}
\C^2\hskip 50pt S_1 \hskip 50pt S_3\hskip 5pt\\
\searrow\hskip 20pt\swarrow\hskip 20pt\searrow\hskip 20pt\swarrow\\
\C^2/G_1\hskip 40pt S_2\hskip 10pt\\
\searrow\hskip 30pt\swarrow\\
\C^2/G\end{array}
$$
Specifically, we have the $A_{2n-1}$ singularity given by 
$$
uv - w^{2n} = 0
$$
and the action of $G/G_1 \cong \Z/2\Z$ which switches $u$ and $v$. The resolution $S_1$ has a chain of $(2n-1)$ Riemann spheres
$$
\circ - \hskip -5pt - \circ - \hskip -5pt - \cdots - \hskip -5pt - \circ - \hskip -5pt - \circ 
$$
which intersect transversely, and the involution of $G/G_1$ switches the opposite spheres. Because we have an odd number of spheres, this involution will preserve the middle $\C\P^1$. It acts by switching two intersection points with two adjacent $\C\P^1$-s. We can call these points $0$ and $\infty$ and observe that every such involution can be written in the form $z\mapsto \frac 1z$, see Exercise 1. Then we have two fixed points $z=\pm 1$. Since these are isolated points, the eigenvalues of the action on the tangent space must be $(-1,-1)$, so the quotient will have an $A_1$ singularity there. Therefore, the image of the chain 
of $(2n-1)$ lines on $S_2$ will be $n$ lines, with the first line having two additional $A_1$ singularities on it. Blowing up these $A_1$ singularities, we get $S_3$ with the $(n+2)$ lines in $D_{n+2}$ configuration.

\smallskip
We can similarly get the $E_6$ singularity as the quotient of the $D_4$ singularity by a $\Z/3\Z$ action. Then if you resolve
the $D_4$ singularity to get a surface with a $D_4$ configurations of projective lines, the group can be seen rotating the 
three ``outside" lines while fixing the middle one. By Exercise 2, any order three automorphism of $\C\P^1$ can be 
written in the form $z\mapsto \ee^{2\pi\ii/3}z$, so it will have two fixed points $0$ and $\infty$. It is then possible to compute that the  quotient will be the $A_2$ singularity, so we get $S_1\to S$ with $E_6$ singularity replaced by two lines with two $A_2$ singular points on one of them. 

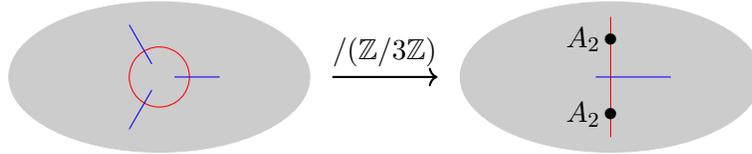
\begin{tikzpicture}
\filldraw[gray!40!white] (3,0) ellipse (2. and 1);
\draw[red] (3,0) ellipse (.4 and .4);
\draw[blue](3.2,0)--(3.8,0);
\draw[blue](2.9, {0.1*sqrt(3)})--(2.6,{0.40*sqrt(3)});
\draw[blue](2.9, {-0.1*sqrt(3)})--(2.6,{-0.40*sqrt(3)});
\node at (6,0)[anchor=south]{$/(\Z/3\Z)$};
\draw[thick, <-]( 6.7,0)--(5.3,0);
\filldraw[gray!40!white] (9,0) ellipse (2 and 1);
\draw[red] (9,.8) -- (9,-.8);
\draw[blue](8.8,0)--(9.8,0);
\node at (9,.5) {$\bullet$};
\node at (9,.5) [anchor=east] {$A_2$};
\node at (9,-.5) {$\bullet$};
\node at (9,-.5) [anchor=east] {$A_2$};
\end{tikzpicture}

\noindent
When we resolve them, we get the $E_6$ diagram below, with the colors matching the colors of the above picture.
$$
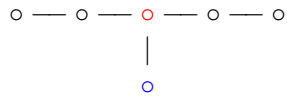
\begin{array}{c}
\circ - \hskip -5pt - \circ - \hskip -5pt -{\red \circ}  - \hskip -5pt - \circ  - \hskip -5pt - \circ\\
\vert\\
{\blue \circ} 
\end{array}
$$

Similarly, the $E_7$ singularity is a $\Z/2\Z$ quotient of the $E_6$ singularity.\footnote{ The set of  eight vertices of the cube naturally splits into two sets of four vertices of a regular tetrahedron, and the group of rotational symmetries of the tetrahedron is an index two subgroup of that of the cube.}  After resolving the singularity, the action on the $E_7$ configuration of lines is by switching the opposing ends. So we get four lines, $C_1$, $C_2$, $C_3$, $C_4$, connected in a chain, and there are three $A_1$ singularities: one one $C_1$, one on the intersection of $C_1$ and $C_2$ and one elsewhere on $C_2$.

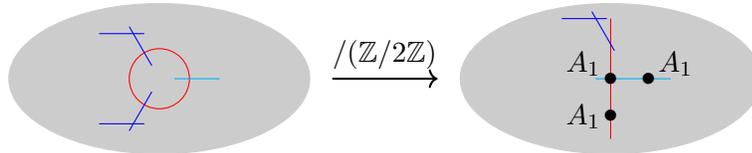
\begin{tikzpicture}
\filldraw[gray!40!white] (3,0) ellipse (2. and 1);
\draw[red] (3,0) ellipse (.4 and .4);
\draw[cyan](3.2,0)--(3.8,0);
\draw[blue](2.9, {0.1*sqrt(3)})--(2.6,{0.40*sqrt(3)});
\draw[blue](2.9, {-0.1*sqrt(3)})--(2.6,{-0.40*sqrt(3)});
\draw[blue] (2.2,.6) -- (2.8,.6);
\draw[blue] (2.2,-.6) -- (2.8,-.6);
\node at (6,0)[anchor=south]{$/(\Z/2\Z)$};
\draw[thick, <-]( 6.7,0)--(5.3,0);
\filldraw[gray!40!white] (9,0) ellipse (2 and 1);
\draw[red] (9,.8) -- (9,-.8);
\draw[cyan](8.8,0)--(9.8,0);
\node at (9,.0) {$\bullet$};
\node at (9,.2) [anchor=east] {$A_1$};
\node at (9.5,.0) {$\bullet$};
\node at (9.5,.2) [anchor=west] {$A_1$};
\node at (9,-.5) {$\bullet$};
\node at (9,-.5) [anchor=east] {$A_1$};

\draw[blue](6.15+2.9,{.2+ 0.1*sqrt(3)})--(6.15+2.6,{.2+0.40*sqrt(3)});
\draw[blue] (6.15+2.2,.2+.6) -- (6.15+2.8,.2+.6);
\end{tikzpicture}

\noindent
When we resolve these $A_1$ singularities, we get the $E_7$ diagram
$$
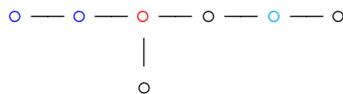
\begin{array}{c}
{\blue \circ} - \hskip -5pt - \,{\blue \circ} - \hskip -5pt -\,{\red \circ}  - \hskip -5pt - \circ  - \hskip -5pt -{\cyan \circ} - \hskip -5pt - \hskip 2pt \circ \\
\vert\hskip 21pt~\\
\circ \hskip 21pt~
\end{array}
$$
with the matching colors.

\smallskip
Unfortunately, the group of symmetries of the regular dodecahedron is not solvable, so one can not obtain the $E_8$ resolution by this process. We can however blow up the origin in $\C^3$ and look at the proper preimage of the surface $u^2+ v^3 + w^5=0$ in it. We get in the $\frac uw, \frac vw, w$ coordinate chart
$$
0= w^2 \left(
\left (\frac uw \right)^2 + \left (\frac vw \right)^3 w + w^3
\right)
$$
which is the $E_7$ singularity. We see in Exercise 3 that the proper preimage has no singularities in other coordinate charts.
The intersection with the special fiber is a line $\frac uw = u = 0$, and one can see, with some effort, that after resolving that $E_7$ singularity we get the $E_8$ diagram.
$$
\begin{array}{c}
\circ - \hskip -5pt - \circ - \hskip -5pt - \circ  - \hskip -5pt - \circ  - \hskip -5pt -\circ  - \hskip -5pt -\circ - \hskip -5pt - \hskip 2pt \circ \\
\vert\hskip 48pt~\\
\circ \hskip 48pt~
\end{array}
$$

{\bf Exercise 1.}
 Prove that every involution of $\C\P^1$ that sends $0$ to $\infty$ can be written in the form $z\mapsto \frac cz$ for some nonzero constant $c$. Then argue that by scaling $z$ we can rewrite it as $z\mapsto \frac 1z$.

{\bf Exercise 2.} Prove that every order three automorphism of $\C\P^1$ can be written as $z\mapsto \ee^{2\pi\ii/3}z$ in some coordinates. \emph{Hint:} Consider the corresponding element in ${\rm PGL}(2,\C)$ and its eigenvalues (up to common scaling).

{\bf Exercise 3.} Check that the proper preimage of $u^2+v^3 + w^5=0$ has no singular points in the coordinate charts with coordinates $u,\frac vu,\frac wu$ and $\frac uv, v, \frac wv$.

\section{Overview of intersection theory. Surfaces.}\label{sec.int}
In this section we will talk about intersection theory. Even for these notes, this section is unusually hand-wavy, with many of the claims partially or totally unsubstantiated. Sorry, not sorry!

\smallskip
Let $X$ be a smooth complex algebraic variety or a complex manifold. We can view it as an orientable real manifold and associate to it two 
collections of abelian groups, namely its integral simplicial homology groups $H_i(X,\Z)$ and cohomology groups
$
H^{i}(X,\Z)
$
for $0\leq i\leq \dim_\R X = 2\dim_\C X$. 
All of these are finitely generated for algebraic varieties $X$. Moreover, $H^*(X,\Z)$ is an associative graded, super-commutative\footnote{ Super-commutativity means that for $a\in H^i(X,\Z)$ and $b\in H^j(X,\Z)$ we have 
$\alpha\cup \beta = (-1)^{ij}\beta\cup\alpha\in H^{i+j}(X,\Z)$.} ring with identity with the multiplication operation cup-product $\cup$. The homology $H_*(X,\Z)$ is a graded module over $H^*(X,\Z)$.
The groups $H^*$ and $H_*$ are also related by the Universal Coefficient Theorem, but we will not use it in these notes.

\smallskip
The nicest case occurs when $X$ is a smooth and compact (for example, projective) algebraic variety.
In this case we have the Poincar\'e duality
$$
H^i(X,\Z) \stackrel{\rm natural}\cong H_{2\dim_\C X - i}(X,\Z).
$$
Under Poincar\'e duality, $1\in H^0(X,\Z)$ corresponds to the class of $[X]$ in $H_{2\dim_\C X }(X,\Z)$, but the more interesting statement is that the class of the point in $H_0(X,\Z)$ corresponds to a special element ${\rm P.D.(point)}$ in  $H^{\rm top}(X,\Z)=H^{2\dim_\C X}(X,\Z)$. In fact, we have 
$$
H^{\rm top}(X,\Z) = \Z\, {\rm P.D.(point)},
$$
and this gives us a concept of the {\blue intersection number}. 

\smallskip
\begin{defn}
Let $Y$ and $Z$ be algebraic subvarieties of $X$, of complementary dimension, i.e. 
$$
\dim_\C Y+ \dim_\C Z = \dim_\C X.
$$
Consider the product of Poincar\'e duals 
$$
{\rm P.D}([Y]) \cup {\rm P.D}([Z ]) =\alpha\, {\rm P.D.(point)} \in H^{\rm top}(X,\Z).
$$
The integer $\alpha$ is called the intersection number of $Y$ and $Z$ and will be denoted by $Y\cdot Z$.
\end{defn}

\smallskip
\begin{rem}
If $\,Y$ and $Z$ intersect transversely at $k$ points, then $Y\cdot Z = k$ (in particular, $Y\cdot Z=0$ if $Y\cap Z=\emptyset$). More generally, if $Y\cap Z$ is a finite set, then $Y\cdot Z$ is a sum of local contributions, each of which is positive. However, the definition makes sense even if $Y\cap Z$ is positive-dimensional, or even $Y\subseteq Z$! In these cases, we just know that $Y\cdot Z$ is integer but can make no claims about the sign of it. Philosophically, we can deform $Y$ as a real cycle to some $Y'$ with the same homology classes $[Y']=[Y]$ so that $Y'$  intersects $Z$ transversely, and we add up $(\pm 1)$ at each intersection point, depending on the orientations.
\end{rem}

\smallskip{\blue \hrule}
The simplest nontrivial example is provided by $X=\C\P^2$. If $Y$ and $Z$ are complex curves in $X$, we have $\dim _CY+ \dim_\C Z = \dim_\C X$, so their intersection product makes sense. The cohomology groups of $X$ are given by
$$
H^{1}(X,\Z) = H^{3}(X,\Z) = 0,~ H^{0}(X,\Z)\cong H^2(X,\Z)\cong H^4(X,\Z)  \cong \Z
$$
with the generators of the latter three groups given by $1$, the Poincar\'e dual of the class $[l]$ of a line, and the Poincar\'e dual of the class of a point, and we have ${\rm P.D.}[l]\cup {\rm P.D.}[l] = {\rm P.D.(point)}$\footnote{ This is a very fancy way of saying that two lines in $\C\P^2$ intersect transversely at a point.}
The class of the curve $Y$ is $(\deg Y)[l]$, and therefore $Y\cdot Z = (\deg Y) (\deg Z)$. This is essentially the 
Bezout's Theorem \ref{Bezout} that we discussed rigorously in Sections \ref{sec.B1} and \ref{sec.B2},  although to make the connection precise we would need to talk about the intersection multiplicities in the non-transversal case.

\smallskip
Now let us consider the blowup of $\C\P^2$ at one point. We get $X\to \C\P^2$ with $E$ the exceptional divisor on $X$.

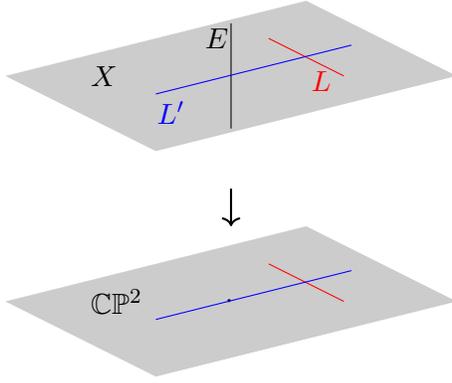
\begin{tikzpicture}
\fill[gray!40!white] (7,0.0)--(11,1)--(13,0)--(9,-1)--(7,0);
\draw (10,.7)--(10,-.7);
\draw (8,0) node [anchor=west]{$X$};
\draw[thick,->] (10,-1.5)--(10,-2);
\fill[gray!40!white] (7,-3)--(11,-2)--(13,-3)--(9,-4)--(7,-3);
\draw (10,-3) node [anchor=west]{$\hskip -6pt \cdot $};
\draw (8,-3) node [anchor=west]{$\C\P^2 $};
\draw[red] (10.5,-2.5)--(11.5,-3.0);
\draw[blue] (9,-3.24)--(11.6,-2.59);
\draw[red] (10.5,0.5)--(11.5,0);
\draw[blue] (9,-0.24)--(11.6,0.41);
\node at (10.1,.5) [anchor=east]{$E$};
\node[blue] at (9.2,.-0.2) [anchor=north]{$L'$};
\node[red] at (11.2,0.2) [anchor=north]{$L$};
\end{tikzpicture}

We have the following curves of interest on $X$: $E$, $L$, $L'$. The curve $E$ was already discussed, the curve $L$ is the preimage of a line in $\C\P^2$ that does not pass through the point we blew up and the curve $L'$ is the proper preimage of line that does pass through a point of blowup.

\smallskip
The key observation is that $L\sim L'+E$, in the sense that $[L] = [L']+[E]$ in $H_2(X,\Z)$. Indeed, more generally there is a map of Weil divisors to $H_2(X,\Z)$ which passes through the class group. Another way of saying the same thing is that a divisor of a rational function is always homologically trivial. And there is a rational function (the ratio of the equations of the equations of the two lines)
which has simple zeros at $E$ and $L'$ and poles at $L$.\footnote{ In this particular case can think of the family of smooth preimages of  lines in $\C\P^2$ degenerating into a singular curve $E\cup L'$.}
We also have 
\begin{equation}\label{Xproduct}
L^2 = 1,~L\cdot E = 0, L'\cdot E = 1, L\cdot L' = 1
\end{equation}
based on the transversal intersections of the corresponding curves. This implies that 
$$
0 = L\cdot E = (L'+E)\cdot E = L'\cdot E + E\cdot E = 1 + E^2,
$$
so we get 
$$
E^2 = -1.
$$
What? This certainly looked weird to me the first time I saw it. How can a curve intersect itself negatively?

\smallskip
Actually, there is no problem. This curve $E$ can not be deformed holomorphically, but you can deform it as a real 
$2$-cycle on $X$ to some $\tilde E$ which intersects $E$ at one point transversely with the ``wrong" orientation. In general, for a smooth curve $C$ on a smooth projective surface $X$ the self-intersection $C^2$ is equal to the degree of the normal bundle to $C$ in $X$, where by the degree of a line bundle I mean the degree of the corresponding Weil divisor class.

\smallskip{\blue \hrule}
In a footnote to Section \ref{sec.ADE1} we indicated that the trees of lines in the resolutions of ADE singularities have an intersection matrix which is negative of the Cartan matrix of the corresponding Dynkin diagram. Since these diagrams are simply-laced, this simply means checking that for each of these lines $E$ there holds $E^2=-2$. We will not do this in full generality but will address the case of the resolution of $A_1$ singularity. In this case, we have a surface $S$ in $\C^3$ given by
$uv - w^2 = 0$, and we blow up $\C^3$ at the origin. The proper preimage $X$ of $S$ is then smooth and the map
$$
\pi:X\to S
$$
contracts a single line $E$ to the origin.

\smallskip
\begin{rem}
If you object to doing intersection theory on a non-compact surface $X$, you are not alone. To avoid this issue, we will compactify $\C^3\subset \C\P^3$ and consider the closure of $S$ in it, given by
$$
uv-w^2 =0
$$
in homogeneous coordinates $(t:u:v:w)$ of $\C^3$. We end up adding a smooth conic in the plane $t=0$, and it 
is easy to see that there is still only one singular point and the blowup of $\C\P^3$ in it still induces a resolution of singularities. We will still use $S$ and $X$ as our notations. Alternatively, one can argue that all we are trying to do is to understand the degree of the normal bundle of $E$ in $X$, and this does not require any compactification.
\end{rem}

\smallskip
Consider two lines $C_1$ and $C_2$ on $S$ given by $u=w=0$ and $v=w=0$ respectively. Their proper preimages on 
$X$ intersect $E$ transversely at one point each. 

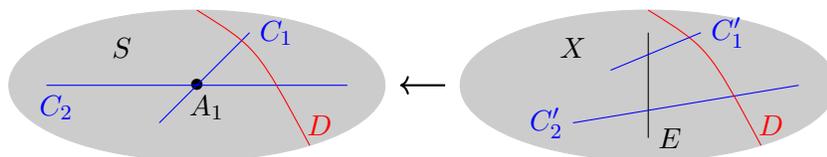
\begin{tikzpicture}
\filldraw[gray!40!white] (3,0) ellipse (2.5 and 1);
\filldraw[gray!40!white] (9,0) ellipse (2.5 and 1);
\node at (3,0){$\bullet$};
\node at (3,0)[anchor=north]{$~~A_1$};
\node at (2,.5){$S$};
\node at (8,.5){$X$};
\draw[thick, <-]( 5.7,0)--(6.3,0);
\draw[blue] (3.7,.7)--(2.5,-.5);
\node[blue] at (3.7,.7) [anchor = west] {$C_1$};
\draw[blue] (5,0)--(1,0);
\node[blue] at (1,0) [anchor = north] {$~~C_2$};
\draw[red] (3,1).. controls(3.8,.5) .. (4.5,-.8);
\node[red] at (4.5,-.8) [anchor=south]{$~~D$};
\draw (9,.7)--(9,-.7);
\node at (9,-.7)[anchor= west]{$E$};
\draw[blue] (11,0)--(8,-.5);
\node[blue] at (8,-.5)[anchor = east]{$C_2'$};
\draw[red] (9,1).. controls(9.8,.5) .. (10.5,-.8);
\node[red] at (10.5,-.8) [anchor=south]{$~~D$};
\draw[blue] (9.7,.7)--(8.5,0.2);
\node[blue] at (9.7,.7) [anchor = west] {$C_1'$};
\end{tikzpicture}

\noindent
We also see that the rational function $\frac wt$ has divisor 
$$
[C_1' ]+ [E] + [C_2'] - [D]
$$
where $D$ is the preimage of the conic at infinity $\{uv-w^2=t=0\}$. Therefore, we 
have $[E]=[D] - [C_1']-[C_2']$ in $H_2(X,\Z)$
and
$$
E^2 = D\cdot E - C_1 \cdot E - C_2 \cdot E = 0 - 1 - 1 = -2
$$
since $D\cap E = \emptyset$ and $C_i$ intersect $E$ transversely at a single point.

\smallskip
\begin{rem}
It is possible to do intersection theory purely algebraically, over any field $k$, even if $k$ is not algebraically closed or has positive characteristics. One needs to replace the (even) homology groups $H_*$ with the so-called Chow groups $A_*(X)$. As a matter of fact, $A_*(X)$ provide more refined invariants, compared to $H_*$. A very good, but also quite advanced reference for it is the Fulton's book \cite{Fulton}.
\end{rem}

{\bf Exercise 1.} Use \eqref{Xproduct} to check that on the blowup $X$ of a point in $\C\P^2$ there holds $(L')^2=0$. What is the geometric meaning of it?

{\bf Exercise 2.} Let $C$ be a smooth curve on a smooth projective surface $S$ and let $X\to S$ be the blowup of $S$ at a point $p\in S$. Denote by $C'$ the proper preimage of $C$ in $X$. Prove that $(C')^2=C^2-1$. This generalizes the example of Exercise 1.

{\bf Exercise 3.} Let $X$ be a blowup of $\C\P^2$ in $5$ points $p_1,\ldots, p_5$ in general position. Prove that 
the proper preimage of any line through two of these points has self-intersection $(-1)$. Prove that the proper preimage of the unique conic that passes through all five points also has self-intersection $(-1)$.

\section{Rational surfaces. Del Pezzo surfaces.}\label{sec.dP}
We will now discuss arguably the simplest class of complex algebraic surfaces, namely the {\blue rational} surfaces.

\smallskip
\begin{defn}
A complex algebraic surface $X$ is called {\blue rational} if and only if it is birational to $\C\P^2$ or, equivalently, the field of rational functions on $X$ is isomorphic to the field of rational functions in two variables.
\end{defn}

\smallskip
This definition makes sense in arbitrary dimension, to give us a notion of a rational algebraic variety. In some sense, rational varieties are what we would wish all varieties to be: their points, at least most of them, are precisely encoded by free variables. In higher dimension, it is a highly nontrivial problem to determine if a variety is rational, but in dimension two it is usually quite manageable due to the following theorem of Castelnuovo, which we will not attempt to prove.

\smallskip
\begin{thm}
Let $X$ be a smooth projective complex algebraic surface. Then $X$ is rational if and only if the square of the canonical line bundle $K_X^{\otimes 2}\to X$ has no nonzero holomorphic sections.
\end{thm}

As was mentioned before, every two projective algebraic surfaces which are birational to each other are obtained by successive blowups of points and blowdowns. Thus, it is natural to define {\blue minimal rational surfaces}  as smooth projective rational surfaces which are not isomorphic to a blowup of any other surface. Since blowdown decreases the rank of the second cohomology group, one can not do them indefinitely, so every rational surface is obtained by a repeated blowup of a minimal one.

\smallskip
Prominent examples of minimal rational surfaces are {\blue Hirzebruch surfaces.}

\smallskip
\begin{defn}\label{Hirz}
Let $n$ be a nonnegative integer and let $W=\mathcal O\oplus \mathcal O(-n)$ be a rank two vector bundle over $\C\P^1$.\footnote{ A theorem due to Grothendieck states that any vector bundle over $\C\P^1$ is a direct sum of line bundles. This is a very peculiar property of $\C\P^1$.} 
We define the Hirzebruch surface 
$$H_n:= 
\P W
$$
as the set of lines through the origin in the fibers of $W$.
\end{defn}

\smallskip
Each of the two summands $\mathcal O$ and $\mathcal O(-n)$ of $W$ gives a line in the fiber of $W$ and thus a point in the fiber of $\P W$. As fibers vary, we get two smooth curves in $H_n$ which are sections of $H_n\to \C\P^1$ in the usual sense of the word. We call these curves $S_0$ and $S_{-n}$. In addition, we will be interested in the fiber $F_0$ and $F_\infty$ over the two points of $\C\P^1$.

\hskip 40pt
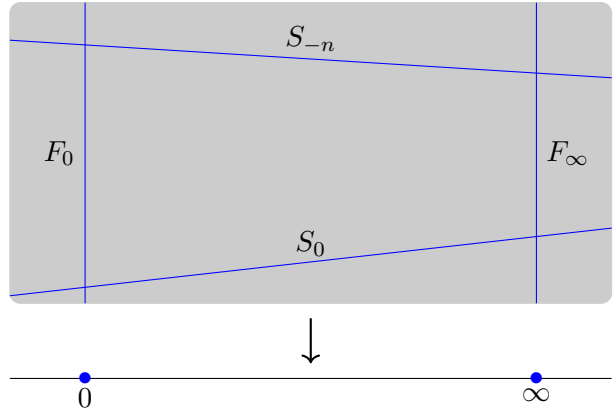
\begin{tikzpicture}
\filldraw[gray!40!white][rounded corners] (0,0) rectangle (8,4);
\draw [blue] (1,0)--(1,4);
\node[] at (1,2)[anchor=east]{$F_0$};
\draw [blue] (7,0)--(7,4);
\node[] at (7,2)[anchor=west]{$F_\infty$};
\draw [blue](0,.1)--(8,1);
\node[] at (4,.5)[anchor=south]{$S_0$};
\node[] at (4,3.2)[anchor=south]{$S_{-n}$};
\draw [blue](0,3.5)--(8,3);
\draw [thick,->] (4,-0.2)--(4,-0.8);
\draw (0,-1)--(8,-1);
\node [blue] at (1,-1) {$\bullet$};
\node [] at (1,-1) [anchor=north]{$0$};
\node [blue] at (7,-1) {$\bullet$};
\node [] at (7,-1) [anchor=north]{$\infty$};
\end{tikzpicture}

\noindent
It is easy to see (Exercise 2) that 
\begin{equation}\label{Hnminus}
H_n \setminus \left(
S_0\cup F_0
\right)\cong \C^2.
\end{equation}
One can show that every Weil divisor on $\C^2$ can be given as a divisor of a rational function. Therefore, the class group of $H_n$ is generated by $[S_0]$ and $[F_0]$. Let us find their intersection numbers.

\smallskip
\begin{prop}
The intersection pairings of $S_0$ and $F_0$ on  $H_n$ are given by 
$$
F_0^2 = 0 ,~ F_0 S_0 = 1,~S_0^2 = -n.
$$
\end{prop}

\begin{proof}
The rational function $\frac {x_0}{x_1}$ on $\C\P^1$ gives a rational function on $H_n$ with divisor
$[F_0] -[F_\infty]$. We thus get 
$$
F_0^2 = F_0 F_\infty = 0,
$$ 
since $F_0\cap F_\infty =\emptyset$.
The curves $F_0$ and $S_0$ intersect transversally at a single point, thus $F_0 S_0=1$.

\smallskip
The computation for $S_0^2$ is a bit more subtle. We claim that there is a rational function on $H_n$ which has divisor
\begin{equation}
[S_0]-[S_{-n}]+n[F_0].
\end{equation}
This would imply that $[S_0] = [S_{-n}] - n[F_0]$ in ${\rm Cl}(H_n)$, so 
$$S_0^2 = S_0 S_{-n}-n S_0 F_0 = 0-n = -n.$$
To construct such rational function, we write $\mathcal O(-n)\to \C\P^1$ in coordinate charts as being glued from two copies
of $\C\times \C\to \C$ over the two charts $\C$ on $\C\P^1$ via
$$
(t_0, (1:\frac {x_1}{x_0})) \sim ( (\frac {x_0}{x_1})^n t_1, (\frac {x_0}{x_1}:1)).\footnote{ If you are like me and struggle with signs, one way to see that you should have $n$ and not $(-n)$ in this formula
is by noticing that section $t_0=1$ in the first chart becomes $t_1=  (\frac {x_0}{x_1})^{-n}$ in the second chart and thus does not extend to a section on all of $\C\P^1$. This is consistent with $\mathcal O(-n)$ not having nonzero sections.
If a different sign were chosen, we would get $\mathcal O(n)$ which does have such sections.}
$$
Then $W\to \C$ is glued from two copies of $\C^2\times \C\to \C$ via
$$
(s_0,t_0, (1:\frac {x_1}{x_0})) \sim (s_1, (\frac {x_0}{x_1})^nt_1, (\frac {x_0}{x_1}:1)).
$$
Then $H_n=\P W$ is glued from two copies of $\C\P^1\times \C\to \C$ via
$$
((s_0:t_0) ,(1:\frac {x_1}{x_0})) \sim ((s_1: (\frac {x_0}{x_1})^nt_1), (\frac {x_0}{x_1}:1)).
$$
The curve $S_0$ is defined by $t_0=0$ in the first chart and $t_1=0$ in the second chart. Similarly, $S_{-n}$ is $s_0=0$ in the first chart and $s_1=0$ in the second chart. The curve $F_0$ is the complement of the first chart and is defined by 
$\frac {x_0}{x_1}=0$ in the second chart.

\smallskip
The rational function
$$
\frac {t_0}{s_0} = \frac {t_1}{s_1} (\frac{ x_0}{x_1})^{n}
$$
has divisor $[S_{0}]-[S_{-n}]$ in the first chart and $[S_{0}]-[S_{-n}] +n [F_0]$ in the second.
\end{proof}

\smallskip
The following result, which we will not prove, describes all minimal rational surfaces.
\begin{thm}
Minimal rational surfaces are either $\C\P^2$ or Hirzebruch surfaces $H_n$ for $n=0$ or $n\geq 2$.
\end{thm}

\smallskip
\begin{rem}
The Hirzebruch surface $H_0$ is simply $\CP^1 \times \C\P^1$. The reason that the Hirzebruch surface $H_{1}$ is not included in the list is that it is actually isomorphic to 
the blowup of $\C\P^2$ at one point. The Hirzebruch surface $H_2$ is the resolution of the $A_1$ singularity of the singular quadric
$x_0x_1 = x_2^2$ in $\C\P^3$.
\end{rem}

\smallskip
{\blue \hrule}
Some of the most fascinating examples of rational surfaces are the so-called {\blue del Pezzo} surfaces. Recall (see Definition \ref{ample}) that  a line bundle $L\to X$ is called ample if global sections of some positive tensor power of it define an embedding of $X$ into a projective space.

\smallskip
\begin{defn}\label{defdp}
A smooth projective surface $X$ is called a del Pezzo surface, if the anticanonical line bundle $K_X^\vee \cong \wedge^2 TX \to X$ is ample.
\end{defn}

The simplest example of a del Pezzo surface is $\C\P^2$. The anticanonical line bundle on $\C\P^2$ is $\mathcal O(3)$, its sections are cubic polynomials in the homogeneous coordinates, and they define an embedding of $\C\P^2$ into $\C\P^9$, a particular case of the Veronese embedding. So the anticanonical line bundle is very ample. Similarly, the anticanonical line bundle on $\C\P^1\times \C\P^1$ is very ample.

\smallskip
It can be shown that all del Pezzo surfaces are either $\C\P^1 \times \C\P^1$ or blowups of at most $8$ distinct points $p_1,\ldots,p_k$
on $\C\P^2$, provided that these points are in a sufficiently general position (for example, three points can not be collinear).
An interesting feature of a del Pezzo surface is its set of $(-1)$-curves, which are the smooth curves of genus zero that can be blown down to get another del Pezzo surface. We will comment on these below.

\begin{itemize}
\item $k=1.$ Blowing up one point on $\C\P^2$, we get the Hirzebruch surface $H_1$. It has a unique $(-1)$-curve, namely the exceptional curve of the blowdown morphism.

\item $k=2.$ When we blow up two points $p_1$ and $p_2$ on $\C\P^2$, we get three $(-1)$-curves: two exceptional curves $E_1$ and $E_2$ that blow down to $p_1$ and $p_2$ respectively, and the proper preimage $l$ of the line $\overline{p_1p_2}\subset \C\P^2$. We can blow down $l$, and the resulting surface turns out to be isomorphic to $\C\P^1\times \C\P^1$. The reader can try to prove it now, or wait until we talk about toric varieties in Section \ref{sec.toric}. 

\item $k=3.$ Since points can not be collinear, we might as well choose to blow up $(1:0:0)$, $(0:1:0)$ and $(0:0:1)$. 
There are now six $(-1)$-curves in $X$, connected to each other like vertices in a hexagon. They are  three exceptional curves and three proper preimages of the coordinate lines. In fact, we have seen this surface in Exercise 2 of Section \ref{sec.blowup}.

\item $k=4.$ There is only one such surface, since every four points on $\C\P^2$, no three of which are collinear, can be
moved to $(1:0:0)$, $(0:1:0)$, $(0:0:1)$ and $(1:1:1)$ by the action of ${\rm PGL}(3,\C)$. This surfaces has 
ten $(-1)$-curves: $4$ exceptional curves and $6$ proper preimages of the lines through two out of $4$ points. 

\item $k=5.$ There is now a two-dimensional family of surfaces, which we looked at in Exercise 3 of Section \ref{sec.int}.
There are $16$  $(-1)$-curves: $5$ exceptional curves of the blowup, $10$ proper preimages of lines and one proper preimage of the conic through $p_1,\ldots, p_5$. Sections of the anticanonical line bundle on $X$ can be identified with 
cubic polynomials in the homogeneous coordinates of $\C\P^2$ which are zero on $p_1,\ldots, p_5$. This is a five-dimensional space, and it gives an embedding of $X$ into $\C\P^4$, so that the ideal of $X$ is generated by two degree two equations.

\item $k=6.$ This case is a bit more famous than the others. There are $27$ $(-1)$-curves of $X$, which come from
$6$ exceptional curves, $15$ proper preimages of lines and $6$ proper preimages of conics through $5$ out of six points we blow up. All such surfaces are embedded into $\C\P^3$ by the sections of $K_X^\vee$ and the image is given by one cubic equation. The $(-1)$-curves become lines in $\C\P^3$ in this embedding. Conversely, every smooth cubic surface in $\C\P^3$ is del Pezzo, has $27$ lines on it, and any collection of pairwise disjoint lines on it can be contracted to get
$\C\P^2$. We do not prove these classical results, but we will see how the expected number of lines can be deduced using calculations in the cohomology ring of the Grassmannian of lines in Exercise 3 of Section \ref{sec.ms}.

\item $k=7.$ In this case, the surface $X$ always has an involution, so that the quotient is $\C\P^2$ and the map
$X\to \C\P^2$ is ramified over a smooth curve of degree $4$. There are $56$ $(-1)$-curves on $X$. 
In addition to the $7$ exceptional curves, proper preimages of $21$ lines through two of the blowup points and 
$21$ conics through five of the points, there are also proper preimages of $7$ cubic curves in $\C\P^2$ which pass through $6$ out of seven points of the blowup and are singular at the $7$-th point.
The anticanonical line bundle is not very ample, i.e. sections of $K_X^\vee$ do not define an embedding.

\item $k=8.$ There are $240$ $(-1)$-curves. Specifically, there are $8$ exceptional curves, $28$ and $56$ proper preimages of lines and conics respectively and $64$ proper preimages of cubic curves which are singular at one of the points and pass through six points 
of the blowup. Furthermore, there are proper preimages of $56$ degree four curves which are singular at three points and pass through the remaining five points, and $28$ degree five curves which are singular at six points and pass through the remaining two. The anticanonical line bundle is not very ample.

\end{itemize}

\smallskip
\begin{rem}
Another interesting family of surfaces is obtained by blowing up $9$ points on $\C\P^2$ which are the intersection
points of two smooth cubic curves $f_1=0$ and $f_2=0$. The resulting surface is not del Pezzo, but its anticanonical line bundle defines  a map to $\C\P^1$. The generic fiber of this map is an elliptic curve, and these surfaces are called {\blue rational elliptic surfaces}.
\end{rem}

{\bf Exercise 1.} Prove that for any vector bundle $W\to X$ and a line bundle $L\to X$ we have 
$$
\P(W\otimes L) \cong \P W.
$$
Conclude that the projectivization of $\mathcal O(a)\oplus\mathcal O(b)\to \C\P^1$
is isomorphic to the Hirzebruch surface $H_{|a-b|}$.

{\bf Exercise 2.} Show that the complement $H\setminus F_0 $ is isomorphic to $\C\P^1 \times \C$. Use it to prove that \eqref{Hnminus}.

{\bf Exercise 3.} A $(-1)$-curve in a smooth surface $X$ is, by definition, isomorphic to $\C\P^1$ and satisfies $E^2=-1$.
Use adjunction formula to prove that it satisfies $K\cdot E= -1$ where $K$ is the Weil divisor that corresponds to the canonical line bundle $K_X\to X$.

\section{Kodaira dimension and classification of algebraic surfaces.}
We will now discuss a very important invariant of algebraic varieties known as the {\blue Kodaira dimension}.

\smallskip
Let $X$ and $Y$ be compact smooth $n$-dimensional algebraic varieties which are birational to each other, i.e. they contain isomorphic Zariski closed subsets $X\supseteq U_X\cong U_Y\subseteq Y$. Because $Y$ is compact, there exists an open subset $V\subseteq X$ that contains $U_X$ such that
\begin{itemize}
\item  ${\rm codim}_\C(X\setminus V) \geq 2$.

\item There is a birational holomorphic map $\mu:V\to Y$. 
\end{itemize}
This is often phrased as ``birational equivalences are isomorphisms  away from a set of codimension one and are morphisms away from a set of codimension two". We will not prove this statement,\footnote{ It is related to the valuative criterion of properness, \cite{Hartshorne}.} but will illustrate it by an example.

\smallskip
The algebraic surface $\C\P^2$ and its blowup at a point $p$ are birational to each other, because they contain isomorphic Zariski open subsets, namely the complement of the point and its preimage.
If we take $Y$ to be the blowup of $\C\P^2$ and $X$ to be $\C\P^2$, then $V= \C\P^2\setminus\{p\}$ maps to $Y$. On the other hand, if we take $X$ to be the blowup of $\C\P^2$ and $Y$ to be $\C\P^2$, then we can take $V$ to be all of $X$,
although the holomorphic map $V\to Y$ is no longer an inclusion. 

\smallskip
Now let $w\in \Gamma(Y,\Lambda^n TY^{\vee})$ be a global holomorphic $n$-form on $Y$. When we pull it back to $V$ we get
$\mu^* w \in  \Gamma(V,\Lambda^n TX^{\vee})$, a holomorphic $n$-form on $V$. We now invoke a statement knowns as the Hartogs' Lemma: a holomorphic section of a line bundle defined outside of the locus of complex codimension $2$ or higher can be uniquely extended to a holomorphic section.\footnote{ It is often stated for functions rather than sections of line bundles, but there is no difference locally.} Thus, our assumption of $V$ implies that $\mu^*w$ can be extended to an element of $\Gamma(X,\Lambda^n TX^{\vee})$. This gives us a pullback-then-extension linear map 
$$
\Gamma(Y,\Lambda^n TY^{\vee})\to \Gamma(X,\Lambda^n TX^{\vee}).
$$
Naturally, we also have the pullback-then-extension map in the other direction
$$
\Gamma(X,\Lambda^n TX^{\vee})\to \Gamma(Y,\Lambda^n TY^{\vee})
$$
and these two maps are inverses of each other, because the compositions (in the appropriate order) preserve the forms on $U_X$ and $U_Y$.

\smallskip
Recall that the line bundle $\Lambda^n TX^{\vee}$ is called the canonical line bundle of $X$ and is denoted by $K_X$. 
A similar argument works for any positive tensor power\footnote{ Just because we can't integrate expressions like 
$g(z) \,(dz_1\wedge \cdots \wedge dz_n)^{\otimes m}$, it doesn't mean that they don't make sense. They can also be pulled back for $m\geq 0$.} of $K_X$ so we see that 
$$
\Gamma(X,K_X^{\otimes m})
$$
is an invariant of birational equivalences between smooth compact algebraic manifolds.
Moreover, the graded ring
$$
R = \bigoplus_{m\geq 0}\Gamma(X,K_X^{\otimes m})
$$
is a birational invariant. One of the great achievements of birational geometry is the proof in 2010,
by Birkar, Cascini, Hacon and McKernan \cite{BCHM1,BCHM2} that 
$R$ is finitely generated (over its degree zero part which is always $\C$). I believe this is still an open problem over the fields of positive characteristics, and I am not willing to venture any guesses if it is true there or not.

\smallskip
\begin{war}
For $m<0$, the vector space 
$$
\Gamma(X,K_X^{\otimes m}) 
$$
is {\red not} a birational invariant. The proof that works for $m>0$ is not applicable, because vector fields do not pull back. For instance, for $m=-1$, 
$$\dim_\C \Gamma(\C\P^2, K_{\C\P^2}^\vee) = \dim_\C \Gamma(\C\P^2,\mathcal O (3))=10$$
but for the blowup $X$ of $\C\P^2$ at a  point $p$ we get 
$\dim_\C \Gamma(X, K_X^\vee) = 9$ as these correspond to homogeneous cubic polynomials in three variables that are zero at $p$.
\end{war}

\smallskip
Kodaira dimension concerns the ``size" of $R$. More precisely, $R$ is always an integral domain, and its quotient field is a field extension of $\C$ of finite transcendence degree. 

\smallskip
\begin{defn}
We define Kodaira dimension of $X$ by
$$
\kappa(X) = \left\{
\begin{array}{ll}
{\rm tr.deg.}(Q.F.(R))-1,&{\rm if~}R_{>0}\neq 0.\\
-\infty.&{\rm if~}R_{>0}= 0.
\end{array}
\right.
$$
\end{defn}

\begin{rem}
Kodaira dimension of $X$ is a birational invariant which takes values in 
$$
\{-\infty, 0, 1, \ldots, \dim_\C X\}.
$$
After the dimension, it is arguably the most important invariant of an algebraic variety.
Occasionally, $\kappa(X)=-\infty$ is denoted by $\kappa(X)=-1$, but the $-\infty$ notation is preferable, because it satisfies
$$
\kappa(X\times Y) = \kappa(X) + \kappa(Y)
$$
in the appropriate sense. 
\end{rem}

\smallskip
In $\dim_\C X = 1$ case, birational equivalence of smooth compact Riemann surfaces implies isomorphism, and Kodaira dimension separates all Riemann surfaces into three classes that we have already seen in Section \ref{sec.RH}.
\begin{itemize}
\item $\kappa(X)=-\infty.$ This is the case of genus $g=0$, and $X\cong \C\P^1$ is the only example. The universal cover of $X$ is $\C\P^1$. The automorphism group is ${\rm PGL}(2,\C)$.

\item $\kappa(X)=0.$ We have $g=1$, $X$ is an elliptic curve, the universal cover is $\C$. The automorphism group is one-dimensional.

\item $\kappa(X)=1.$ Here $g\geq 2$, the universal cover is the unit disc (equivalently, upper half plane). The automorphism group is finite.

\end{itemize}

We will now talk about the next case, $\dim_\C X = 2$. The situation is a lot more complicated. We will make no attempts at proving anything.

\smallskip
First let's talk about 
$\kappa(X) = -\infty.$ Examples include all of the rational surfaces, such as $\C\P^2$, $\C\P^1\times \C\P^1$, Hirzebruch and del Pezzo surfaces. In addition, this class contains products $\C\P^1 \times Y$ for a curve $Y$, and more generally ruled surfaces ($X\to Y$ with generic fiber isomorphic to $\C\P^1$).

\smallskip
The next case of 
$\kappa(X) = 0$ is arguably the most fascinating. In this case either the canonical line bundle $K_X$ or some positive tensor power of it is trivial. Examples include abelian surfaces,
so-called hyperelliptic surfaces, K3 surfaces and Enriques surfaces.

\smallskip 
Abelian surfaces are topologically equivalent to $(S^1)^4$ and are isomorphic to $\C^2/L$ for certain lattices $L\cong \Z^4$. They form countably many three-dimensional families inside of a four-dimensional family of complex manifolds (see Exercise 1), most of which are not algebraic. 

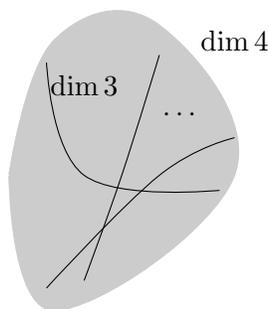
\begin{tikzpicture}
\filldraw[gray!40!white] plot[smooth, tension=.7] coordinates {(1.,0.01)(1.6,1.8) (3,2)(4,0) (1.6,-1.8)(1.,-0.01)};
\draw plot[smooth, tension=.7] coordinates {(1.5,-1.5)(3,0)(4,0.5)};
\draw plot[smooth, tension=.7] coordinates {(1.5,1.5)(2,0)(3.8,-0.2)};
\draw plot[smooth, tension=.7] coordinates {(2,-1.4)(2.5,0)(3,1.6)};
\node at (4,1.8){$\dim 4$};
\node at (3.3,0.8){$\cdots$};
\node at (2,1.2){$\dim 3$};
\node at (9,0) {$\begin{array}{c}
{\rm You~try~drawing~infinitely~many} \\
{\rm three-dimensional~loci~inside }\\
{\rm a~four-dimensional~space!}
\end{array}$};
\end{tikzpicture}

\smallskip 
Hyperelliptic (a.k.a. bielliptic) surfaces are quotients $(E\times F)/G$ of certain products of elliptic curves by certain finite group actions. There are seven families of such surfaces.

\smallskip 
K3 surfaces have been studied by many authors, but they were named as such in 1958 by Weil, to honor 
Kummer, K\"ahler and Kodaira, as well as the K2 mountain in the Himalayas. Similarly to the abelian surfaces, they form countably many $19$-dimensional families inside a $20$-dimensional family of  complex surfaces. In what follows, I will describe some of these families and give a (non-rigorous) count of the parameters. 

\smallskip 
\begin{itemize}
\item Consider the double cover $X\to \C\P^2$ ramified over a smooth degree $6$ curve $f(x_0,x_1,x_2)=0$ which roughly means looking at $t^2 = f(x_0,x_1,x_2)$ in some appropriate space, see Exercise 2.\footnote{ We do not explain why $K_X\to X$ is  trivial in this case.} The parameter count goes as follows. We have 
a $28$-dimensional space of homogeneous polynomials of degree $6$, which gives $27$ parameters up to scaling. Moreover, we should subtract the dimension of the automorphism group of $\C\P^2$, which is $\dim {\rm PGL}(3,\C)=8$.
Thus, we get $27-8=19$.

\item
The next case is that of degree $4$ hypersurfaces $X$ in $\C\P^3$ that one can view as a direct generalization of cubic curves in $\C\P^2$. Note that by the adjunction formula of Proposition \ref{adjform} we have
$$
K_X \cong \mu^*(K_{\C\P^3} \otimes \mathcal O(X)) \cong \mu^*(\mathcal O(-4)\otimes \mathcal O(4)) = \mu^*\mathcal O
$$
is trivial.
The space of polynomials of degree $4$ is of dimension $35$ and we get as before
$$
35 - 1 - \dim  {\rm PGL}(4,\C)= 35 -1 - 15 = 19.
$$

\item
Now let us consider the surfaces $X$ in $\C\P^4$ whose homogeneous ideal is generated by one quadratic and one cubic equation. We can think of $X$ as an intersection of a quadric $Q$ and a cubic $C$ in $\C\P^4$. The canonical class  of $X$ is trivial by a repeated application of the adjunction formula, first going from $\C\P^4$ to $C$ and then going from $C$ to $X$. The space of quadratic polynomials in $5$ variables is of dimension $15$, so we get a dimension $14$ parameter space of $Q$. The space of equations of $C$ is of dimension $35$, but we need to not only account for the scaling (subtract $1$) but also for adding any products of the equation of $Q$ and a degree one polynomial. And let's not forget about the automorphisms of $\C\P^4$. All in all, we get
$$
14 + (35 - 1 - 5) - \dim  {\rm PGL}(5,\C) = 14 + 29 - 24 = 19.
$$
\end{itemize}

\smallskip
\begin{rem}
We do list one more family in Exercise 3 and there are, as I mentioned earlier, infinitely many families, but the description gets more complicated further on. One really needs Hodge theory for the proper treatment of K3 surfaces.
What is remarkable, and not at all obvious, is that all of the above are {\blue diffeomorphic to each other as real manifolds!}
\end{rem}

\smallskip 
To finish the  $\kappa(X) = 0$ case,
Enriques surfaces are quotients of certain special K3 surfaces by $\Z/2\Z$ that acts without fixed points. They form a single $10$-parameter family.

\smallskip
In the case of $\kappa(X) = 1$, the surfaces in question have a structure of {\blue elliptic fibration}, i.e. there is a map
$X\to Y$, where $Y$ is a curve, and most fibers are smooth curves of genus $1$. There is a detailed classification of what kind of special fibers could occur, due to Kodaira.

\smallskip
The case of $\kappa(X) = 2$ is that of {\blue surfaces of general type}. It includes in particular products $C_1\times C_2$ 
where both $C_i$ are curves of genus at least two. Other simple examples are smooth surfaces in $\C\P^3$ of degree at least five. Overall, we are nowhere near classifying surfaces of general type, although there are multiple very interesting constructions of some of them with small Betti numbers. My hobby over the last eight years or so has been trying to construct explicit equations of a very special type of such surfaces, known as {\blue fake projective planes}.\footnote{ They have been classified as $100$ group quotients of a two-dimensional ball by \cite{CS}, following \cite{PY}, but this does not provide explicit equations.}
For a review of related surfaces, I highly recommend \cite{BCP}.

{\bf Exercise 1.} Count the parameters of rank $4$ lattices in $\C^2$, up to linear transformations of $\C^2$ to show that they form a four-dimensional family.

{\bf Exercise 2*.} The {\blue weighted projective space} $W\P(1,1,1,3)$ is defined as the quotient  of $\C^4\setminus \{(0,0,0,0\}$ with coordinates $(x_0,x_1,x_2,t)$ by the action
of $\C^*$ given by $(x_0,x_1,x_2,t)\mapsto (\lambda x_0,\lambda x_1,\lambda x_2,\lambda^3 t)$. As a set,  $W\P(1,1,1,3)$ can be covered by the charts $U_i=\{x_i\neq 0\}$ and $U_t=\{t\neq 0\}$ which gives it a structure of a (singular) algebraic variety. Prove that $U_{0,1,2}$ are in natural bijection to $\C^3$ and $U_t$ is in bijection to the quotient of $\C^3$ by the group $\Z/3\Z$ that acts by scaling coordinates by the third roots of $1$. Prove that the surface
$$
t^2 = x_0^6+ x_1^6+ x_2^6
$$
is smooth and maps $2:1$ to $\C\P^2$, with ramification divisor given by 
$$0 =x_0^6+ x_1^6+ x_2^6.$$

{\bf Exercise 3.} One more accessible example of K3 surfaces is that of surfaces $X$ in $\C\P^5$ cut out by three quadratic equations. Do the rough parameter count, taking into account that only the linear span of the three equations matters. \emph{Hint:} Count the ordered bases of the space of generators, which overcounts the parameter space by
$\dim {\rm GL}(3,\C)=9$. If you know the formula for the dimensions of Grassmannians, that could be used as well.

\section{Chern classes of vector bundles. Chern character. Euler sequence.}
Let $X$ be a smooth complex manifold and $W\to X$ be a rank $r$ holomorphic vector bundle.
Then it is possible, although not easy, to construct the {\blue total Chern class} of $W$
$$
c(W) = 1 + c_1(W) + c_2(W) + \ldots + c_r(W) \in H^{\rm even}(X,\Z) = \bigoplus_{i}H^{2i}(X,\Z)
$$ 
where $1\in H^0(X,\Z)$ and $c_i(W)$ is an element of $H^i(X,\Z)$. The individual terms
$c_i(W)$ are called the $i$-th Chern classes of $W$.

\smallskip
The case of line bundles $L\to X$ is of particular interest. Here the only nontrivial Chern class is $c_1(L)$.
If $L$ corresponds to the Weil divisor $D=\sum_i a_iD_i$, the class
$c_1(L)$ turns out to be the Poincar\'e dual of the homology element $[D]$ in $H_{2\dim_\C X -2}(X,\Z)$.
So in this sense we can think of Chern classes as some kind of higher rank invariants that generalize the
correspondence between line bundles and Weil divisor classes.

\smallskip
Chern classes carry a number of properties, which we list below, obviously without proof.
\begin{itemize}
\item $c_k(W) = 0$ for $k>\rk W$.

\item (Functoriality) For a morphism $f:X\to Y$ and vector bundle $W\to Y$ there holds
$$
c_k (f^* W) = f^*(c_k(W))
$$
for all $k$. On the left hand side we have the Chern class of the pullback of the vector bundle, and on the right we have the cohomology pullback $H^{2k}(Y,\Z) \to H^{2k}(X,\Z)$ of the Chern class.

\item For a short exact sequence of vector bundles\footnote{ This means we have maps of vector bundles that induce short exact sequences in each fiber.}
$$
0\to W_1 \to W_2 \to W_3\to 0
$$
on $X$ there holds
$$
c(W_2) = c(W_1)c(W_3).
$$
\end{itemize}

 As the corollary of the last statement, if a vector bundles $W$ is a direct sum of line bundles
$$
 W \cong L_1 \oplus L_2\oplus \ldots \oplus L_r,
$$
then 
$$
c(W) = (1+ c_1(L_1))(1+c_1(L_2))\cdot\ldots\cdot (1+c_1(L_r)).
$$

\smallskip{\blue \hrule}
Now I want to explain  how one can work with Chern classes. A typical problem might be the following. Suppose we have a rank $2$ vector bundle $W$ with the total Chern class $c(W) = 1+ c_1 +c_2$, and we want to compute the Chern class of the third symmetric power ${\rm Sym}^3(W)$.

\smallskip
Because ${\rm Sym}^3(\C^2)\cong \C^4$, the vector bundle ${\rm Sym}^3(W)$ has rank $4$.
Suppose for a moment that $W$ is a direct sum of line bundles 
$$W \cong L_1\oplus L_2,$$
in which case 
$${\rm Sym}^3(W)\cong L_1^{\otimes 3} \oplus (L_1^{\otimes 2} \otimes L_2 )\oplus  (L_1 \otimes L_2^{\otimes 2})\oplus L_2^{\otimes 3}.
$$ 
If we denote $c_1(L_1) = x_1,~c_2(L_2)=x_2$, then we have $c_1=x_1+x_2\in H^2(X,\Z)$ and $c_2 = x_1 x_2 \in H^4(X,\Z)$. Since tensor product of line bundles corresponds to the addition of the corresponding Weil divisor classes, we get
$$
c({\rm Sym}^3(W)) = (1+3x_1)(1+2x_1+x_2)(1+x_1+2x_2)(1+3x_2).
$$
On the right, we have a symmetric function in $x_1,x_2$, which we can write in terms of its elementary symmetric polynomials $c_1$ and $c_2$ as
\begin{equation}\label{sym3}
c({\rm Sym}^3(W)) = 1 + 6c_1 + (11c_1^2 + 10c_2) + (6c_1^3 + 30 c_1 c_2) + (18 c_1^2 c_2 + 9c_2)
\end{equation}

The black magic known as the {\blue splitting principle} then says that $c({\rm Sym}^3(W))$ is given by the formula \eqref{sym3} even if $W$ is not isomorphic to a direct sum of line bundles! This allows us to compute Chern classes of vector bundles that are constructed from given vector bundles.

\smallskip
An important terminology convention is to write 
$$c(W) = \prod_{i=1}^{\rk W} (1+ x_i(W))$$
and call $x_i(W)$ the {\blue Chern roots} of $W$. If $W$ is a direct sum of line bundles, then its Chern roots are just their first Chern classes. More generally, the Chern roots themselves are not assigned a precise meaning, but their elementary symmetric functions are the Chern classes of $W$. As a consequence, a degree $k$ symmetric polynomial in the Chern roots of $W$ is a polynomial in the Chern classes of $W$, and is an element of $H^{2k}(X,\Z)$.
A particular case is that of the tangent bundle $TX\to X$, where we write
$$
c(TX) = \prod_{i=1}^{\dim_\C X} (1+x_i)
$$
and call $x_i$ the Chern roots of $X$.\footnote{ I have also seen these more properly referred to as the Chern roots of the tangent bundle to $X$. Perhaps this will become the common convention in the future.}

\smallskip
{\blue \hrule}
Let us now do an important example. We would like to compute the total Chern class of the tangent bundle to 
$\C\P^n$.

\smallskip
\begin{prop}\label{Euler}
There is a natural sequence of vector bundles on $\C\P^n$
$$
0\to \mathcal O \to \mathcal O(1)^{\oplus(n+1)}\to T\C\P^n \to 0
$$
usually called the {\blue Euler sequence}.
\end{prop}

\begin{proof}
I will only sketch the proof. The key to this is the presentation
$$
\C\P^n = (\C^{n+1}\setminus \{{\mathbf 0}\})/\C^*.
$$
We can think of vector fields on $\C\P^n$ as $\C^*$-equivariant vector fields on $\C^{n+1}\setminus \{{\mathbf 0}\}$ of ``weight $1$",
up to the vector field $\sum_{i=0}^{n} x_i \frac \partial{\partial x_i}$ along the fibers of the quotient map.
\end{proof}

As a consequence of the Euler sequence, we get 
$$
c(T\C\P^n) = \frac {c(\mathcal O(1))^{n+1}}{c(\mathcal O)} = \frac {(1+H)^{n+1}}{1+0} = (1+H)^{n+1}
$$
where $H=c_1(\mathcal O(1))$ is the standard generator of $H^2(\C\P^n,\Z)$.
Note, however, that it would be wrong to claim that $H,\ldots,H$ ($n+1$ times) are Chern roots of $T\C\P^n$, because that would give the wrong rank.

\smallskip
{\blue \hrule}
We will now discuss a very important way of packaging the Chern class information, known as the {\blue Chern character}.

\smallskip
\begin{defn}\label{ch}
Let $W$ be a rank $r$ vector bundle on $X$ with 
$$c(W)=\prod_{i=1}^r(1+y_i)$$
where $y_i$ are the Chern roots. Then we define the Chern character of $W$ by
$$
ch(W) := \ee^{y_1} + \ee^{y_2} + \cdots + \ee^{y_r}
= r + (y_1 + \ldots y_r) + (\frac 12 y_1^2 + \ldots + \frac 12 y_r^2) + \cdots
$$
in  $H^{\rm even}(X,\Q)$. We have to use the coefficient ring $\Q$ because of the denominators in the exponential. Note that the formal power series of $\ee^y$ can be truncated to a polynomial, since $H^{>2\dim_\C X}(X,\Q)=0$.
\end{defn}

\smallskip
\begin{rem}
People these days plug {\red everything} into the exponential function! My personal favorite is $\ee^X$, which is the disjoint union of properly understood quotients of the products of $n$ copies of a manifold $X$ by the symmetric group of permutations.
\end{rem}

\smallskip
\begin{rem}
The standard notations $ch(W)$ and $c(W)$ are close enough to be somewhat confusing at first reading, even though total Chern class and Chern character have very different properties.
\end{rem}

The following properties of the Chern character can be deduced from those of the total Chern class and the splitting principle.
\begin{itemize}
\item The $H^0(X,\Q)$ part of $ch(W)$ is always the rank of $W$.

\item For a short exact sequence 
$$
0\to W_1 \to W_2 \to W_3 \to 0
$$
there holds $ch(W_2) = ch(W_1)+ch(W_3)$.

\item For $f:X\to Y$ and vector bundles $W$ on $Y$ there holds $ch(f^*(W))=f^*(ch(W))$.

\item $ch(W_1\otimes W_2)= ch(W_1)ch(W_2)$.

\end{itemize}

The last property is perhaps the most useful, it is much simpler than the corresponding statement for the total Chern classes.

\smallskip
As a corollary, we can use the Euler sequence on $\C\P^n$ to get
$$
ch(T\C\P^n) = (n+1)\ee^H -1.
$$

\smallskip
There is also some strange but important class called the {\blue Todd class} of $X$. It is defined as follows.

\smallskip
\begin{defn}
Let $x_i$ be the Chern roots of $TX$, i.e. 
$$c(TX)=1+ c_1 +c_2 + \ldots =\prod_i(1+x_i).$$
Then 
$$
{\rm Td}(X) = \prod_{i=1}^{\dim_\C X} \frac {x_i}{1-\ee^{-x_i}}
$$
where we expand $\frac {x}{1-\ee^{-x}}$ as a formal power series in $x$. It can be explicitly written as
\begin{equation}\label{td}
{\rm Td}(X)= 1 + \frac {c_1}2 + \frac {c_1^2 + c_2} {12}+\frac {c_1c_2}{24}+\cdots,
\end{equation}
see Exercise 3.
\end{defn}

{\bf Exercise 1.} Verify \eqref{sym3}

{\bf Exercise 2.} Prove that the first Chern class $c_1(W)$ of a vector bundle $W$ is the first Chern class of its top exterior power $\Lambda^{\rk W} W$.

{\bf Exercise 3.} Prove that the coefficients of the Todd class of $X$ can be written as polynomials in the Chern classes of $TX$, without using the dimension of $X$, and verify \eqref{td}. \emph{Hint:} By looking at 
$$\ln c(X) = \sum_i \ln(1+x_i),$$
we can write the sums $\sum_{i=1}^{\dim_\C X} x_i^k$ as polynomials in $c_j$. Then consider $\ln {\rm Td}(X)$.

\section{Cohomology of vector bundles and Hirzebruch-Riemann-Roch formula. 
Example: $\chi(\C\P^n,\mathcal O(k))$.}
Recall that for a  vector bundle $\pi:W\to X$ on a smooth complex algebraic variety, sections $s$ of it are  maps $s:X\to W$ such that $\pi \circ s = {\rm id}_X$. Sections of $W$ have a natural structure of a complex vector space, which we (for now) denote by $\Gamma(X,W)$.  

\smallskip
Suppose that we have a short exact sequence 
$$
0\to W_1 \to W_2 \to W_3 \to 0
$$
of vectors bundles on $X$, which means that we have this short exact sequence of vector spaces in each fiber.
Any section of $W_1$ induces a section of $W_2$ and a section of $W_2$ induces one of $W_3$
 and we have the following result.
\smallskip
\begin{prop}\label{lex}
There is a left exact sequence
$$
0\to \Gamma(X,W_1)
\to \Gamma(X,W_2)
\to \Gamma(X,W_3)
$$
\end{prop}

\begin{proof}
The fact that $\Gamma(X,W_1)
\to \Gamma(X,W_2)$ is injective is obvious, since a zero in a fiber of $W_2$ is also a zero in a fiber of $W_1$.
For the exactness in the middle, suppose that $s:X\to W$ is a section of $W_2\to X$ which maps to the zero section
of $W_3\to X$. This means that for every $x\in X$ the image of $s(x)$ in the fiber of $W_3\to X$ is zero. But this means that $s(x)$ actually lies in the fiber of $W_1$ inside that of $W_2$, and $s$ is an image of a section $s_1$ of $W_1\to X$. 
\end{proof}

\smallskip
As you may have guessed, there is generally {\blue no right exactness}. In other words, a surjective map of vector bundles
$$W_2\to W_3\to 0$$
may or may not induce a surjective map of the corresponding spaces of sections. The issue here is that given a section $s:X\to W_3$, we may always pick a lift to $W_2$ at each fiber, but we may not be able to do it in a consistent (say, holomorphic) fashion over all of $X$.
We provide a specific example of non-surjectivity below.

\smallskip
Let $X=\C\P^n$. Consider the Euler sequence
$$
0\to \mathcal O \to \mathcal O(1)^{\oplus n+1} \to TX \to 0.
$$
We can dualize it to get 
$$
0\to (TX)^\vee \to \mathcal O(-1)^{\oplus n+1} \to \mathcal O \to 0.
$$
The line bundle $\mathcal O$ is the trivial bundle $\C\times X\to X$, whose global sections are the holomorphic functions on $X$, which are constants. So $\Gamma(X,\mathcal O) \cong \C$. We claim that the line bundle 
$\mathcal O(-1)\to \C\P^n$ has only zero holomorphic sections. One way of seeing it is that the coordinates $x_0,\ldots, x_n$ on $\C\P^n$ are sections of $\mathcal O(1)$, so for any section $s\in \Gamma(\C\P^n,\mathcal O(-1))$ the tensor product $s \otimes x_0$ is a section of $\mathcal O$ and is therefore a constant. Since it is zero at $x_0=0$, we see that it must be a zero constant. This is implies that $s$ is zero for $x_0\neq 0$ and is thus zero everywhere by continuity.

\smallskip
{\blue \hrule}
It turns out that there is a way of working with the lack of exactness on the right, other than just giving up.
Specifically, for a short exact sequence 
$$
0\to W_1 \to W_2\to W_3 \to 0
$$
of vector bundles on $X$ there exists a long exact sequence
\begin{equation}\label{coho}
\begin{array}{l}
0\to \Gamma(X,W_1) \to \Gamma(X,W_2) \to \Gamma(X,W_3) \to\\[.5em]
\to H^1(X,W_1)\to H^1(X,W_2)\to H^1(X,W_3)\to\\[.5em]
\to H^2(X,W_1)\to \cdots 
\end{array}
\end{equation}
where $H^i(X,W)$ are certain {\blue cohomology} vector spaces which depend functorially on $W$. 
\begin{rem}
From now on, we will use $H^0(X,W)=\Gamma(X,W)$ to make the long exact sequence \eqref{coho} look more uniform.
\end{rem}

While we don't have the tools needed to define $H^i(X,W)$ (the hardest part are the so-called connecting homomorphisms $H^i(W_3)\to H^{i+1}(W_1)$) we will list some of their properties.
\begin{itemize}
\item $H^i(X,W)=0$, unless $0\leq i \leq \dim_\C X$.

\item For $X$ compact, $H^i(X,W)$ are finite-dimensional.

\item (Serre Duality) For a smooth projective variety $X$, for all vector bundles $W$ and all $i$ we have a natural duality 
$$
H^i(X,W) \cong H^{\dim_\C X -i} (X, W^\vee \otimes K_X)^\vee
$$
where $K_X$ is the canonical line bundle on $X$.

\item (Kodaira vanishing theorem) Let $L\to X$ be an ample line bundle on a smooth projective complex variety $X$, see Definition \ref{ample}. Then $H^i(X, L \otimes K_X)=0$ for all $i>0$.

\end{itemize}

\smallskip
\begin{defn}
For any vector bundle $W$ on a smooth projective variety $X$ we define the {\blue Euler characteristics} of $W$ by
$$
\chi(X,W):= \sum_{i=0}^{\dim X} \dim_\C H^i(X,W) = \dim_\C H^0(X,W)- \dim_\C H^1(X,W)+\cdots. 
$$
\end{defn}

The main property of the Euler characteristics of a vector bundle is that it is additive on short exact sequences. More precisely, for any 
$$
0\to W_1 \to W_2\to W_3 \to 0
$$
there holds $\chi(W_2) = \chi(W_1)+\chi(W_3)$,
see Exercise 1.
This is why it is often much easier to compute $\chi(X,W)$ than the more geometrically meaningful $\dim_\C H^0(X,W)$. Fortunately, in some cases of interest Kodaira vanishing or other similar theorems allow one to conclude that 
$\chi(X,W) = \dim_\C H^0(X,W)$. 

\smallskip
Additivity of $\chi$ on short exact sequences should remind you of the additivity of the Chern character. Not coincidentally, the famous {\blue Hirzebruch-Riemann-Roch formula} allows one to compute $\chi(X,W)$ if one has a good understanding of $ch(W)$, ${\rm Td}(X)$ and the product in $H^{\rm even}(X,\Q)$.

\smallskip
\begin{thm}\label{hrr}
Let $X$ be a smooth projective complex algebraic variety.
We define $\int_X:H^{\rm even}(X,\Q)\to \Q$ as reading off the coefficient at the Poincar\'e dual of the point in the top cohomology $H^{2\dim_\C X}(X,\Q)$. Then for any vector bundle $W$ on $X$ we have
$$
\chi(X,W) = \int_X ch(W) \,{\rm Td}(X).
$$
\end{thm}

We have no intention of proving Theorem \ref{hrr} but we will look at some of its corollaries.

\begin{itemize}
\item $\rk  =1$, $\dim_\C X = 1$. We will denote the line bundle by $L$.
Since $H^{>2}(X,\Q)=0$, we get
$$ch(W) = \ee^{c_1(L)} = 1+ c_1(L) = 1+ (\deg L ){\rm P.D.(point)}.$$
We also have
${\rm Td}(X)=1+\frac12 c_1(TX) = 1 + \frac 12 (2-2g) {\rm P.D.(point)}$ where $g$ is the genus of $X$. Therefore,
we have 
$$
\chi(X,L) = \int_X (1+ (\deg L ){\rm P.D.(point)}  +(1-g) {\rm P.D.(point)})= \deg L + 1-g.
$$
Here we used Proposition \ref{degcan} that the degree of the canonical line bundle is $2g-2$, and the fact that $TX$ is the dual of the canonical line bundle.

\item $\rk =1$, $\dim_\C X = 2$. We have 
$$
ch(L) = 1+ c_1(L) + \frac 12 c_1(L)^2,~ {\rm Td}(X) = 1+\frac 12 c_1 + \frac {c_1^2 + c_2}{12}
$$
where $c_i$ denotes $c_i(TX)$. This gives
$$
\chi(X,L) = \frac 12 c_1(L)^2 + \frac 12 c_1(L) c_1 + \frac 12 c_1 + \frac {c_1^2 + c_2}{12}.
$$
In the particular case $L=\mathcal O$ we get the {\blue Noether's formula}\footnote{ It is worth mentioning that $c_1^2$ is the self-intersection of the canonical divisor and $c_2$ is the topological Euler characteristics of $X$, i.e. the alternating sum of its Betti numbers.}
$$
\chi(\mathcal O) =  \frac {c_1^2 + c_2}{12}.
$$
Then for general $L$ that corresponds to a Weil divisor $D$ we get
$$
\chi(X,L) = \frac 12  D(D-K)+\chi(\mathcal O).
$$

\item
$X=\C\P^n$ and $L = \mathcal O(k)$. The cohomology ring of $\C\P^n$ is generated in degree two by the Poincar\'e dual $H$ of a hyperplane, with $H^n={\rm P.D.(point)}$ being the generator of the top cohomology.
$$
H^{*}(\C\P^n,\Z) = H^{\rm even}(\C\P^n,\Z) = \C\oplus \C H\oplus \cdots \oplus\C H^n.
$$
We have $ch(L) = \ee^{kH}$, but the computation for the Todd class is a bit more complicated. We have
$$
c(TX) = (1+H)^{n+1},
$$
but these are not the Chern roots. Nevertheless\footnote{ One can define Todd classes of vector bundles in a way that's multiplicative for short exact sequences and then use the Euler sequence.},
$$
{\rm Td}(X) = \left(\frac {H}{1- \ee^{-H}} \right)^{n+1}
$$
so Hirzebruch-Riemann-Roch formula says
\begin{align*}
\chi(X,\mathcal O(k)) = \int_X \frac {\ee^{kH} H^{n+1}}{(1- \ee^{-H})^{n+1}}
={\rm coeff.~at~}H^n\left( \frac {\ee^{kH} H^{n+1}}{(1- \ee^{-H})^{n+1}}\right)\\
={\rm Res}_{H=0}\frac {\ee^{kH} }{(1- \ee^{-H})^{n+1}} 
=\frac 1{2\pi\ii} \oint_{|H|=\varepsilon}\frac {\ee^{kH} }{(1- \ee^{-H})^{n+1}} \,dH 
\\
\stackrel {t=1-\ee^{-H}}{=}
\frac 1{2\pi\ii} \oint_{|t|
=\varepsilon}\frac {(1-t)^{-k-1} }{t^{n+1}}  \, dt = 
\frac {(k+1)(k+2)\cdot\ldots\cdot(k+n)}{n!}
\end{align*}
where we treat $H$ as a complex variable in the second line. In the last identity we used that it is the coefficient at $t^n$ of the Maclaurin series of $(1-t)^{-k-1}$.
We will see the more elementary reason behind this formula in Exercise 3.
\end{itemize}

\smallskip

{\bf Exercise 1.} Prove that Euler characteristics is additive on short exact sequences. \emph{Hint:} Use \eqref{coho} and the general statement about alternating sum of dimensions of vector spaces in a long exact sequence.

{\bf Exercise 2.} Use Kodaira vanishing theorem to show that  for all $k>-n-1$ we have $H^{>0}(\C\P^n,\mathcal O(k))=0$. \emph{Hint:} Use
Exercise 3 of Section \ref{sec.CD}.

{\bf Exercise 3.} One can show (see Exercise 2 of Section \ref{sec.maps}) that for $k\geq 0$ the global sections of $\mathcal O(k)$ on $\C\P^n$ are homogeneous polynomials of degree $k$ in the homogeneous coordinates on $\C\P^n$. Together with Exercise 2, show that this is consistent with the above calculation of $\chi(\C\P^n, \mathcal O(k))$.

\section{Introduction to Grassmannians.}
In the next two sections we will talk about the {\blue Grassmannian} varieties. 

\smallskip
\begin{defn}
Let $V$ be a complex vector space of dimension $n$. For each $k\in \{0,\ldots,n\}$,
the Grassmannian $\Gr(k,V)$ is defined as the set of all dimension $k$ vector subspaces of $V$.
We also use $\Gr(k,n)$ notation for it if we don't particularly care about the nature of $V$.
\end{defn}

Clearly, for $k=0$ or $k=n$ the set $\Gr(k,V)$ consists of one element.
For $k=1$ we have $\Gr(k,V) = \P V\cong \C\P^{n-1}$.
For $k=n-1$ we have $\Gr(k,V) = \P V^\vee\cong \C\P^{n-1}$. This makes $\Gr(2,4)$ the first interesting case.

\smallskip
It is reasonable to expect that $\Gr(k,n)$ has a natural structure of an algebraic variety and a complex manifold. Indeed, this is the case. As with projective spaces, we can define it by using coordinate charts.
We will think of $V$ as the space of length $n$ {\blue row} vectors. Then a dimension $k$ subspace $W$ of $V$, {\blue together with a choice of a basis} can be regarded as a $k\times n$ matrix $M$ of rank $k$. Different choices of said basis of $W$ are related by invertible row transformations of $M$.
Thus we get, as sets,
$$
\Gr(k,n) = \{{\rm max~rank~} k\times n {\rm~complex~matrices}\}/{\rm GL}(k,\C)
$$
for the left multiplication action of ${\rm GL}(k,\C)$. Note, that in the $k=1$ case, this perfectly reproduces the definition of $\C\P^{n-1}$. 

\smallskip
Recall that $M$ has rank $k$ if and only if it has a nonzero minor. This inspires us to consider, for each set of indices 
$$
1\leq i_1 < i_2 <\cdots < i_k\leq n,
$$
the subset $U_{i_1,i_2,\ldots,i_k}$ of $\Gr(k,n)$ such that
$$
\det M_{i_1,\ldots,i_k} \neq 0
$$
where the determinant is that of the square
submatrix of $M$ made from columns 
$i_1,i_2,\ldots,i_k$. Note that this property is preserved under invertible linear row transformations.

\smallskip
Let us look in detail at the particular case of $U_{1,2,\ldots,k}$. This means that the first minor of $M$ is invertible, so Gaussian elimination gives
$$
M\sim \left(
\begin{array}{ccccccc}
1 & 0 &\ldots & 0 & * &\ldots &*\\
0 & 1 &\ldots & 0 & * &\ldots &*\\
\ldots & \ldots &\ldots & \ldots & \ldots &\ldots &\ldots\\
0 & 0 &\ldots & 1 & * &\ldots &*
\end{array}
\right)
$$
with the unique reduced row echelon form. This is just a fancy way of saying that we multiply $M$ by $M_{1,2,\ldots,k}^{-1}$ to get the first submatrix to be identity. The entries of the remaining $k\times (n-k)$ matrix can be arbitrary complex numbers, so we can identify $U_{1,\ldots, k}$ with $\C^{k(n-k)}$. For arbitrary $i_1<\cdots<i_k$, the chart $U_{i_1,\ldots, i_k}$ can be identified with $\C^{k(n-k)}$ by looking at the entries of $M_{i_1,i_2,\ldots,i_k}^{-1}M$ outside of the $i_j$-th columns.

\smallskip
On the intersection of $U_{i_1,\ldots,i_k}$ and $U_{j_1,\ldots,j_k}$, the transition functions are clearly\footnote{ To make it even more clear, we ask the reader to find this change of variables explicitly in Exercise 1.} rational and holomorphic, which gives 
$\Gr(k,n)$, and more generally $\Gr(k,V)$ the structure of a smooth complex algebraic variety. 

\smallskip
\begin{rem}
Compactness of $\Gr(k,n)$ is also fairly easy and is proved similarly to that of $\C\P^n$. Every $k$-dimensional subspace $W$ of $\C^n$ has an orthonormal basis with respect to the standard Hermitean form on $\C^n$. The space of such bases is a closed subset of $(S^{2n-1})^{k}$ (given by the orthogonality conditions) and is therefore compact. It maps continuously and surjectively onto $\Gr(k,n)$, which proves compactness.
\end{rem}

\smallskip{\blue \hrule}

We will now  embed $\Gr(k,V)$ into some projective space.
The idea is rather elegant: given a dimension $k$ subspace $W\subseteq V$ we get
$$
\Lambda^k W \subseteq \Lambda^k V.
$$
Since $\dim_\C \Lambda^k W=1$, we naturally associate to each $W$ a point in the projective space 
$\P\Lambda^k V$ (of dimension $\frac {n!}{k!(n-k)!}-1$). So we get a map 
$$
\mu:\Gr(k,V)\to \P\Lambda^k V.
$$

\begin{rem}
It is not a particularly mysterious map. 
If we think of $\Gr(k,n)$ as the quotient as before, then
the map is induced by 
$$
M\mapsto \left(\det M_{i_1,\ldots,i_k}\right)_{1\leq i_1<\cdots <i_k\leq n}.
$$
Indeed, left multiplication by a matrix $A$ scales all of the above determinants by $\det A$, so the map is well-defined.
It is also clear from this description that $\mu$ is holomorphic (and algebraic).
\end{rem}

\smallskip
\begin{rem}
It can be seen from the same construction that the map is an embedding. If two points go to the same point in the projective space, then they must have the same nonzero minor, which we can for simplicity assume to be $M_{1,\ldots, k}$. We can assume that $M$ is in the reduced form and observe that determinants of minors that use $(k-1)$ of the first $k$ columns pick up precisely the coordinates of the reduced form of $M$. It is also easy to see that tangent vectors do not collapse under $\mu$. 

Another, more abstract, way of thinking about injectivity is that 
$$
v \wedge (\Lambda^k W) = 0\in \Lambda^{k+1}V \iff v \in W,
$$
so $\Lambda^k W$ determines $W$ uniquely.
\end{rem}

\smallskip
{\blue \hrule}
This embedding $\mu$ is called the {\blue Pl\"ucker embedding} of the Grassmannian $\Gr(k,V)$ and 
the homogeneous coordinates on $\P \Lambda^k V$ are called the {\blue Pl\"ucker coordinates}. I now want to discuss the 
equations that cut out $\Gr(k,V)$ in terms of Pl\"ucker coordinates. Hmm, what should we call them?

\smallskip
But, first, let us talk about the exterior algebra $\Lambda^*V$. If $e_1,\ldots, e_n$ form a basis of $V$, then $\Lambda^* V$ is spanned by $e_{i_1}\wedge \cdots \wedge e_{i_m}$ with the associativity rules, and $e_i \wedge e_j = -e_j \wedge e_i$.
In fact, up to one measly sign, the rules are the same as those for the polynomials in $e_i$, and the exterior algebra can be viewed as a ring of polynomials in odd (in physics terminology ``fermionic") variables. Just as one has partial differentiation operators acting on the polynomial ring, there are odd analogs of them here, known as {\blue contractions}.
Specifically, for $\alpha\in V^\vee$ we can define linear maps ${\rm contr}_\alpha:\Lambda^m V \to \Lambda^{m-1} V$ with the property that
for any $v_1,\ldots, v_m\in V$ there holds
\begin{equation}\label{contr}
\begin{array}{l}
{\rm contr}_\alpha(v_1\wedge\cdots \wedge v_m) = 
\alpha(v_1) v_2 \wedge \cdots \wedge v_m \\
\hskip 20pt - \alpha(v_2) v_1 \wedge  v_3\wedge \cdots \wedge v_m +\ldots+
(-1)^{m-1} \alpha(v_m) v_1\wedge \cdots \wedge v_{m-1}.
\end{array}
\end{equation}

\smallskip
Suppose we have a point in $\Gr(k,V)\subseteq \P \Lambda^k V$ or, equivalently, a decomposable element 
$w=v_1\wedge \cdots \wedge v_k\subseteq  \Lambda^k V$. Then by \eqref{contr}, for arbitrary $\alpha_1,\ldots,\alpha_{k-1} \in V^{\vee}$ we see that
$$
{\rm contr}_{\alpha_{k-1}}\circ \cdots \circ {\rm contr}_{\alpha_1}(w)
$$
is a linear combination of $v_1,\ldots,v_k$, so we have 
$$
\left({\rm contr}_{\alpha_{k-1}}\circ \cdots \circ {\rm contr}_{\alpha_1}(w)\right)\wedge w = 0 \in \Lambda^{k+1} V.
$$
We can further restate it by saying that if $w$ is decomposable, then for all $\alpha_1,\ldots, \alpha_{2k}\in V^\vee$ there holds the following {\blue Pl\"ucker relation}.
\begin{equation}\label{plucker}
{\rm contr}_{\alpha_{2k}}\circ \cdots \circ {\rm contr}_{\alpha_k} \left(\left({\rm contr}_{\alpha_{k-1}}\circ \cdots \circ {\rm contr}_{\alpha_1}(w)\right)\wedge w \right)= 0
\end{equation}

\smallskip
We comment that it suffices to pick $\alpha_i$ in \eqref{plucker} among the basis elements of $V^{\vee}$, so we have a finite and  easily computable set of quadratic equations in coefficients of $w$. It is true, though definitely not obvious, that Pl\"ucker relations \eqref{plucker} cut out exactly $\Gr(k,V)$. As a matter of fact, they generate the homogeneous ideal of $\Gr(k,V)$ in $\P\Lambda^k V$.

\smallskip
Let us compute the smallest interesting example, that of $\Gr(2,4)$. We have a vector space $V$ with the basis $e_1,\ldots,e_4$ and we want to characterize the decomposable elements $w\in \Lambda^2V$. 
Every $w\in  \Lambda^2V$ can be written as
$$
w = x_{12}  \,e_1 \wedge e_2 +  x_{13}  \,e_1 \wedge e_3 +  x_{14}  \,e_1 \wedge e_4 + x_{23}  \,e_2\wedge e_3 
+ x_{24} \,e_2\wedge e_4 + x_{34} \,e_3\wedge e_4,
$$
and we think of $x_{12},\ldots, x_{34}$ as the homogeneous coordinates of $\C\P^5$. Since the dimension of $\Gr(2,4)$ is $4$, its image under the Pl\"ucker embedding is a hypersurface in $\C\P^5$, so we expect a single quadratic equation.
More precisely, if we take $\alpha_1$ to be the dual basis vector  $e_1^\vee$, then we have 
$$
{\rm contr}_{\alpha_1} w = x_{12} \, e_2 + x_{13}  \,e_3 + x_{14} \, e_4.
$$
We wedge it with $w$ to get 
\begin{align*}
&({\rm contr}_{\alpha_1} w)\wedge w =(x_{12} e_2 + x_{13} e_3 + x_{14} e_4)\wedge (
 x_{12} e_1 \wedge e_2 + \cdots + x_{34}e_3\wedge e_4)\\
 &=
x_{12}x_{13} \,e_2\wedge e_1 \wedge e_3  + x_{12}x_{14} \, e_2 \wedge e_1 \wedge e_4 
+ x_{12}x_{34}  \,e_2\wedge e_3 \wedge e_4 
\\& + x_{13}x_{12} \, e_3 \wedge e_1 \wedge e_2 + x_{13}x_{14} \, e_3 \wedge e_1 \wedge e_4 
+ x_{13}x_{24} \, e_3 \wedge e_2 \wedge e_4 
\\
&+ x_{14}x_{12}  \,e_4 \wedge e_1 \wedge e_2+x_{14}x_{13}  \,e_4 \wedge e_1 \wedge e_3+x_{14}x_{23} \, e_4 \wedge e_2 \wedge e_3
\\
&= (x_{12} x_{34} - x_{13} x_{24} + x_{14} x_{23} ) e_2 \wedge e_3 \wedge e_4.
\end{align*}
Therefore, the equation of $\Gr(2,4)\subseteq \C\P^5$ is
$$
0=x_{12} x_{34} - x_{13} x_{24} + x_{14} x_{23},
$$
and $\Gr(2,4)$ is a smooth (rank $6$) quadric hypersurface in $\C\P^5$.

{\bf Exercise 1.} Find explicitly the change of variables from $U_{1,3}$ to $U_{3,4}$ in $\Gr(2,4)$. 

{\bf Exercise 2.} Verify that in the special case of $\Gr(2,4)$ one can write the Pl\"ucker relation as $w\wedge w=0$.

{\bf Exercise 3.} Give an example of an element $w$ in $\Lambda^3 \C^6$ with $w\wedge w = 0$ which can not be written
as $w=v_1\wedge v_2 \wedge v_3$.

\section{Grassmannians continued. Number of lines that intersect four given lines in $\C\P^3$.}

In this section we  continue our discussion of the complex Grassmannians $\Gr(k,n)$. 

\smallskip
The singular cohomology ring $H^*(\Gr(k,n),\Z)$ of $\Gr(k,n)$ is well-studied but is intricate. 
Cohomology only occurs in even degrees and has a free generator set given by the Poincar\'e duals of the so-called Schubert cycles which are described as follows.

\smallskip
Let us think about points of  $\Gr(k,n)$  as full rank $k\times n$ matrices up to row transformations and recall that every such matrix has a unique reduced row echelon form. Such reduced row echelon form is characterized by the location of the pivot columns, and the Schubert cycle is defined as the closure of the locus of $M$ whose reduced row echelon form has given pivot columns. Instead of setting up general notation, we will illustrate this in the particular case of $\Gr(2,4)$. We can think of the points in $\Gr(2,4)$ as parameterizing lines $l\cong\C\P^1$ in the three-dimensional projective space $\C\P^3$, instead of $\C^2\subset \C^4$.\footnote{ This leads to a competing convention where what we call $\Gr(2,4)$ is called $\Gr(1,3)$, but I believe that ours is more widely used.}
There are six possible choices of pivot columns, and we list the Schubert cycles below, together with their geometric descriptions. 
\begin{align}\label{sch}
&
\overline{\left(
\begin{array}{cccc}
1&0&*&*\\
0&1&*&*
\end{array}
\right)},
\hskip 50pt
\overline{\left(
\begin{array}{cccc}
1&*&0&*\\
0&0&1&*
\end{array}
\right)},
\hskip 50pt
\overline{\left(
\begin{array}{cccc}
1&*&*&0\\
0&0&0&1
\end{array}
\right)},
\notag\\
&
{\rm ~all~of~}\Gr(2,4)
\hskip 50pt
l\cap \{(0:0:*:*)\}\neq \emptyset
\hskip 45pt
(0:0:0:1)\in l
\notag
\\[.5em]
&
\overline{\left(
\begin{array}{cccc}
0&1&0&*\\
0&0&1&*
\end{array}
\right)},
\hskip 50pt
\overline{\left(
\begin{array}{cccc}
0&1&*&0\\
0&0&0&1
\end{array}
\right)},
\hskip 50pt
\overline{\left(
\begin{array}{cccc}
0&0&1&0\\
0&0&0&1
\end{array}
\right)}.
\notag\\
&
l\subset\{(0:*:*:*)\}
\hskip 10pt
(0:0:0:1)\in l\subset \{(0:*:*:*)\}
\hskip 10pt
l=(0:0:*:*)
\end{align}
\begin{rem}
There is an equivalent description of Schubert cycles in terms of dimensions of the intersection of the subspace $W$ with elements of a complete flag on $V$, see \cite{GH}.
\end{rem}

\smallskip
{\blue \hrule}
There is a very nice alternative description of the cohomology ring of $\Gr(k,n)$ in terms of the Chern classes of certain vector bundles. Specifically, we have a natural exact sequence of vector bundles on $\Gr(k,n)$
\begin{equation}\label{naturalGr}
0\to S\to \mathcal O^{\oplus n} \to Q\to 0.
\end{equation}
All fibers of $\mathcal O^{\oplus n}$ are naturally identified with $\C^n$, 
the fiber of $S\to \Gr(k,n)$ over the point that corresponds to $W\subset \C^n$ is $W$, and the fiber of $Q$ over this point is 
$\C^n/W$. We call $S$ and $Q$ the tautological subbundle and the tautological quotient bundle respectively.

\smallskip
From the properties of Chern classes, we have
$$
1=c(S) c(Q) = (1+s_1 + s_2 + \cdots + s_k)(1+q_1 + q_2 + \cdots + q_{n-k}).
$$
This gives relations on $s_i$ and $q_i$ and there is a beautiful result that states that $H^*(\Gr(k,n),\Z)$ is generated by $s_i$ and $q_j$ subject to these relations. 

\smallskip
We will illustrate this in the case of $\Gr(2,4)$.
We have $(1+s_1+s_2)(1+q_1+q_2)=1$. Therefore,
\begin{equation}\label{sq}
s_1+ q_1 = 0,~ s_1 q_1 + s_2 + q_2 = 0,~ s_1 q_2 + s_2 q_1 =0, s_2 q_2 = 0.
\end{equation}
The quotient of $\Z[s_1,s_2,q_1,q_2]$ by the relations \eqref{sq} is isomorphic to
$$
\Z[s_1,s_2]/(s_1^3 - 2 s_1 s_2, s_1^2 s_2 - s_2^2),
$$
see Exercise 1. 
This implies $s_1s_2^2=s_1^3 s_2 =2 s_1 s_2^2$, so $s_1 s_2^2 = s_1^3s_2=0$. From here we get that $s_1 s_2^2=s_2^3=0$ and $s_1^5=0$.  As the result, we get
\begin{align*}
H^{0}(\Gr(2,4),\Z) = \Z ,~H^{2}(\Gr(2,4),\Z) = \Z s_1,~H^4(\Gr(2,4),\Z) = \Z s_1^2 \oplus \Z s_2,\\
H^{6}(\Gr(2,4),\Z) = \Z s_1s_2,~H^{8}(\Gr(2,4),\Z) = \Z s_2^2.\hskip 80pt
\end{align*}
The geometric meaning of the $s_i$ classes is the following. 
\begin{itemize}
\item
The $q_1=-s_1$ class is the Poincar\'e dual of the Schubert cycle of lines in $\C\P^3$ that intersect a given line (the second entry in \eqref{sch}). To see that, we dualize the short exact sequence of vector bundles \eqref{naturalGr} to get
$$
0\to Q^\vee \to( \mathcal  O^{\oplus 4})^\vee \to S^\vee \to 0.
$$
Linear functions on $\C^4$ give sections of the bundle $S^\vee$ by looking at their restrictions on
the fibers of $S$. Then the Pl\"ucker coordinates give sections of $\Lambda^2 S^\vee$, and one can see that this Schubert cycle is given by $x_{12}=0$.
The Poincar\'e dual of the divisor that corresponds to a section of $\Lambda^2 Q$ is therefore $c_1(S^\vee)=-c_1(S)=-s_1$.

\item The $s_2$ class is the Poincar\'e dual of the Schubert class of lines inside a given plane in $\C\P^3$ (the fourth entry).
To explain this, we invoke the general claim the top Chern class of a vector bundle is equal, under some transversality conditions, to the Poincare dual of the zero locus of a section. In our case, with a section given by a coordinate on $\C^4$, 
and the vanishing locus is that of lines $l$ that lie in the corresponding plane.

\item The $(-s_1)s_2$ class is the Poincar\'e dual of the Schubert class of lines that contain a given point and lie inside a given plane in $\C\P^3$ (the fifth entry). Indeed, we can think of this condition as lying in a given plane and intersecting a 
line.

\item The class $s_2^2$ is the Poincar\'e dual of a point. Indeed, we can think of it as intersection of the two Schubert cycles that correspond to $s_2$, but for different planes and use that two distinct planes in $\C\P^2$ intersect transversely in a line.

\item We also note that the Poincar\'e dual of the Schubert class of lines that contain a given point (the third entry), is given by $s_1^2 -s_2$. To explain this\footnote{ That's not exactly a proof.} we look at the Poincar\'e dual of $(-s_1)^2$ as the locus of lines that
intersect two given lines $l_1$ and $l_2$ in $\C\P^3$. If $l_1$ and $l_2$ intersect at a point $p$, then $l$ that intersects both $l_1$ and $l_2$ come in two flavors: lines in the span of $l_1$ and $l_2$ (dual to $s_2$) and lines through $p$.
\end{itemize}

\smallskip
\begin{rem}
The fact that $(-s_1)^4 = 2(-s_1)^2 s_2 = 2 s_2^2= 2\,{\rm P.D.(point)}$ is the manifestation of $\Gr(2,4)$ being a degree two hypersurface in $\C\P^5$. Indeed, $(-s_1)$ is the pullback of the hyperplane class in $\C\P^5$ and for a submanifold of complex dimension $d$ in $\C\P^n$, the $d$-th power of this class measures the degree of the submanifold.
\end{rem}

\smallskip
{\blue \hrule}
We end our discussion of Grassmannians with the following interesting calculation.
Let $l_1,l_2,l_3,l_4$ be four lines in $\C\P^3$ which we assume to be in general position. 
{\blue How many lines $l$ in $\C\P^3$ intersect all $l_i$?}
Clearly, this locus is the Poincar\'e dual to  the intersection $(-s_1)^4$, so we should get two points, so we have two lines $l$ that intersect all of $l_i$.
One must always be careful in algebraic geometry with various transversality issues, but it works out fine here.

\smallskip
\begin{rem}
There is another argument for the same statement that does not involve Grassmannians. The space of quadratic polynomials in homogeneous coordinates of $\C\P^3$ has dimension $10$, and the conditions of vanishing on a given line $l_i$ is that of vanishing of three coefficients of the restriction. Thus, for general $l_1,l_2,l_3$ there is a unique quadric surface $X$ in $\C\P^3$ that contains all three. If the lines are generic, the quadric will have a full rank and $X$ will be isomorphic to $\C\P^1\times \C\P^1$ in its Segre embedding. The lines on $X$ are precisely the fibers of 
two projections to $\C\P^1$ (see Exercise 3) and since $l_1,l_2,l_3$ are disjoint, they must be coming from the same projection. The line $l_4$ intersects $X$ in two points $p_1,p_2$. Any line $l$ that intersects $l_1,l_2,l_3$ has at least three intersection points with $X$. Therefore, $l$ lies in $X$, since a nonzero quadratic polynomial can not have three different roots. Thus $l$ must the the fiber of the second projection that passes through either $p_1$ or $p_2$.
\end{rem}

{\bf Exercise 1.} Prove that the quotient of $\Z[s_1,s_2,q_1,q_2]$ by the relations \eqref{sq} is isomorphic to
$
\Z[s_1,s_2]/(s_1^3 - 2 s_1 s_2, s_1^2 s_2 - s_2^2).
$

{\bf Exercise 2.} Give a geometric meaning to the relations
$(s_1^2-s_2)s_2=0$ and $(s_1^2 - s_2)^2={\rm P.D.(point)}$ in terms of intersections of the corresponding Schubert classes.

{\bf Exercise 3.} Let $X\subset\C\P^3$ given by $x_0 x_3 - x_1 x_2 = 0$. Prove that all lines in $X$
are either
$$
l_{a,b}=\{
(x_0:x_1:x_2:x_3) = \{(a y_0: a y_1: by_0:b y_1),~(y_0:y_1)\in \C\P^1\}
$$
or 
$$
l'_{a,b}=\{
(x_0:x_1:x_2:x_3) = \{(a y_0: b y_0: ay_1:b y_1),~(y_0:y_1)\in \C\P^1\}
$$
for some $(a,b)\neq(0,0)$.

\section{Mirror Symmetry.}\label{sec.ms}
In this section, I will give a brief introduction into an active area of algebraic geometry known as Mirror Symmetry. It is a rather recent development, starting in the early 90-s. As opposed to other sections, I have earned the right to present a highly subjective overview of the field, as this has been my primary research area over the last three decades.

\smallskip
We start with the classical example of the smooth quintic threefold\footnote{ The terms ``threefold" is an amalgamation of ``three" and ``manifold". It simply means a three-dimensional algebraic variety. There are also fourfolds, fivefolds, etc.} in $\C\P^4$. Given a homogeneous polynomial $F(x_0,\ldots, x_4)$ of degree $5$ in the homogeneous coordinates of $\C\P^4$, we consider 
$$
Q:=\{F(x_0,\ldots, x_4)=0\} \subset \C\P^4.
$$
It is a simple three-dimensional example a so-called {\blue Calabi-Yau variety}, in particular $K_Q = \mathcal O$ by the adjunction formula of Proposition \ref{adjform}.

\smallskip
We will ask the following question: How many lines in $\C\P^4$ lie in $Q$? This is a reasonable question, in the sense that the dimension of the space of lines in $\C\P^4$ is $\dim \Gr(2,5) = 6$ and for each line
$$
(x_0:\ldots:x_4) = (a_0 u + b_0 v: a_1 u + b_1 v:\ldots:a_4 u + b_4 v),~(u:v)\in \C\P^1
$$
the condition of being in $Q$ is the condition that the coefficients of the restriction of $F$ to $l$ are all zero. There are six coefficients in a homogeneous polynomial of degree $5$ in $u$ and $v$, so it is plausible that, at least for generically chosen $F$, there are finitely many lines in $Q$. 

\smallskip
\begin{rem}
Why would anybody care about the number of lines in $Q$? It is a very long tradition to look for some nice subvarieties in a given variety and it doesn't get nicer than lines. Finding them may be difficult, so one can settle for finding their number.
\end{rem}

Let us try to compute it. We have the usual short exact sequence of vector bundles on $\Gr(2,5)$
$$
0 \to S \to \C^5\times \Gr(2,5) \to Q \to 0
$$
which we dualize to get 
$$
0 \to Q^\vee \to (\C^5)^\vee \times \Gr(2,5) \to S^\vee \to 0.
$$
Each $x_i$ for $i=0,\ldots, 4$ gives a section of $S^\vee$ and $F(x_0,\ldots, x_5)$ gives a global section $s_F$
of the rank $6$ vector bundle ${\rm Sym}^5 (S^\vee)$. In simple terms, the fibers of  ${\rm Sym}^5 (S^\vee)$ encode the degree five homogeneous functions on lines $l$. The section $s_F$ encodes the restrictions of $F$ to the lines. For for a point $p\in \Gr(2,5)$ that corresponds to the line $l\subset \C\P^4$ we have
$$
(l\subset Q) \iff s_F(p) =0.
$$

\smallskip
We can use Chern classes to figure out the number of zeros of a section. For a rank $r$ vector bundle $W\to X$ on an $r$-dimensional variety $X$ we have $c_r(W)= n\,{\rm P.D.(point)}$ where $n$ is the expected number of zeros of a global section of $W$. We do not prove it, but in the special case where $W = L_1\oplus \cdots\oplus L_r$ and $L_i$ have sections $t_i$, the zeros of the section $(t_1,\ldots, t_r)$ of $W$ occur at the intersection of the zeros of $t_i$, which matches $c_r(W) = c_1(L_1)\cdots c_1(L_r)$.

\smallskip
The short exact sequence $0\to S\to \mathcal O^{\oplus 5} \to Q\to 0$ gives
$$
(1+s_1+s_2)(1+q_1+q_2+q_3)=1
$$
for the Chern classes of $S$ and $Q$. We can compute
$$
(1+s_1+s_2)^{-1} = 1- (s_1+s_2) + (s_1+s_2)^2 - (s_1+s_2)^3 + (s_1+s_2)^4  - \ldots
$$
which allows us to compute $q_i$ in terms of $s_1$ and $s_2$. It also gives us relations on $s_i$ by looking at the coefficients of degree $4$, $5$ and $6$.
$$
s_2^2 - 3s_1^2 s_2 + s_1^4 =-3s_1s_2^2 + 4s_1^3 s_2 - s_1^5= -s_2^3 + 6s_1^2 s_2^2 - 5s_1^4 s_2 +s_1^6=0
$$
After some boring linear algebra, we use these equations to get 
$$s_1^6 = 5s_2^3, ~
s_1^4 s_2 =2s_2^3,~s_1^2s_2^2 = s_2^3.
$$
We also note that $s_2^3$ is the Poincare dual of a point. Indeed, a section of $S_2^\vee$ gives the lines inside a $\C\P^3\subset \C\P^4$, and intersection of three generic hyperplanes in $\C\P^4$ is a line.

\smallskip
We have $c(S^\vee) = 1- s_1 +s_2 = (1+\alpha)(1+\beta)$ for the Chern roots $\alpha,\beta$, and 
$$
c({\rm Sym}^5 (S^\vee)) = (1+5\alpha)(1+4\alpha + \beta)\cdots(1+5\beta)
$$
so 
\begin{align*}
c_6({\rm Sym}^5 (S^\vee)) &= 25\alpha\beta(4\alpha + \beta)(\alpha + 4\beta)(3\alpha + 2\beta)(2\alpha + 3\beta)
\\
&=25 s_2(4(s_1^2-2s_2) + 17s_2)(6(s_1^2-2s_2) + 13s_2)
\\
&=
25 s_2(4s_1^2 + 9s_2)(6s_1^2 + s_2)
\\& = 600 s_1^4 + 1450 s_1^2 s_2 + 225 s_2^3 
\\&
= (1200 + 1450 + 225)s_2^3 = 2875\,{\rm P.D.(point)}
\end{align*}
from which we conclude that we expect $2875$ lines in $Q$. 

\smallskip
\begin{rem}
The reason we only ``expect" this number of lines has to do with the transversality. To know that this is the precise number, we need to ensure that image of $s_F$ intersects the zero section of ${\rm Sym}^5(S^\vee)$ transversely, which we will not do. Note also that there exist smooth quintics $Q$ which have {\blue infinitely} many lines in them, see Exercise 1.
\end{rem}

\smallskip
One can similarly ask about the number of {\blue degree two} Riemann spheres in $Q$. What we mean by this is looking at smooth conic curves in some $\C\P^2\subset \C\P^4$. The space of $\C\P^2$-s inside of $\C\P^4$ is a Grassmannian $\Gr(3,5)$ of dimension $6$, there is a dimension $5$ family of conics in each $\C\P^2$, and the restriction of $F$ to the conic is a polynomial of degree $10$, with $11$ coefficients. So the expected dimension of the solution space is zero, and we expect a finite number. Indeed, this number has been found to be $609250$.
There is a similar number $n_d$ for every integer degree $d>0$, at least conjecturally.

\smallskip
In the early 1990s physicists\footnote{ We call string-theorists physicists, but not all physicists agree.} came up with, at the time, a completely weird way of computing $n_d$, \cite{CdOGP}. It involved looking at some hypergeometric functions which are solutions of certain linear ODEs, and doing some strange change of variables in formal power series. In particular, they predicted $n_3 = 317206375$, which was confirmed by mathematicians by a much more complicated calculation. This has lead to an explosion in mathematical research inspired by string theory. It goes under the moniker {\blue mirror symmetry} for the following reason. There are certain invariants of algebraic varieties known as Hodge numbers, and for the smooth quintic $Q$ they are
$$
\begin{array}{c}
1\\
0\hskip 24pt 0\\
0\hskip 24pt 1\hskip 24pt 0\\
1 \hskip 17pt 101 \hskip 13pt 101 \hskip 17pt 1\\
0\hskip 24pt 1\hskip 24pt 0\\
0\hskip 24pt 0\\
1
\end{array}
$$
while the ODEs come from looking at the {\blue mirror manifolds} with Hodge numbers:
$$
\begin{array}{c}
1\\
0\hskip 24pt 0\\
0\hskip 20pt 101\hskip 20pt 0\\
1 \hskip 24pt 1 \hskip 24pt 1 \hskip 24pt 1\\
0\hskip 20pt 101\hskip 20pt 0\\
0\hskip 24pt 0\\
1
\end{array}
$$
So the Hodge numbers are flipped across the diagonal line, hence the term ``mirror".  I want to stress that this is a very strange phenomenon -- some invariants of the tangent bundle on one manifold are equal to the same invariants of the cotangent bundle on another manifold.

\smallskip
I will now briefly mention different avenues of research inspired by this calculation.
\begin{itemize}
\item {\bf Gromov-Witten invariants.} This is a way to rigorously define counts of curves in algebraic varieties which have certain properties, for example passing through a given number of points or lying in some subvarieties. The main technical tool is Kontsevich's moduli spaces of stable maps, see \cite{Kontsevich,FP}. In particular, Givental in 1996 proved the original physicists' statement about counts of curves in a quintic.

\item {\bf More examples of mirrors.} In 1992 Batyrev \cite{Batyrev} realized that the original quintic example has a combinatorial underpinning in terms of so-called {\blue reflexive polytopes}. The geometric way of thinking about it is passing from hypersurfaces in $\C\P^4$ to hypersurfaces in related, but more complicated, spaces known as {\blue Gorenstein toric Fano varieties}, encoded by the so-called reflexive polytopes. In dimension two, there are $16$ reflexive polytopes, up to equivalence, but the numbers grow fast. There are $4319$ such polytopes in $\dim =3$, there are $473800776$ of them for $\dim =4$ (computed by Kreuzer and Skarke \cite{KS}), and $\dim =5$ case seems hopeless. However, the fact that there are so many in dimension four (these give Calabi-Yau threefolds) was a surprise, and not a pleasant one, as it was originally hoped that there would be relatively few Calabi-Yau threefolds, and their special properties would indicate that string theory constraints lead to some highly constrained geometry and thus some simple description of the physical world. 

\smallskip
My contribution to this subject was to extend the construction from hypersurfaces to {\blue complete intersections} in toric varieties. Jointly with Batyrev, we proved the duality of appropriately defined\footnote{ These  Hodge numbers were defined by Batyrev and Kontsevich.} Hodge numbers for these complete intersections.

\item{\bf Classifying Calabi-Yau threefolds.} It is currently unknown if Calabi-Yau threefolds fall into a finite number of topological types, or even have a bound on their Euler characteristics, although some important special cases have been settled positively. Also, after allowing certain {\blue conifold transitions}, one may hope that all of these threefolds are connected to each other, this is the so-called Miles Reid's dream.\footnote{ I have also seen it attributed to Hirzebruch.} As far as I know, this is wide open, and it is not clear if we have the tools to make progress, although all our current constructions only lead to finitely many families.

\item{\bf Open strings.} There is a version of mirror symmetry that deals, from the physical point of view, with open strings propagating on Calabi-Yau manifolds. It was proposed by Kontsevich that the boundary conditions on these strings should be the so-called derived category of coherent sheaves on one side and the derived Fukaya category on the other side. The statement that these triangulated categories for mirror manifolds are equivalent, and the area of research that aims at clarifying it, are known as the {\blue Homological Mirror Symmetry}. It is the most active current area in mirror symmetry, combining ideas from algebraic and symplectic geometry.

\item{\bf Generalized varieties.} Constructions of mirror symmetry naturally lead people to consider some generalizations of varieties, such as Deligne-Mumford stacks, Landau-Ginzburg models, and various flavors of noncommutative geometry.
\end{itemize}

{\bf Exercise 1.} Prove that the Fermat quintic threefold $\{x_0^5+ \ldots + x_4^5\}=0$ is a smooth subvariety of $\C\P^4$. Verify
that for any point $(a:b:c)$ on the curve $\{a^5+b^5+c^5 =0\}\subset \C\P^2$ the line
$$
(x_0:\ldots:x_4) = (u: \ee^{\frac {\pi \ii}5} u: a v: b v: cv),~(u:v)\in \C\P^1
$$
lies in the Fermat quintic.

{\bf Exercise 2.} Follow our method to deduce that one expects to find $27$ lines on a cubic surface $F=0$ in $\C\P^3$. In fact, in this case every smooth complex cubic surface has exactly $27$ lines on it. We have seen this from another perspective in Section \ref{sec.dP}.

{\bf Exercise 3.} A (generic) degree $d$ rational curve in $\C\P^4$ is given by 
$$
(x_0:\ldots:x_4) = (f_0(u,v):\ldots :f_4(u,v)),~(u:v)\in \C\P^1
$$
where $f_i$ are generically chosen homogeneous polynomials on degree $d$ in variables $u$ and $v$. Give a rough count of the dimension of the space of such curves and show that it is equal to the number of coefficients of the restriction of the quintic equation to the curve, thus giving us an expected finiteness of the number of curves.
\emph{Hint:} Take into account scaling of the variables and the automorphisms of $\C\P^1$.

\section{Schemes.}
Since we are getting close to the end of these notes, we might as well talk about the language of schemes. It is the standard rigorous 
framework of algebraic geometry, introduced by Grothendieck and his collaborators. We will do the best we can within the constraints of one section.

\smallskip
We start with the definitions of presheaves and sheaves.
\begin{defn}\label{presheaf}
Let $X$ be a topological space. A {\blue presheaf} $\mathcal F$ of abelian groups is a contravariant functor from the category of open sets in $X$ together with inclusions to the category of abelian groups. In more concrete terms, the data of  $\mathcal F$ are:
\begin{itemize}
\item abelian groups $\mathcal F(U)$, one for each open set $U\subseteq X$,
\item group homomorphisms $\rho_{V\subseteq U}:\mathcal F(U) \to\mathcal F(U) $
for all inclusions $V\subseteq U$,
with the property that for $W\subseteq V\subseteq U$ there holds 
$$\rho_{W\subseteq U}= \rho_{W\subseteq V}\circ\rho_{V\subseteq U}.$$
We also require $\rho_{U\subseteq U}={\rm id}_{\mathcal F(U)}$.
\end{itemize}
\end{defn}

\smallskip
\begin{rem}
The abelian groups ${\mathcal F(U)}$ are called the groups of sections of $\mathcal F$ on $U$.
The maps $\rho$ are called the {\blue restriction maps}. Because the notation is somewhat heavy, we will just say which open set we want to restrict to and mean by this the image of the corresponding map of the groups of sections.
\end{rem}

\smallskip
\begin{defn}\label{sheaf}
A presheaf of abelian groups $\mathcal F$ is called a {\blue sheaf} if it satisfies the following gluing property. For any open set $U\subseteq X$ and any open cover
$$
U = \bigcup_\alpha U_\alpha,
$$
and any collection of $s_\alpha\in \mathcal F(U_\alpha)$ such that for all $\alpha,\beta$ the restrictions of $s_\alpha$ and $s_\beta$ to $U_\alpha \cap U_\beta$ are the same, there exists a unique $s\in \mathcal F(U)$ that restricts to all $s_\alpha\in \mathcal F(U_\alpha)$.
\end{defn}

\smallskip
\begin{rem}
It is common to put a condition $\mathcal F(\emptyset) = 0$ for both presheaves and sheaves. Technically, since we can do an empty cover of an empty set, the uniqueness of gluing implies this statement for sheaves.\footnote{ This reminds me of the statement that the determinant of a $0\times 0$ matrix is $1$, because it is the sum of $0!=1$ terms, the term is a product of $0$ entries (so it is equal to $1$), and the sign of the term is positive because there are zero inversions.}
\end{rem}

\smallskip{\blue \hrule}
Here are some examples of sheaves.
\begin{itemize}
\item Let $X$ be a real manifold. For each open set $U\subseteq X$ we define $\mathcal F(U)$ to be the space of continuous real-valued functions on $U$. The restriction maps are, obviously, the restriction maps. It is a sheaf, because 
any collection of continuous functions on open subsets  $U_\alpha$ which is compatible on the intersections naturally glues to a continuous function on $\bigcup_\alpha U_\alpha$.

\item Let $\pi:W\to X$ be a holomorphic vector bundle over $X$. We denote by $\mathcal F(U)$ the space of holomorphic sections of 
$\pi^{-1}U\to U$, and the restriction maps are again the restriction maps.
\end{itemize}

\smallskip
\begin{rem}
An example of a presheaf that is not a sheaf on a real manifold $X$ can be given by {\blue constant} functions. The problem is that if we have a disjoint union $U=U_1\sqcup U_2$ then constant functions on $U_1$ and $U_2$ do not glue to a constant function on $U$. On the other hand, {\blue locally constant} functions do form a sheaf.
\end{rem}

\smallskip{\blue \hrule}
{\bf Grothendieck's idea.} For any commutative (associative, with $1$) ring $A$ one can construct a topological space ${\rm Spec}\, A$ with a sheaf of rings $\mathcal O$ on it.
As a set, 
$${\rm Spec} \,A= \{{\rm prime~ideals~of~}A\}.
$$
The topology is given by the Zariski topology, i.e. for any subset (equivalently, ideal) $I$ of $A$ the
set 
$$V(I):=\{p\in {\rm Spec} \,A, p\supseteq I\}$$
is declared a closed set.

\smallskip
We will now construct the sheaf $\mathcal O$ on ${\rm Spec}\,A$. For any prime ideal $p$ we can consider the localization $A_p$ of $A$ which is the set of formal fractions $\frac as$ with $a\in A$ and $s\not\in p$, up to the equivalence relation
$$
\left(
\frac {a_1}{s_1} \sim \frac {a_2}{s_2}\right) \iff \left({\rm there~exists~}s_3\not\in p,{\rm~such~that~}s_3(s_1a_2-s_2a_1)=0\right).
$$
There is a ring homomorphism $A\to A_p$, called the localization map,  which sends $a$ to $\frac a1$. 
Informally, $\mathcal O$ is the ``sheaf of collections of elements in $A_p$, which are locally given by a fraction". The precise definition is the following. For a Zariski open subset $U\subseteq {\rm Spec}\,A$ we define
$$
\mathcal O(U):=\left\{
\begin{array}{c}
{\rm collections~of~}\alpha_p\in A_p {\rm ~for~all~}p\in U,\\
{\rm such~that ~for~each~}p\in U~{\rm there~exists~}p\in U_1\subseteq U, a_1,s_1\in A,\\
{\rm so~that~for~all~}q\in U_1~{\rm there~holds~}\alpha_q= \frac {a_1}{s_1}\in A_q.
\end{array}
\right\}
$$
In particular, in the above definition $\frac{a_1}{s_1}$ makes sense in $A_q$, i.e. $s_1\not\in q$ for all $q\in U_1$.

\smallskip
It is clear what the restriction maps are -- we just use the same $\alpha_p$ but for a potentially smaller Zariski open subset.
It is also clear that each $\mathcal O(U)$ comes with a ring structure,  induced from the ring structure of the localizations,
see Exercise 1.
What is not at all clear is how one can compute any groups $\mathcal O(U)$, for example $\mathcal O({\rm Spec}\,A)$.
The following key proposition answers this in a very satisfactory way.

\smallskip
\begin{prop}
For every $a\in A$, the collection of elements $\frac a1$ in $A_p$ for all $p$ gives a section of $\mathcal O$ on ${\rm Spec}\,A$. 
The resulting map $A\to \mathcal O({\rm Spec}\,A)$ is a ring isomorphism. 
In particular, one can recover the ring $A$ back from the data $({\rm Spec}\,A,\mathcal O)$. 
\end{prop}

\begin{proof}
The proof is left to the reader of Hartshorne\footnote{ \cite[Proposition II.2.2]{Hartshorne}}.
\end{proof}

\smallskip
{\blue \hrule}
It is important to define morphisms of schemes. If we have a ring homomorphism $f:A\to B$, 
we get the map of sets
$$
{\rm Spec}\,B\to {\rm Spec}\,A
$$
which sends the prime ideal $p\subset B$ to the prime ideal $f^{-1}(p)\subset A$.\footnote{ Preimages of prime ideals are prime, while the same can not be said about the maximal ideals. This is a good reason why we want to include all prime ideals into ${\rm Spec}$.}
This is a continuous map of topological spaces, see Exercise 2. There is also some additional data that mimics 
the pullback of holomorphic functions for holomorphic maps of complex manifolds.

\smallskip
We will not get into the technical details, but one defines schemes as topological spaces $X$ with sheaves of rings $\mathcal O_X$ on them which locally look like $({\rm Spec}\,A,\mathcal O)$, and morphisms are defined in such a way that scheme morphisms from ${\rm Spec}\,B$
to ${\rm Spec}\,A$ are in natural bijection with ring homomorphisms $A\to B$.

\smallskip
Here are some notable examples of schemes.
\begin{itemize}
\item ${\rm Spec}\,\C$. This is my favorite. It consists of just one point, namely the prime ideal $\{0\}\subset \C$. The 
sections of $\mathcal O$ on it are $\C$.

\item ${\rm Spec}\,k$ for any field $k$. This is also just a single point, but the sections of $\mathcal O$ are now $k$.

\item ${\rm Spec}\,\C[x_1,\ldots, x_n]$. Algebraic geometers often denote this scheme by $\C^n$, but in addition to the usual points on $\C^n$ which correspond to maximal ideals of $\C[x_1,\ldots, x_n]$ we have all sorts of other points that encode prime ideals.
Note that the ring homomorphism $\C\to \C[x_1,\ldots, x_n]$ gives a scheme map  ${\rm Spec}\,\C[x_1,\ldots, x_n]\to
{\rm Spec}\,\C$. This is a particular case of a ``scheme over $\C$", which is equivalent to saying that all of the rings are in fact 
$\C$-algebras. One can recover the usual points of $\C^n$ as maps ${\rm Spec}\,\C\to {\rm Spec}\,\C[x_1,\ldots, x_n]$ 
of schemes over $\C$.

\item ${\rm Spec}\,\C[t]/(t^2)$. Again, as a topological space, it is just a single point, but it is what's called a ``fat point". Maps 
from it to a $\C$-scheme $X$ are in bijection to tangent vectors to $X$. This example is a good motivation for allowing arbitrary rings and not just integral domains. 
\end{itemize}

\smallskip
Where there are rings, there are modules. To every module $M$ over a commutative ring $A$ one can associate a sheaf of $\mathcal O$-modules over ${\rm Spec}\,A$, denoted by $\tilde M$. It is built from localizations of $M$ in the same way as $\mathcal O$ but
with numerators of fractions now in $M$ instead of $A$.
Then one defines a quasi-coherent sheaf on a scheme $X$ as a sheaf of $\mathcal O_X$-modules that locally look like $\tilde M$.

\smallskip
We can in particular look at locally free quasi-coherent sheaves of finite rank $r$ as analogs of sheaves of sections of vector bundles over a smooth complex manifold.\footnote{ This is why we have been using calligraphic $\mathcal O$ in our $\mathcal O(k)$ notation for line bundles on $\C\P^n$.} There is also a concept of cohomology groups of arbitrary sheaves, which we will not go into. 

\smallskip
The richness and flexibility of working with arbitrary commutative\footnote{ Noncommutative rings are not entirely hopeless, but localization does not seem to work very well unless denominators are central, or at least normal. This really limits attempts of extending the theory.} rings give scheme theory an advantage over the more naive point-based formulations. The disadvantage is a layer of bureaucracy that takes some getting used to.
There is obviously a lot more one can say about this topic, but hopefully this section gives the reader a flavor of the language of schemes, and a bit of a preview, if they are to go on to study algebraic geometry in rigorous detail.

\smallskip
{\bf Exercise 1.} Prove that for two sections $(\alpha_p)_{p\in U}$ and $(\beta_p)_{p\in U}$ of $\mathcal O(U)$ their sum and product 
$(\alpha_p+\beta_p)_{p\in U}$, $(\alpha_p\beta_p)_{p\in U}$ are also sections of $\mathcal O(U)$.

{\bf Exercise 2.} For a ring homomorphism $f:A\to B$, prove that the map 
$
{\rm Spec}\,B\to {\rm Spec}\,A
$
which sends  $p\subset B$ to $f^{-1}(p)\subset A$ is continuous in Zariski topology.

{\bf Exercise 3.} Give an example of a commutative ring homomorphism so that the preimage of a maximal ideal is prime, but not maximal.

\section{Modular curves and modular forms.}
As we know by now, elliptic curves are given by $\C/L$ where $L$ is a discrete rank two additive subgroup of $\C$. 
For any choice of free generators $(l_1,l_2)$ of $L$, exactly one of the fractions $\frac {l_1}{l_2}$ and $\frac{l_2}{l_1}$ has a positive imaginary part. We call the ordered pair $(l_1,l_2)$ an oriented basis if ${\rm Im}(\frac{l_2}{l_1})>0$, in other words, $l_2$ is located counterclockwise from $l_1$. 

\smallskip
For any oriented basis $(l_1,l_2)$ of $L$, we define $\tau = \frac {l_2}{l_1}$ in the upper half plane $\mathcal H =\{{\rm Im}\,\tau>0\}$ and observe that
$$
\C/L \cong \C/(\Z  + \Z \tau),
$$
with the isomorphism induced by the multiplication by $l_1^{-1}$. Conversely, if we scale $L$ to make it in the form $\Z+ \Z\tau$, the preimages of $1$ and $\tau$ form an oriented basis.

\smallskip
What happens to $\tau$ if we pick another basis? Different oriented bases are related by the group ${\rm SL}(2,\Z)$ of orientation-preserving automorphism of $\Z^2$. If 
$$
\left(
\begin{array}{c}
l_2'\\
l_1'\\
\end{array}
\right)
=\left( \begin{array}{cc}
a&b\\
c&d\\
\end{array}
\right)
\left(
\begin{array}{c}
l_2\\
l_1\\
\end{array}
\right)
$$
then
$$\tau'=\frac {l_2'}{l_1'} = \frac {a l_2 + b l_1 }{c l_2 + d l_1} 
=
\frac {a\frac {l_2}{l_1} + b }{c\frac {l_2}{l_1} + d } = \frac {a\tau + b}{c\tau + d}.
$$
Therefore, the action of ${\rm SL}(2,\Z)$ on the upper half plane that sends $\tau$ to  $\frac {a\tau + b}{c\tau + d}$ preserves the isomorphism class of the elliptic curve $\C/(\Z+ \Z\tau)$. 
The converse is also true, see Exercise 1.
Thus we see that the set of all elliptic curves up to isomorphism is the quotient of the upper half-plane $\mathcal H$ by
the above action of the group ${\rm SL}(2,\Z)$. 

\smallskip
It is possible to explicitly describe a fundamental domain of this action, see the picture below. 

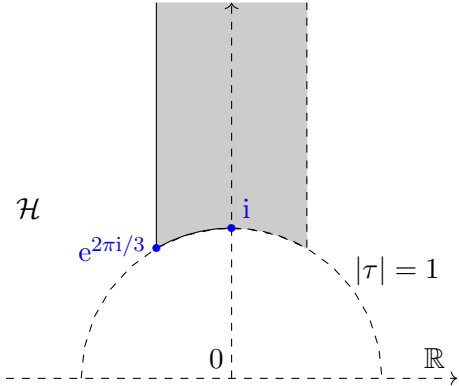
\begin{tikzpicture}
\filldraw [gray!40!white] (-1,5)--
(-1.000, 1.732)--(-0.8135, 1.827)--(-0.6180, 1.902)--(-0.4158, 
  1.956)--(-0.2091, 1.989)--(0, 2.000)--(0.2091, 1.989)--(0.4158, 
  1.956)--(0.6180, 1.902)--(0.8135, 1.827)--(1.000, 1.732)--(1,5);
\draw[dashed]
(-2.000, 0)--(-1.956, 0.4158)--(-1.827, 0.8135)--(-1.618, 
  1.176)--(-1.338, 1.486)--(-1.000, 1.732)--(-0.6180, 
  1.902)--(-0.2091, 1.989)--(0.2091, 1.989)--(0.6180, 1.902)--(1.000, 
  1.732)--(1.338, 1.486)--(1.618, 1.176)--(1.827, 0.8135)--(1.956, 
  0.4158)--(2.000, 0);
\draw (-1,5)--
(-1.000, 1.732)--(-0.8135, 1.827)--(-0.6180, 1.902)--(-0.4158, 
  1.956)--(-0.2091, 1.989)--(0, 2.000);
\draw[dashed](0, 2.000)--(0.2091, 1.989)--(0.4158, 
  1.956)--(0.6180, 1.902)--(0.8135, 1.827)--(1.000, 1.732)--(1,5);
\draw [dashed,->](-3,0)--(3,0);
\draw [dashed,->](0,0)--(0,5);
\node at (2.7,0) [anchor = south] {$\R $};
\node at (1.000+.5, 1.732-.3)[anchor = west] {$|\tau|=1 $};
\node[blue] at (0,2)  {\hskip 3pt\circle* 3};
\node[blue] at (.2,2)  [anchor = south] {$\ii$};
\node[blue] at (-1.000, 1.732)  {\hskip 3pt\circle* 3};
\node[blue] at (-1.000, 1.732)   [anchor =east] {$\ee^{2\pi\ii/3}$};
\node at (-2.7,2) [anchor = south] {$\mathcal H$};
\node at (-.2,0)[anchor = south] {$0$};
\end{tikzpicture}

\noindent
We do not prove that this is the fundamental domain but instead refer to Chapter 7 of \cite{Serre}. 
Most points of $\mathcal H$
have the stabilizer $\{\pm {\rm Id}\}$ which corresponds to the Kummer involution $z\mapsto (-z)$ on $\C/L$. For $\tau=\ii$, the stabilizer is isomorphic to $\Z/4\Z$, generated by $\left(\begin{array}{cc}0&1\\-1&0\end{array}\right)$. This corresponds to the square lattice of Gaussian integers with the extra automorphism that is the multiplication by $\ii$. Of course, we have conjugate stabilizers in the ${\rm SL}(2,\Z)$ orbit of $\tau=\ii$. Similarly, for $\tau = \ee^{\frac {2\pi\ii}3}$, the stabilizer is $\Z/6\Z$, generated by $\left(\begin{array}{rr}1&1\\-1&0\end{array}\right)$. The lattice is the hexagonal lattice of Eisenstein integers, again with more automorphisms.

\smallskip
We will now define modular forms.

\smallskip
\begin{defn}
A holomorphic function $f:\mathcal H \to \C$ is called a modular form of weight $k$ for the group ${\rm SL}(2,\Z)$ if 
it satisfies the following properties.
For every $\left( \begin{array}{cc}a&b\\c&d\end{array}\right)\in {\rm SL}(2,\Z)$ there holds 
$$
f\left(\frac {a\tau+b}{c\tau+d}\right) = (c\tau+d)^k f(\tau).
$$
This means in particular that $f(\tau+1)=f(\tau)$, so $f$ can be written as a Laurent power series in
$q=\ee^{2\pi\ii\tau}$ at $q=0$. We furthermore assume that $f$ extends to a holomorphic function in
a neighborhood of $q=0$, i.e. $f(\tau) = \sum_{n\geq 0}a_n \ee^{2\pi\ii n\tau}$.
\end{defn}

We remark that the weight of a nonzero modular form for ${\rm SL}(2,\Z)$ has to be even, because $-{\rm id}\in {\rm SL}(2,\Z)$ implies that 
$f(\tau) = (-1)^k f(\tau)$.

\smallskip
Examples of modular forms are provided by the coefficients of the Weierstrass function $\mathcal P$ (for the
lattice $L=\Z+\Z\tau$). Recall that
$$
\mathcal P(z) = \frac 1{z^2} + \sum_{0\neq l \in L} \left(\frac 1{(z-l)^2}-\frac 1{l^2}\right).
$$
Coefficients at $z^k$ in the Laurent are found by integrating $z^{-k-1}\mathcal P(z)$ over a small circle around the origin,
so absolute uniform convergence of the series implies that one can simply add up the corresponding coefficients of the expansion. We have
$$
\frac 1{(z-l)^2}-\frac 1{l^2}= \sum_{k\geq 1} (k+1) \frac {z^k}{l^{k+2}}.
$$
Thus, the coefficients of $\mathcal P$ at $z^k$ for $k\geq 1$ are, up to constant factor, given by
$$
\sum_{l\neq 0} \frac 1{l^{k+2}} = \sum_{(m,n)\neq (0,0)} \frac 1{(m\tau + n)^{k+2}}.
$$
This prompts the following definition.

\smallskip
\begin{defn}\label{defG}
We define the Eisenstein series
$$G_k(\tau) =  \sum_{(m,n)\neq (0,0)} \frac 1{(m\tau + n)^{k}}$$
for $k\geq 3$. 
\end{defn}

\smallskip
\begin{prop}
The Eisenstein series $G_k$ is zero for odd $k$. For even $k\geq 4$, $G_k$ is a modular form of weight $k$ 
for  ${\rm SL}(2,\Z)$.
\end{prop}

\begin{proof}
It is easy to see that the series $E_k$ is absolutely convergent.
We have
$$
G_k\left(\frac {a\tau + b}{c\tau + d}\right) = \sum_{(m,n)\neq (0,0)} \frac 1{(m\frac {a\tau + b}{c\tau + d}+n)^k}
= \sum_{(m,n)\neq (0,0)} \frac {(c\tau+d)^k}{((ma+nc)\tau + mb+nd)^k}
$$
which equals $(c\tau+d)^k G_k(\tau)$
by reindexing the summation. The first statement then follows from modularity (or just observing that $(m,n)$ and $(-m,-n)$ terms cancel out).
\end{proof}

It is common to normalize the series $G_{k}$ to give it the constant term $1$. Moreover, one can write explicitly the 
$q$-expansions.
$$
E_k(\tau) = \frac {G_k(\tau)}{2\zeta(k)} = 1 - \frac {4k}{B_k} \sum_{n} \sum_{d|n} d^{k-1} q^n
$$
where $q=\ee^{2\pi\ii \tau}$, $\zeta(k) =\sum_{n\geq 1}\frac 1{n^k}$ and $B_k$ are the Bernoulli numbers.\footnote{ We do not prove it in general, but treat the $k=4$ case in Exercise 3.} In particular, we have
\begin{equation}\label{E4E6}
E_4(\tau) = 1 + 240  \sum_{n} \sum_{d|n} d^3 q^n,~~
E_6(\tau) = 1 - 504  \sum_{n} \sum_{d|n} d^5 q^n.
\end{equation}

\smallskip
The ring of modular forms has a remarkably simple structure.

\smallskip
\begin{thm}
Every modular form is a polynomial in $E_4(\tau)$ and $E_6(\tau)$ with constant coefficients.
\end{thm}

\begin{proof}
We will only sketch the argument, see \cite{Serre}. If we denote by $m_\infty$ the order of vanishing of a
nonzero weight $k$ modular $f(\tau)$ as a power series in $Q$, and by $m_p$ order of vanishing of it at a point $p\in \mathcal H$, then we have
\begin{equation}\label{sumSerre}
m_\infty + \frac 12 m_{\ii} + \frac 13 m_{\ee^{2\pi\ii/3}} +\tilde \sum_{p}m_p = \frac k{12}
\end{equation}
where $\tilde\sum$ is the sum over other zeros of $f$ in the fundamental domain.
This is not too hard to prove by integrating $\frac {f'}f$ over the boundary of the fundamental domain.

\smallskip
Then one can prove that $E_6^2- E_4^3$ is not vanishing anywhere on $\mathcal H$ and has a simple zero at $q=0$. One can use this to argue that the dimension of the space of weight $k+12$ modular forms is one larger than that of weight $k$ modular forms, and the rest is fairly easy.
\end{proof}

\smallskip
The set of orbits of ${\rm SL}(2,\Z)$ on $\mathcal H$ can be mapped to $\C\P^1$ by
\begin{equation}\label{E62E43}
\tau \mapsto (E_6^2(\tau):E_4^3(\tau)).
\end{equation}
Indeed, when you change $\tau$ to $\frac {a\tau +b}{c\tau+d}$, both $E_6^2$ and $E_4^3$ acquire a factor of $(c\tau +d)^{12}$. We can also argue that $(E_4^3 : E_6^2)$ determines the elliptic curve uniquely, up to isomorphism, for example by looking at the equation of the curve in its Weierstrass embedding into $\C\P^2$. This identifies $\mathcal \H/{\rm SL}(2,\Z)$ with $\C\P^1\setminus (1:1)$. We compactify  $\mathcal \H/{\rm SL}(2,\Z)$ to $\C\P^1$ by adding the {\blue cusp} $q=0$ (a.k.a. $\ii\infty$).

\smallskip
\begin{rem}
The map \eqref{E62E43} looks suspiciously like mapping to $\C\P^N$ using sections of a line bundle. Indeed, one can think of modular forms as sections of line bundles, but not on a variety but rather on a so-called Deligne-Mumford stack. This is beyond what we could cover in these notes.
\end{rem}

\smallskip
One is often interested in quotients of the upper half plane by other groups. The following are the most well-studied.
They are given by matrices $\gamma =\left( \begin{array}{cc}a&b\\c&d\end{array}\right)$ in ${\rm SL}(2,\Z)$ 
with certain properties modulo a fixed integer $n$. Specifically, we have subgroups 
$\Gamma(n)$, $\Gamma_1(n)$ and $\Gamma_0(n)$ of ${\rm SL}(2,\Z)$ given by the conditions
\begin{align*}
\gamma = \left( \begin{array}{cc}1&0\\0&1\end{array}\right)\,{\rm mod}\,n,~
 \gamma = \left( \begin{array}{cc}1&*\\0&1\end{array}\right)\,{\rm mod}\,n,~
 \gamma = \left( \begin{array}{cc}*&*\\0&*\end{array}\right)\,{\rm mod}\,n
\end{align*}
respectively. One can, and does, talk about modular forms of weight $k$ with respect to these groups. 
The quotients of $\mathcal H$ by these groups are compactified by adding a finite number of cusps, corresponding to the orbits of the group on $\Q\sqcup\{\ii\infty\}$. The resulting curves $X(n)$, $X_1(n)$ and $X_0(n)$ have genus $g$ that grows roughly as $c\,n^3$, $c\,n^2$ and $c\,n$ respectively.

\smallskip
We finish this section by a couple of fascinating observations that connect the theory of modular forms with that of the largest sporadic finite simple group $\mathbb M$ of size approximately $8\cdot 10^{53}$, known as the Monster. 
Consider the {\blue $j$-invariant} 
$$
J(\tau) = (12)^3 \frac {E_4^3}{E_4^3-E_6^2}.
$$
It has a Laurent power series expansion in $q=\ee^{2\pi\ii\tau}$ 
$$
J(\tau) = \frac 1q + 744 + 196884\,q + 21493760\,q^2 + \ldots
$$
and a remarkable observation, due to John McKay, is that the coefficients at small positive degrees of $q$ are {\blue small positive linear combinations} of dimensions of the irreducible complex representations of $\mathbb M$! For example, the smallest irrep of the Monster, other than
the trivial one-dimensional representation, has dimension $196883$.

\smallskip
Another interesting relation is the following. The involution $w_n:\mathcal H\to \mathcal H$ defined by 
$$\tau\to -\frac 1{n\tau}$$ 
normalizes $\Gamma_0(n)$ and thus acts on the modular curve $X_0(n)$. Then for prime $n$ the quotient $X_0(n)^+ = X_0(n)/(w_n)$ has genus zero if and only if $n$ divides the size of $\mathbb M$,
i.e.
$$
n\in\{
2,3,5,7,11,13,17 ,19 , 23 ,29 , 31, 41,47 ,59 ,71\}.
$$
Remarkably, both of these observations have something to do with so-called vertex operator algebras, which is an algebraic construction inspired by string theory.

\smallskip
It is fair to say that the theory of modular curves and modular forms is a fascinating topic, on the intersection of algebraic geometry, arithmetic geometry, analytic number theory, and even string theory. We barely scratch the surface here.

{\bf Exercise 1.} Prove that if two elliptic curves $\C/L$ and $\C/L'$ are isomorphic as complex manifolds, then there exists $\lambda \in \C$ such that $L'= \lambda L$. \emph{Hint:} Argue that the space of holomorphic one-forms on $\C/L$ is one-dimensional. Then (after arranging that the isomorphism preserves $0$) compare the pullback of the differential form $dz$
on $\C/L'$ with $dz$ on $\C/L$ and use $t=\int_{0}^t \,dz$.

{\bf Exercise 2.} Prove that $G_4(\ee^{2\pi\ii/3})=G_6(\ii)=0$.

{\bf Exercise 3.} Prove the formula for $E_4$ in \eqref{E4E6}.
\emph{Hint:} The constant term in $G_4(\tau)$ is $2\zeta(4) = \frac {\pi^4}{45}$.
By splitting the sum in Definition \ref{defG} into sums for fixed $m$, it suffices to show
$
\sum_{n\in \Z} \frac 1{(x + n)^4} = \frac {8\pi^4}3\sum_{r\geq 0} r^3 \ee^{2\pi\ii r x},
$
which follows by differentiation from the well-known formula
$
\sum_{n\in \Z} \frac 1{(x + n)^2}  =\frac {\pi^2}{\sin^2 \pi x}= - 4\pi^2 \sum_{r\geq 0} r \ee^{2\pi\ii r x}.
$

\section{Toric varieties.}\label{sec.toric}
In this section we will talk about the topic that is dear to my heart, that of toric varieties. They are a special class of algebraic varieties which are encoded by the combinatorial data of lattices and convex cones and polytopes. As always, we will work over the complex numbers.

\smallskip
Throughout this section, by a lattice $M$ we will simply mean a free abelian group $M$. It naturally sits
inside a real vector space $M_\R:=M \otimes_\Z \R$.

\smallskip
\begin{defn}\label{ratcone}
A rational convex polyhedral cone $\sigma \subseteq M_\R$ is defined to the positive span
$$
\sum_i\R_{\geq 0} v_i
$$
of a finite number of elements $v_i\in M$.
\end{defn}

\smallskip
\begin{rem}
One can equivalently define a convex polyhedral cone as an intersection of a finite number of rational subspaces
$$
\bigcap_i\{v\in M_\R,~\mu_i(v)\geq 0\}
$$
where $\mu_i$ are elements of the dual lattice ${\rm Hom}(M,\Z)$ in the dual vector space $M_\R^\vee$. It is rather tedious to prove that these two concepts are equivalent, so we will not do that. The equivalence is clear in dimension two, which will be our primary focus in this section.
\end{rem}

\smallskip
To every rational polyhedral cone $\sigma$ we can associate an {\blue affine toric variety} as follows.

\smallskip
\begin{defn}
The set $\sigma \cap M$ has a natural semigroup structure. We denote by $\C[\sigma \cap M]$ the corresponding semigroup ring. It is a vector space with a basis indexed by $m\in \sigma \cap M$. We denote the basis elements by
$[m]$ and define the multiplication by $[m_1][m_2]=[m_1+m_2]$, naturally extended to the whole semigroup ring.
\end{defn}

It is easy to see that if $\sigma = M_\R$, then $\sigma\cap M=M\cong \Z^{\rk M}$ and $\C[\sigma\cap M]$ is isomorphic to the Laurent polynomial ring in $\rk M$ variables. One can view all $\C[\sigma\cap M]$ as subrings of this ring.

\smallskip
\begin{prop}
The semigroup $\sigma\cap M$ is finitely generated, and therefore $\C[\sigma \cap M]$ is a finitely generated $\C$-algebra.
\end{prop}

\begin{proof}
By Definition \ref{ratcone}, every lattice element $m\in \sigma\cap M$ can be written as a nonnegative linear combination
of the generators $v_i$ of $\sigma$
$$
m = \sum_{i=1}^k \alpha_i v_i,
$$
which we will rewrite as
$$
m = \sum_{i=1}^k \lfloor\alpha_i\rfloor v_i + \sum_{i=1}^k \{\alpha_i\} v_i.
$$
The set $\sum_{i=1}^k[0,1)v_i$ is bounded, so the set of elements in $\sigma\cap M$ that can be written as $\sum_{i=1}^k \{\alpha_i\} v_i$ is finite. This set, together with all of $v_i$, is then seen to generate the semigroup.
\end{proof}

\smallskip
\begin{defn}\label{affinetoric}
The affine toric variety associated to the cone $\sigma$ is defined as
$$
\A_\sigma:={\rm Spec}\,\C[\sigma \cap M].
$$
If someone does not want to use the language of schemes,\footnote{ That would be us.} they should think of $\A_\sigma$ as the set of maximal ideals in $\C[\sigma\cap M]$ or, even better, as a subvariety in $\C^r$ obtained by looking at $r$ generators of $\C[\sigma\cap M]$ and the ideal of their relations.
\end{defn}

\smallskip{\blue \hrule}
We will illustrate the construction with a few examples. In all of them, we will have $M=\Z^2$.
\begin{itemize}
\item Consider $\sigma =\R_{\geq 0}^2$. The semigroup $\sigma\cap M= \Z_{\geq 0}^2$ is freely generated by
$(1,0)$ and $(0,1)$. Therefore, the semigroup ring $\C[\sigma\cap M]$ is isomorphic to $\C[x_1,x_2]$ with $x_1=[(1,0)]$ and 
$x_2=[(0,1)]$, and $\A_\sigma \cong \C^2$.

\item Consider $\sigma = \R_{\geq 0}(1,0) + \R_{\geq 0}(1,2)$, see picture below.

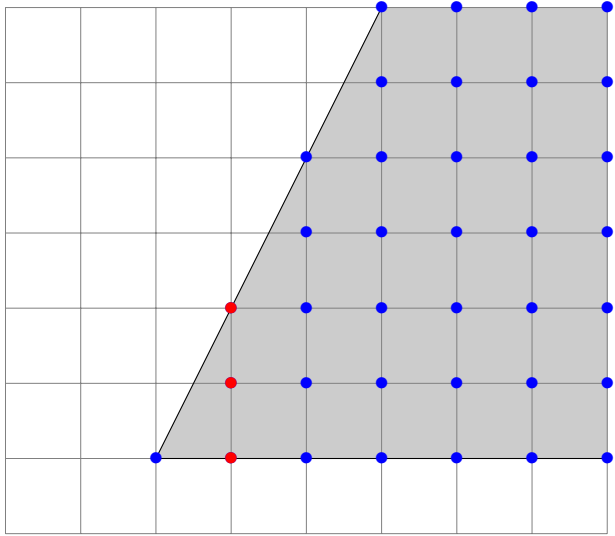
\begin{tikzpicture}
\filldraw[gray!40!white] (0,0)--(3,6)--(6,6)--(6,0)--cycle;
\draw[step=1cm,gray,very thin] (-2,-1) grid (6,6);
\draw (0,0)--(6,0);
\draw (0,0)--(3,6);
\foreach \x in {0,1,2,3,4,5,6} \node[blue] at (\x,0) {$\bullet$};
\foreach \x in {1,2,3,4,5,6} \node[blue] at (\x,1) {$\bullet$};
\foreach \x in {1,2,3,4,5,6} \node[blue] at (\x,2) {$\bullet$};
\foreach \x in {2,3,4,5,6} \node[blue] at (\x,3) {$\bullet$};
\foreach \x in {2,3,4,5,6} \node[blue] at (\x,4) {$\bullet$};
\foreach \x in {3,4,5,6} \node[blue] at (\x,5) {$\bullet$};
\foreach \x in {3,4,5,6} \node[blue] at (\x,6) {$\bullet$};
\foreach \y in {0,1,2} \node[red] at (1,\y) {$\bullet$};
\end{tikzpicture}

The semigroup $\sigma\cap M$ is generated by $(1,0)$, $(1,2)$ and $(1,1)$. The corresponding generators 
$u=[(1,0)]$, $v=[(1,2)]$, $w=[(1,1)]$ satisfy $uv- w^2=0$, so we get precisely the singular surface in $\C^3$ with the $A_1$ singularity at the origin.

\item More generally, if $\sigma= \R_{\geq 0}(1,0) + \R_{\geq 0}(n-1,n)$, then $\C[\sigma\cap M]$ is generated by 
$u=[(1,0)]$, $v=[(n-1,n)]$, $w=[(1,1)]$ which satisfy the relation
$
uv - w^{n} = 0,
$
and we get 
$$
\A_\sigma \cong \{uv-w^n=0\}\subset \C^3,
$$
which is a surface $\C^2/(\Z/n\Z)$ with the $A_{n-1}$ singularity at the origin, considered in Section \ref{sec.ADE1}. 
\end{itemize}

\smallskip
\begin{rem}
The connection to (the beginning of) Section \ref{sec.ADE1} is even more clear if one considers a different lattice $M$, namely
$$
\{(a,b)\in \Z^2,~a=b\,{\rm mod}\,n\}
$$
and the cone $\R_{\geq 0}^2$. More generally, for any finite abelian subgroup $G$ of ${\rm GL}(n,\C)$, the quotient 
$\C^n/G$ is an affine toric variety. Indeed, one can diagonalize the action of $G$, and consider the lattice $M$ of multi-degrees of $G$-invariant Laurent monomials, and the cone $\sigma =\R_{\geq 0}^n$.
\end{rem}

The real power of toric geometry comes from being able to glue affine toric varieties to get more complicated toric varieties,
for example projective spaces. The relevant data is that of the {\blue fans} in the dual space.

\smallskip
\begin{defn}
Let $N\subseteq N_\R$ be a lattice. A fan $\Sigma$ in $N$ is a nonempty finite collection of rational polyhedral cones in $N_\R$ which satisfy the following.
\begin{itemize}
\item All cones $\sigma\in\Sigma$ satisfy $\sigma \cap (-\sigma)=\{{\mathbf 0}\}$. 
\item For all  $\sigma_1,\sigma_2\in \Sigma$, the intersection $\sigma_1\cap\sigma_2$ is a face\footnote{ A face of a cone  $\sigma$ is an intersection of $\sigma$ with the boundary of a linear halfspace that contains it.} in each $\sigma_i$.
\item For any $\sigma\in \Sigma$, all faces of $\sigma$ are in $\Sigma$. In particular, $\{{\mathbf 0}\}\in\Sigma$. 
\end{itemize}
\end{defn}

\smallskip
The stereotypical example of a fan is the following.
The lattice is $\Z^2$, and the fan $\Sigma$ consists of seven cones: three of dimension $2$, three of dimension $1$ and the zero cone $\{(0,0)\}$, namely
\begin{equation}\label{fanCP2}
\begin{array}{c}
\R_{\geq 0} (1,0)+\R_{\geq 0} (0,1),~\R_{\geq 0} (0,1)+\R_{\geq 0} (-1,-1),~
\R_{\geq 0} (-1,-1)+\R_{\geq 0} (1,0),\\
~\R_{\geq 0} (1,0),~\R_{\geq 0} (0,1),~\R_{\geq 0} (-1,-1),~\{(0,0)\}.
\end{array}
\end{equation}

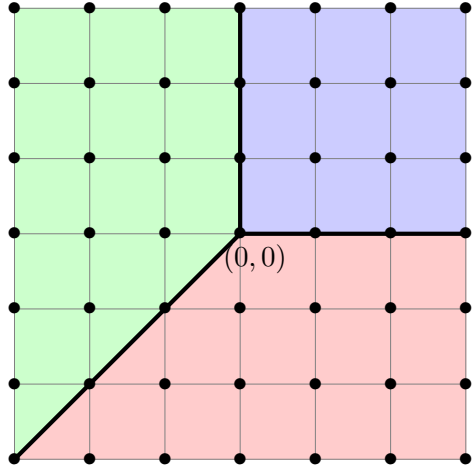
\begin{tikzpicture}
\filldraw[blue!20!white](0,0)--(3,0)--(3,3)--(0,3)--cycle;
\filldraw[red!20!white](0,0)--(3,0)--(3,-3)--(-3,-3)--cycle;
\filldraw[green!20!white](0,0)--(0,3)--(-3,3)--(-3,-3)--cycle;
\draw[step=1cm,gray,very thin] (-3,-3) grid (3,3);
\draw[ultra thick](0,0)--(3,0);
\draw[ultra thick](0,0)--(0,3);
\draw[ultra thick](0,0)--(-3,-3);
\foreach \x in {-3,-2,-1,0,1,2,3} 
\foreach \y in {-3,-2,-1,0,1,2,3} 
\node at (\x,\y) {$\bullet$};
\node at (.2,0) [anchor=north]{$(0,0)$};
\end{tikzpicture}

The precise definition is a bit bothersome, but to every fan $\Sigma$ one can associate a toric variety $\P_\Sigma$ obtained by gluing affine toric varieties ${\rm Spec}\,\C[\sigma^\vee \cap M],~\sigma\in\Sigma$  where $M$ is the lattice dual to $N$ and $\sigma^\vee$ is the cone of linear functions that are nonnegative on $\sigma$. We will now give, obviously without proof, several examples to indicate how some of the varieties we have seen are in fact toric.
\begin{itemize}
\item Consider the lattice $N=\Z$ and the fan $\Sigma =\{\R_{\geq 0},\R_{\leq 0}, \{{\mathbf 0}\}\}$. 

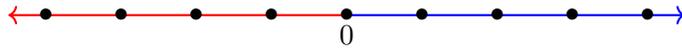
\begin{tikzpicture}
\draw [blue,thick,->](0,0)--(4.5,0);
\draw [red,thick,->](0,0)--(-4.5,0);
\foreach \x in {-4,-3,-2,-1,0,1,2,3,4} 
\node at (\x,0) {$\bullet$};
\node at (0,0) [anchor=north]{$0$};
\end{tikzpicture}

The corresponding toric variety $\P_\Sigma$ is naturally isomorphic to $\C\P^1$. The three ${\rm Spec}\,\C[\sigma^\vee \cap M]$ are (in the usual notations) $x_0\neq 0$, $x_1\neq 0$ and their intersection.

\item For the fan of \eqref{fanCP2}, we have $\P_\Sigma\cong \C\P^2$. The maximum dimensional cones correspond to the usual charts $x_i\neq 0$, and the rest correspond to their intersections.

\item Consider the fan in $\Z^2$ that consists of the cones $\R_{\geq 0} (1,0)+\R_{\geq 0} (1,1)$, $\R_{\geq 0} (1,1)+\R_{\geq 0} (0,1)$ and their faces. 

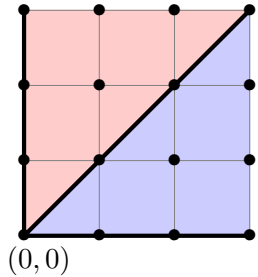
\begin{tikzpicture}
\filldraw[blue!20!white](0,0)--(3,0)--(3,3)--cycle;
\filldraw[red!20!white](0,0)--(0,3)--(3,3)--cycle;
\draw[step=1cm,gray,very thin] (0,0) grid (3,3);
\draw[ultra thick](0,0)--(3,0);
\draw[ultra thick](0,0)--(0,3);
\draw[ultra thick](0,0)--(3,3);
\foreach \x in {0,1,2,3} 
\foreach \y in {0,1,2,3} 
\node at (\x,\y) {$\bullet$};
\node at (.2,0) [anchor=north]{$(0,0)$};
\end{tikzpicture}

\noindent
The corresponding variety $\P_\sigma$ is the blowup of $\C^2$ at the origin.
\end{itemize}

\smallskip
\begin{rem}
For $\sigma=\{{\mathbf 0}\}$, the dual cone is $M_\R$, so the affine toric variety is simply ${\rm Spec}\,\C[M]$, which is isomorphic to
$(\C^*)^{\rk N}$. The latter is known as the {\blue algebraic torus} because it is homotopic to $(S^1)^{\rk N}$. 
Moreover, the group structure on $(\C^*)^{\rk N}$ extends to the action of it on $\P_\Sigma$. This is where the term {\blue toric} comes from.
\end{rem}

\smallskip{\blue \hrule}
Some maps between toric varieties can be be described combinatorially. Suppose that we have a map of lattices
$\mu:N'\to N$ which induces the map $N'_\R\to N_\R$ which we also call $\mu$. Suppose we have  fans $\Sigma'$ and $\Sigma$ in respective lattices. If for every cone $\sigma'\in \Sigma'$ there exists a cone
$\sigma$ in $\Sigma$ such that $\mu(\sigma')\subseteq \sigma$, then there is a natural map $\P_{\Sigma'}\to \P_\Sigma$.
For example, the maps from the blowup of $\C^2$ at the origin to $\C^2$ and to $\C\P^1$ can be described in this way.

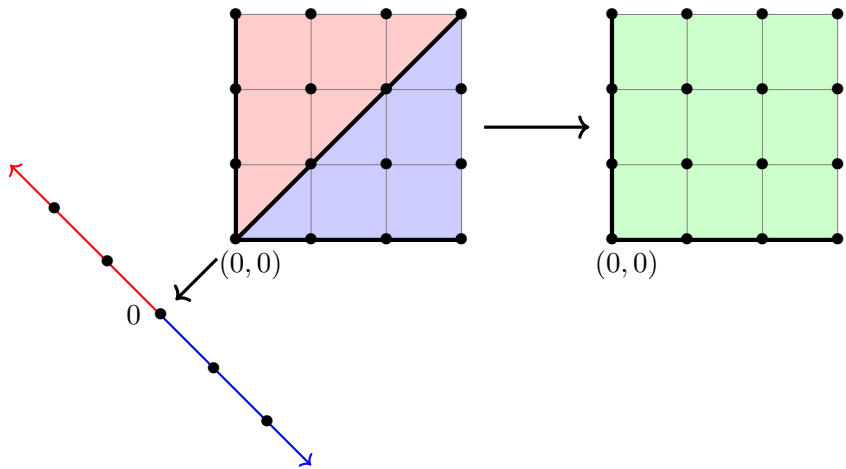
\begin{tikzpicture}
\filldraw[blue!20!white](0,0)--(3,0)--(3,3)--cycle;
\filldraw[red!20!white](0,0)--(0,3)--(3,3)--cycle;
\draw[step=1cm,gray,very thin] (0,0) grid (3,3);
\draw[ultra thick](0,0)--(3,0);
\draw[ultra thick](0,0)--(0,3);
\draw[ultra thick](0,0)--(3,3);
\foreach \x in {0,1,2,3} 
\foreach \y in {0,1,2,3} 
\node at (\x,\y) {$\bullet$};
\node at (.2,0) [anchor=north]{$(0,0)$};

\draw[very thick,->] (3.3,1.5)--(4.7,1.5);

\filldraw[green!20!white](5,0)--(8,0)--(8,3)--(5,3)--cycle;
\draw[step=1cm,gray,very thin] (5,0) grid (8,3);
\draw[ultra thick](5,0)--(8,0);
\draw[ultra thick](5,0)--(5,3);
\foreach \x in {5,6,7,8} 
\foreach \y in {0,1,2,3} 
\node at (\x,\y) {$\bullet$};
\node at (5.2,0) [anchor=north]{$(0,0)$};

\draw[thick,blue, ->] (-1,-1)--(1,-3);
\draw[thick,red, ->] (-1,-1)--(-3,1);

\foreach \y in {-2,-1,0,1,2} 
\node at (\y*1.4142/2-1,-\y*1.4142/2-1) {$\bullet$};
\node at (-1,-1)[anchor=east]{$0~$};

\draw[very thick,->] (-.25,-.25)--(-.8,-.8);

\end{tikzpicture}

\smallskip
Another example of a toric morphism is the following. Consider the fan in $\Z^3$ with seven cones
\begin{equation}\label{fanC30}
\begin{array}{c}
\R_{\geq 0} (1,0,0)+\R_{\geq 0} (0,1,0),~\R_{\geq 0} (0,1,0)+\R_{\geq 0} (0,0,1),~\\
\R_{\geq 0} (0,0,1)+\R_{\geq 0} (1,0,0),
~\R_{\geq 0} (1,0,0),~\R_{\geq 0} (0,1,0),~\R_{\geq 0} (0,0,1),~\{{\mathbf 0}\}.
\end{array}
\end{equation}
The map $\Z^3\to \Z^2$ given by $(a,b,c)\mapsto (a-c,b-c)$ sends these cones exactly into the same cones 
for the  $\C\P^2$ fan, see the very instructive picture below.

\hskip -10pt
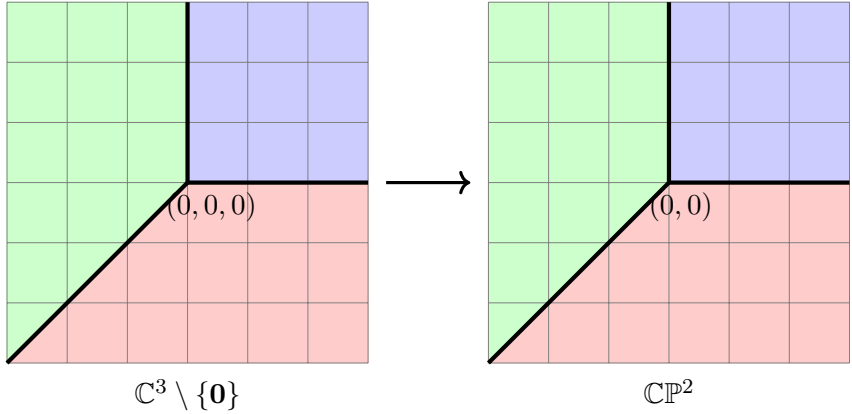
\begin{tikzpicture}[scale = .8]
\filldraw[blue!20!white](0,0)--(3,0)--(3,3)--(0,3)--cycle;
\filldraw[red!20!white](0,0)--(3,0)--(3,-3)--(-3,-3)--cycle;
\filldraw[green!20!white](0,0)--(0,3)--(-3,3)--(-3,-3)--cycle;
\draw[step=1cm,gray,very thin] (-3,-3) grid (3,3);
\draw[ultra thick](0,0)--(3,0);
\draw[ultra thick](0,0)--(0,3);
\draw[ultra thick](0,0)--(-3,-3);
\node at (.4,0) [anchor=north]{$(0,0,0)$};

\draw[very thick,->] (3.3,0)--(4.7,0);

\filldraw[blue!20!white](8,0)--(11,0)--(11,3)--(8,3)--cycle;
\filldraw[red!20!white](8,0)--(11,0)--(11,-3)--(5,-3)--cycle;
\filldraw[green!20!white](8,0)--(8,3)--(5,3)--(5,-3)--cycle;
\draw[step=1cm,gray,very thin] (5,-3) grid (11,3);
\draw[ultra thick](8,0)--(11,0);
\draw[ultra thick](8,0)--(8,3);
\draw[ultra thick](8,0)--(5,-3);
\node at (8.2,0) [anchor=north]{$(0,0)$};

\node at (0,-3.1) [anchor = north] {$\C^3\setminus\{{\bf 0}\}$};
\node at (8,-3.1) [anchor = north] {$\C\P^2$};
\end{tikzpicture}

\noindent
The trick it to view the picture on the left as three-dimensional (the proper faces of the positive octant), while viewing the picture on the right as flat.
The corresponding morphism of toric varieties
is the natural quotient map $\C^3 \setminus \{{\bf 0}\} \to \C\P^2$. This idea of mimicking cones of a fan as some set of faces of the positive octant generalizes to the homogeneous coordinate ring construction of Cox \cite{Cox}.

\smallskip
When the map $\mu:N'\to N$ is an isomorphism, as in the case of the blowdown morphism above, the map
 $\P_{\Sigma'}\to \P_\Sigma$ is a birational morphism. This, in particular, allows one to find resolutions of toric singularities.
Namely, a toric variety $\P_\Sigma$ is smooth if and only if every cone $\sigma\in \Sigma$ is generated by a part of an integral basis of $N$. As a result, a resolution of singularities of $\P_\Sigma$ is given by subdividing all cones of $\Sigma$ into smaller cones that happen to satisfy the above smoothness property.
The best example of this is the resolution of the $A_{n-1}$ singularity. The dual cone to $\R_{\geq 0}^2$ with the lattice
$$
\{(a,b)\in \Z^2,~a=b\,{\rm mod}\,n\}
$$
is the cone $\R_{\geq 0}^2$ with the lattice
$$
N=\Z^2 + \Z(\frac 1n,-\frac 1n).
$$
The corresponding subdivision of the cone is given in the picture below, and corresponds precisely to the 
resolution of $A_{n-1}$ singularities by successive blowups that we considered in Section \ref{sec.ADE1}. The blue rays introduced in the middle of the cone correspond to the exceptional $\C\P^1$-s on the resolution.

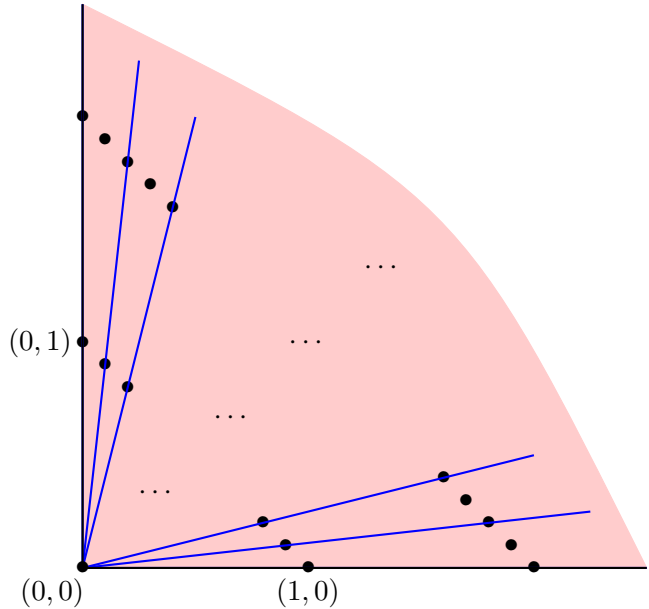
\begin{tikzpicture}
\filldraw[red!20!white] (0,0)--(7.5,0).. controls(5,5)..(0,7.5)--cycle;
\foreach \x in {0,1,2,8,9,10} 
\node[] at (3-3*\x/10,0+3*\x/10) {$\bullet$};
\foreach \x in {0,1,2,3,4,16,17,18,19,20} 
\node[] at (6-3*\x/10,0+3*\x/10) {$\bullet$};
\foreach \x in {0,1,2,8,9,10} \draw[thick,blue] (0,0)--({2.5*(3-3*\x/10)},{2.5*(0+3*\x/10)}); 
\foreach \x in {0,10} \draw[thick] (0,0)--({2.5*(3-3*\x/10)},{2.5*(0+3*\x/10)}); 
\node at (4,4) {$\cdots$};
\node at (3,3) {$\cdots$};
\node at (2,2) {$\cdots$};
\node at (1,1) {$\cdots$};
\node[] at (0,0) {$\bullet$};
\node at (3,0) [anchor=north]{$(1,0)$};
\node at (0,3) [anchor=east]{$(0,1)$};
\node at (-.4,0) [anchor=north]{$(0,0)$};
\end{tikzpicture}

\smallskip
\begin{rem}
Among the ADE singularities, only $A_n$ singularities are toric. The groups in $D_{n\geq 4},E_6,E_7,E_8$ cases are not abelian, so the quotients are more complicated.
\end{rem} 

\smallskip
Toric varieties provide some natural examples of various pathologies. For instance, they can give examples of Weil divisors that are not Cartier. There are also compact toric  varieties that can not be embedded into a projective space. We refer the interested reader to the Fulton's book \cite{Fulton-toric}.

{\bf Exercise 1.} Let $\sigma =\R_{\geq 0}(1,0) + \R_{\geq 0}(1,\sqrt 2)$ be a non-rational polyhedral cone. Prove that the 
semigroup $\sigma \cap \Z^2$ is not finitely generated. Prove that the semigroup ring $\C[\sigma \cap \Z^2]$ is not Noetherian. 

{\bf Exercise 2.} Prove that the blowup of $\P^2$ at two points is isomorphic to the blowup ot $\C\P^1 \times \C\P^1$ at one point using toric geometry. \emph{Hint:} Subdivide the second and third cone in \eqref{fanCP2} to blow up two points on $\C\P^2$.

{\bf Exercise 3*.} Consider the action of $G=\Z/7\Z$ on $\C^2$ with the generator
$$
(x,y) \mapsto (\xi x, \xi^3 y)
$$
where $\xi=\ee^{2\pi\ii/7}$ is the seventh root of $1$. Compute $\C^2/G$ as a toric variety, and find its toric resolution of singularities.

\section{Final comments.}
In this section I  want to list several topics that unfortunately did not make it into these notes but could perhaps be included in an introductory course in algebraic geometry.

\smallskip
\begin{itemize}

\item
Probably the most important topic omitted here is Hodge theory. One can use it to talk about polarizations 
of abelian varieties. Torelli theorem can definitely be stated.

\item
Moduli spaces of curves could be briefly covered without too much trouble. One could even try to hint at the concept of Deligne-Mumford stack, probably in its differential geometry incarnation as an orbifold.

\item 
I have not touched upon the positive characteristics phenomena, but an expert could likely introduce them. 

\item
While I talked about rationality, the related concepts of unirationality and rational connectedness were omitted. One could definitely make a lecture out of that.

\end{itemize}


\begin{thebibliography}{ZZZ}

\bibitem{AKMW} D. Abramovich, K. Karu, K. Matsuki, J. W\l{}odarczyk, 
\textit{Torification and factorization of birational maps,}
J. Amer. Math. Soc.15(2002), no.3, 531-572.
https://arxiv.org/abs/math/9904135

\bibitem{AM} M. Atiyah, I. Macdonald, \emph{Introduction to commutative algebra.}
Addison-Wesley Ser. Math.
Westview Press, Boulder, CO, 2016.

\bibitem{Batyrev}
V. Batyrev, 
\textit{Dual polyhedra and mirror symmetry for Calabi-Yau hypersurfaces in toric varieties,}
J. Algebraic Geom.3(1994), no.3, 493-535.
https://arxiv.org/abs/alg-geom/9310003

\bibitem{BCP} I. Bauer, F. Catanese, R. Pignatelli,
\textit{Surfaces of general type with geometric genus zero: a survey.} Complex and differential geometry, 1-48.
Springer Proc. Math., 8, Springer, Heidelberg, 2011.
https://arxiv.org/abs/1004.2583

\bibitem{BCHM1}
C. Birkar, Caucher, P. Cascini, C. Hacon, J. McKernan, 
\textit{Existence of minimal models for varieties of log general type.}
J. Amer. Math. Soc.23(2010), no.2, 405-468.

\bibitem{BCHM2}
C. Birkar, Caucher, P. Cascini, C. Hacon, J. McKernan, 
\textit{Existence of minimal models for varieties of log general type. II.}
J. Amer. Math. Soc.23(2010), no.2, 469-490.


\bibitem{CdOGP} P.Candelas, X. de la Ossa, P. Green, L. Parkes, 
\textit{A pair of Calabi-Yau manifolds as an exactly soluble superconformal theory,}
Nuclear Phys. B 359 (1991), no. 1, 21-74.


\bibitem{CS} D. Cartwright, T. Steger, \textit{Enumeration of the 50 fake projective
planes,} C. R. Acad. Sci. Paris, Ser. I 348 (2010), 11-13.


\bibitem{Chandrasekharan} K. Chandrasekharan, \textit{Elliptic functions.}
Springer-Verlag, Berlin, 1985.


\bibitem{Cox} D. Cox, \textit{The homogeneous coordinate ring of a toric variety,}
J. Algebraic Geom. 4 (1995), no. 1, 17-50.
https://arxiv.org/pdf/alg-geom/9210008

\bibitem{DF} D. Dummit, R. Foote,
\textit{Abstract algebra.}
John Wiley \& Sons, Inc., Hoboken, NJ, 2004.


\bibitem{Eisenbud} D. Eisenbud, \textit{Commutative algebra.
With a view toward algebraic geometry.}
Grad. Texts in Math., 150,
Springer-Verlag, New York, 1995.

\bibitem{Elkies} N. Elkies, \textit{Three lectures on elliptic surfaces and curves of high rank,}
https://arxiv.org/abs/0709.2908

\bibitem{Fulton} W. Fulton, \textit{Intersection Theory.} Springer-Verlag, Berlin, 1998.

\bibitem{Fulton-toric} W. Fulton, \textit{Introduction to toric varieties.} Princeton University Press, 1993.

\bibitem{Frey1} G. Frey, \textit{Rationale Punkte auf Fermatkurven und getwisteten Modulkurven,} J. reine angew. Math., 331, 185-191.

\bibitem{Frey2} G. Frey, \textit{Links between stable elliptic curves and certain Diophantine equations,} Annales Universitatis Saraviensis. Series Mathematicae, 1 (1): iv+40.

\bibitem{FP} W. Fulton, R. Pandharipande, \textit{
Notes on stable maps and quantum cohomology.} Algebraic geometry -- Santa Cruz 1995, 45-96.
Proc. Sympos. Pure Math., 62, Part 2, American Mathematical Society, Providence, RI, 1997.
https://arxiv.org/abs/alg-geom/9608011

\bibitem{GH} P. Griffith, J. Harris,
\textit{Principles of algebraic geometry.}
Wiley Classics Lib.
John Wiley \& Sons, Inc., New York, 1994.


\bibitem{Hartshorne} R. Hartshorne, \textit{Algebraic Geometry}, Grad. Texts in Math., No. 52
Springer-Verlag, New York-Heidelberg, 1977.

\bibitem{Hellegouarch}
Y. Hellegouarch, Yves, \textit{Points d'ordre $2p^h$ sur les courbes elliptiques,} Polska Akademia Nauk. Instytut Matematyczny. Acta Arithmetica, 26 (3), 253-263. 

\bibitem{Koblitz} N. Koblitz,
\textit{Introduction to elliptic curves and modular forms,} Grad. Texts in Math., 97, 
Springer-Verlag, New York, 1993.

\bibitem{Kontsevich} 
M.~Kontsevich, \textit{Enumeration of rational curves via torus actions.} The moduli space of curves (Texel Island, 1994), 335-368. Progr. Math., 129
Birkh\"auser Boston, Inc., Boston, MA, 1995.
https://arxiv.org/abs/hep-th/9405035

\bibitem{KS} M. Kreuzer, H. Skarke, 
\textit{Complete classification of reflexive polyhedra in four dimensions,}
Adv. Theor. Math. Phys.4(2000), no.6, 1209-1230.

\bibitem{Nagata} M. Nagata, \textit{On the fourteenth problem of Hilbert.} Proc. ICM Edinburgh (1958), 459-462.

\bibitem{PY} G. Prasad, S.-K. Yeung, \textit{Fake projective planes}, Invent. Math. 168
(2007), 321-370; Addendum, 182
(2010), 213-227; Addendum.
Invent. Math. 168(2007), 321-370.

\bibitem{Serre}  J.-P. Serre,  \textit{A course in arithmetic.}
Grad. Texts in Math., No. 7.
Springer-Verlag, New York-Heidelberg, 1973.

\bibitem{Vakil} R. Vakil, \textit{The rising sea. Foundations of algebraic geometry.}
https://math.stanford.edu/\~{}vakil/216blog/FOAGfeb2124public.pdf

\bibitem{VO} A. Onishchik, E. Vinberg, 
\textit{Lie groups and algebraic groups.}
Springer Ser. Soviet Math., Springer-Verlag, Berlin, 1990.

\end{thebibliography}
\end{document}